\documentclass[11pt]{amsart}
\usepackage[mathscr]{eucal}
\usepackage{enumerate, pgf,braids,tikz,amscd,mathtools,leftidx,bm}
\usepackage{cite}
\usepackage[all,cmtip]{xy}
\usepackage{tikz-cd}

\DeclareSymbolFont{extraup}{U}{zavm}{m}{n}
\DeclareMathSymbol{\varheart}{\mathalpha}{extraup}{86}
\DeclareMathSymbol{\vardiamond}{\mathalpha}{extraup}{87}

\usetikzlibrary{decorations.pathreplacing}

\usepackage[utf8]{inputenc}
\usepackage[english]{babel}

\DeclareMathOperator{\rank}{rank}

\DeclareMathOperator{\DD}{D}
\DeclareMathOperator{\HH}{H}
\DeclareMathOperator{\Tw}{Tw}
\DeclareMathOperator{\dg}{dg}

\DeclareMathOperator{\cone}{Cone}
\DeclareMathOperator{\Cone}{Cone}

\DeclareMathOperator{\Perf}{\mathscr{P}\mathrm{erf}}

\newtheorem{proposition}{{Proposition}}[section]

\newtheorem{theorem}[proposition]{{Theorem}}
\newtheorem{conj}[proposition]{{Conjecture}}
\newtheorem{corollary}[proposition]{{Corollary}}
\newtheorem{lemma}[proposition]{{Lemma}}
\newtheorem{defn}[proposition]{Definition}
\newtheorem{remark}[proposition]{{Remark}}
\newtheorem{ex}[proposition]{Example}

\theoremstyle{definition}

\usepackage{setspace}
\usepackage[margin=1in,footskip=0.25in]{geometry}

\begin{document}

\begin{abstract}

We study perverse sheaves of categories their connections to classical algebraic geometry. We show how perverse sheaves of categories encode naturally derived categories of coherent sheaves on $\mathbb{P}^1$ bundles, semiorthogonal decompositions, and relate them to a recent proof of Segal that all autoequivalences of triangulated categories are spherical twists. Furthermore, we show that perverse sheaves of categories can be used to represent certain degenerate Calabi--Yau varieties. 
\end{abstract}

\title[Perverse sheaves of categories and some applications]{Perverse sheaves of categories and some applications}
\author{Andrew Harder}
\address{Department of Mathematics, Lehigh University, Bethlehem, Pennsylvania, USA 18015}
\email[A.~Harder]{anh318@lehigh.edu}
\author{Ludmil Katzarkov}
\address{Department of Mathematics, University of Miami, Coral Gables, Florida 33146, Fakult\"at f\"ur Mathematik, Universit\"at Wien, Oskar-Morgenstern-Platz, 1090 Wien, Austria and National Research University Higher School of Economics, Russian Federation}
\email[L.~Katzarkov]{lkatzarkov@gmail.com}
\maketitle
\tableofcontents

\section{Introduction}

\subsection{Background}

The theory of perverse sheaves arose in the 1980s, in groundbreaking work of Bernstein, Kashiwara, Sato, Schapira, and others in order to deal with sheaves of solutions to D-modules under the Riemann-Hilbert correspondence, and to formalize the theory of intersection cohomology on singular spaces. It was observed early on by Galligo, Granger and Maisonobe \cite{ggm} that if one specifies a stratification on $\mathbb{C}^n$ that is given by a normal crossings union of hyperplanes, then perverse sheaves with respect to this stratification are equivalent to representations of certain quivers. The most famous example of this is the category of perverse sheaves on the disc with respect to the stratification given by a single point. The category of perverse sheaves with respect to this stratification is equivalent to the data $(\phi,\psi,\mathrm{can},\mathrm{var})$ where $\phi$ and $\psi$ are vector spaces, and $\mathrm{can}: \psi \rightarrow \phi$ and $\mathrm{var}: \phi \rightarrow \psi$ are homomorphisms with the property that 
\begin{equation}\label{eq:inverses}
\mathrm{id}_\phi + \mathrm{can}\cdot \mathrm{var}
\end{equation}
is invertible. More generally, according to Gelfand, MacPherson and Vilonen, \cite{gmv}, the category of perverse sheaves on a complex algebraic variety with characteristic variety contained in a fixed conical Lagrangian subvariety of $T^*X$ is equivalent to the category of representations of a quiver. In \cite{mv1}, MacPherson and Vilonen also showed that perverse sheaves on the disc which are perverse with respect to the stratification given by a collection of points $\{p_1,\dots, p_k\}$ are equivalent to the data of a vector space $\phi_i$ for each point $p_i$, a vector space $\psi$ along with maps $\mathrm{can}_i : \psi \rightarrow \phi_i$ and $\mathrm{var}_i: \phi_i \rightarrow \psi$ for each $i$ satisfying Equation \ref{eq:inverses}. The equivalence between this data and the category of perverse sheaves is fixed by the choice of non-intersecting branch cuts for each point $p_i$.

In their seminal work \cite{ks1}, Kapranov and Schechtman observed that one may obtain a notion of perverse sheaves on a disc with values in dg categories by replacing the vector spaces $\phi_i$ with a triangulated dg categories $\mathscr{A}_i$, the category $\psi$ with a triangulated dg category $\mathscr{C}$, and linear maps with functors $F_i : \mathscr{C} \rightarrow \mathscr{A}_i$. The correct analogue of Equation \ref{eq:inverses} and the existence of $\mathrm{can}_i$ is that the functors $F_i$ are \emph{spherical} in the sense of Anno and Logvinenko \cite{al}. The data of a choice of points $\{p_1,\dots , p_k\}$, pretriangulated dg categories $\mathscr{C}$ and $\mathscr{A}_1,\dots, \mathscr{A}_k$, spherical functors $F_i : \mathscr{A}_i \rightarrow \mathscr{C}$ and a certain graph $K$ embedded into the disc is called a ($K$-coordinatized) \emph{ perverse schober}. Recently, Donovan \cite{don} has applied perverse schobers to the study of GIT wall crossings.

%

Despite the fact that perverse schobers are a rather recent invention, the structure that they describe has appeared in several guises over past few decades. Odeskii and Feigin \cite{of}, following Sklyanin \cite{Skly} developed a theory of certain graded rings called Sklyanin algebras which are flat deformations of a polynomial ring in $n$ variables and whose Hilbert series coincide with those of the coordinate ring of $\mathbb{P}^n$. These algebras are often specified by the data of an elliptic curve $E$ along with a pair of line bundles on them of degree $n+1$, and are denoted $A_n(E,\mathscr{L}_1,\mathscr{L}_2)$ in this introduction.

Bondal and Polishchuk studied the abelian category of right graded noetherian modules over $A_2(E,\mathscr{L}_1,\mathscr{L}_2)$ modulo torsion objects, which is denoted $\mathrm{qgr}(A_2(E,\mathscr{L}_1,\mathscr{L}_2))$. They showed that $\DD^b(\mathrm{qgr}(A_2(E,\mathscr{L}_1,\mathscr{L}_2)))$ is equivalent to $\DD^b(\mathrm{rep}(R))$ where
\[
R = \mathrm{End}_E(\mathscr{O}_E \oplus \mathscr{L}_1 \oplus (\mathscr{L}_1 \otimes \mathscr{L}_2)),
\]
and furthermore that one can recover $A_2(E,\mathscr{L}_1, \mathscr{L}_2)$ from $R$.

Attached to the objects $\mathscr{O}_E, \mathscr{L}_1$ and $\mathscr{L}_1 \otimes \mathscr{L}_2$ in $\DD^b(E)$ there are functors
\[
F_{\mathscr{O}_E}  : \DD^b(\mathrm{vect}_k) \rightarrow \DD^b(E), \,  F_{\mathscr{L}_1}  : \DD^b(\mathrm{vect}_k) \rightarrow \DD^b(E), \, F_{\mathscr{L}_1 \otimes \mathscr{L}_2}  : \DD^b(\mathrm{vect}_k) \rightarrow \DD^b(E)
\]
sending $k$ to $\mathscr{O}_E,\mathscr{L}_1$ and $\mathscr{L}_1 \otimes \mathscr{L}_2$ respectively. According to Seidel and Thomas \cite{st}, these functors are spherical. Therefore, the data defining the Sklyanin algebra $A_2(E,\mathscr{L}_2,\mathscr{L}_2)$ precisely defines a $K$-coordinatized perverse schober for an appropriate choice of $K$. There is a natural notion of the category of global sections of a perverse schober (Section \ref{sect:pschob}). The category of global sections of the perverse schober associated to $(E,\mathscr{L}_1,\mathscr{L}_2)$ is exactly the category $\DD^b(\mathrm{qgr}(A_2(E,\mathscr{L}_1,\mathscr{L}_2)))$. Here $\mathrm{qgr}(A)$ denotes the category of finitely generated $A$ 

Similar perverse schobers appear in the work of Seidel on Fukaya categories associated to Lefschetz fibrations \cite{seidmut1,seidmut2}. If $\pi : E \rightarrow \mathbb{C}$ is a symplectic fibration with only Morse type singularities and $E$ is an exact symplectic manifold, then Seidel defines a category, often called the directed Fukaya category or the Fukaya-Seidel category, whose objects correspond to Lagrangian vanishing cycles of $\pi$. Let $\Sigma$ be the set of critical values of $\pi$, and let $\gamma_i$ be a collection of counterclockwise oriented paths connecting a base point $s$ to each point $p_i$ in $\Sigma$ to $s$. We then obtain a Lagrangian vanishing thimble for each $\gamma_i$ which defines an ordered collection $a_1,\dots, a_k$ of Lagrangian spheres in the Fukaya category of the fiber over $s$. Each of these Lagrangian vanishing spheres defines a spherical functor from $\DD^b(\mathrm{vect}_k)$ to the derived Fukaya category of the fiber over $s$. This canonically provides the data of a perverse schober whose category of global sections is called the directed Fukaya category associated to $(E,\pi)$. The analogy between Seidel's construction and the work of Bondal and Kapranov lies at the root of Auroux, Katzarkov and Orlov's proof of homological mirror symmetry for noncommutative del Pezzo surfaces \cite{ako1}.

There is a deeper relationship between perverse sheaves of categories and Fukaya categories, originating in unpublished work of Bondal \cite{oberwolfach}\footnote{This is also sometimes attributed to Wang.}. It is known that the exact Fukaya category of a Riemann surface with several punctures is the category of global sections of a sheaf of categories on $S$ \cite{stz,kap-dyck,sib-et,hkk}. Furthermore, according to work of Nadler and Zaslow \cite{nz} and Nadler \cite{nad1,nad2,nad3}, culminating in ongoing work of Ganatra, Pardon and Shende \cite{gps}, the Fukaya category of a Weinstein manifold $M$ can be recovered as the category of global sections of a perverse sheaf of categories on a singular Lagrangian skeleton of $M$. Similar ideas appear in work of Tamarkin \cite{tam} and Tsygan \cite{tsy}.

\subsection{Outline}

The purpose of this paper is to forge connections between classical geometry, category theory and the newly minted theory of perverse sheaves of categories. In the process, several new and interesting questions are raised. 

The first two sections are devoted to developing in detail the theory of perverse sheaves of categories and their categories of global sections.

In Section \ref{sect:bg} we outline some basic facts about dg categories, including a number of ways to glue them together; we describe explicit models of homotopy fiber product and equalizers along with several constructions of categories admitting semiorthogonal decompositions coming from work of Tabuada \cite{tabthesis}, Orlov \cite{orl-glue} and Kuznetsov and Lunts \cite{kl}.

In Section \ref{sect:psc} we will look at perverse sheaves of categories in the language of Kapranov and Schechtman \cite{ks1}. We give a definition of the category of global sections of a perverse schober as the homotopy fiber product over a certain diagram of categories. We show that it agrees with the gluing construction of Kuznetsov and Lunts \cite{kl}. We give a general concept of $K$-coordinatized perverse sheaves of categories, whose definition was hinted at by Kapranov and Schechtman \cite{ks2} and by Kontsevich \cite{kontsevich}. We define their monodromy and categories of global sections in the process.

Sections \ref{sect:recon}, \ref{sect:sod} and \ref{sect:typeII} are dedicated to describing several situations in algebraic geometry in which perverse sheaves of categories appear naturally.

In Section \ref{sect:recon} we study two interesting types of perverse sheaves of categories on $S^1 \times [0,1]$, both of which are related to examples which appear in the prescient work of Kontsevich \cite{kontsevich}. These sheaves of categories are both built from a pair of a category $\mathscr{C}$ and a monodromy autoequivalence $\Phi$ of $\mathscr{C}$. If we let $\mathscr{C} = \Perf_k$ and $\Phi = \mathrm{id}$, then the global sections of the corresponding sheaves of categories are either $\DD^b(\mathbb{P}^1)$ or $\Perf(\mathbb{A}^1)$. If we let $\mathscr{C}$ be $\DD^b_{\dg}(X)$ for a smooth variety $X$, and we let monodromy be tensor product with a line bundle $\mathscr{L}$ on $X$, then the global sections of these sheaves of categories recover the derived category of coherent sheaves on $\mathbb{P}_X(\mathscr{O}_X\oplus \mathscr{L})$ (Theorem \ref{thm:orlov}) or the total space of $\mathscr{L}$. When monodromy is more general, the category of global sections should be regarded as the total space of a {\it noncommutative} $\mathbb{P}^1$ bundle or line bundle over $\mathscr{C}$. 

Recently Segal \cite{seg} proved that all autoequivalences of triangulated categories can be thought of as spherical twists by constructing a category which can be regarded as the total space of a noncommutative line bundle. We show that our construction can be used to partially recover Segal's results.

From the constructions of Section \ref{sect:psc} it follows that the category of global sections of a perverse schober admits a semiorthogonal decomposition. In Section \ref{sect:sod} we note that the converse often holds; dg or triangulated categories which admit semiorthogonal decompositions appear as the global sections of a perverse schober whenever this category admits a spherical functor whose cotwist is the Serre functor up to twist.
\begin{theorem}[Theorem \ref{thm:recon}]
Let $\mathscr{T}$ and $\mathscr{C}$ be dg enhanced triangulated categories and let $F: \mathscr{T} \rightarrow \mathscr{C}$ be a spherical functor whose cotwist automorphism is quasiequivalent to the Serre functor on $\mathscr{T}$ up to shift. If $\mathscr{T} = \langle \mathscr{A}_1,\dots, \mathscr{A}_k \rangle$ is a semiorthogonal decomposition of $\mathscr{T}$ and if $\alpha_j : \mathscr{A}_j \rightarrow \mathscr{C}$ are the natural embeddings, then there is a perverse schober made up of the data of $\mathscr{A}_i$ and functors $F_j = F\alpha_j$ whose category of global sections is $\mathscr{T}$.
\end{theorem}

For instance, if $X$ is a Fano variety whose bounded derived category of coherent sheaves admits a semiorthogonal decomposition and $Z$ a smooth anticanonical divisor, then there is a perverse schober whose generic fiber is $\DD^b(Z)$ and which encodes the semiorthogonal decomposition on $\DD^b(X)$. Conjecturally, the same relation should hold between the Fukaya-Seidel category of a pair $(Y,w)$ and the Fukaya category of the smooth fiber of $w$.

In Section \ref{sect:typeII}, we show that mild degenerations of Calabi--Yau varieties, called Tyurin degenerations, or more generally simplified type II degenerations (Definition \ref{def:simptII}) can be represented in terms of perverse sheaves of categories over $S^2$ with a number of boundary components. We show that Friedman's d-semistability condition \cite{fried} is intimately related to the structure of the monodromy of these perverse sheaves of categories. Perverse sheaves of categories are used to show that type I modifications of type II degenerations of K3 surfaces leave the category of perfect complexes on the degenerate fiber of a type II degeneration of K3 surfaces invariant (Theorem \ref{thm:typeImod}).

\subsection{Acknowledgements}
The authors would like to thank D. Auroux, C. Doran, M. Kapranov, M. Kontsevich, T. Logvinenko and A. Thompson for their useful comments regarding this work. We would also like to thank P. Belmans for pointing out many typos in an earlier draft and for discussion regarding noncommutative geometric aspects of this paper. The authors would like to express their gratitude to the anonymous referee for providing invaluable feedback and corrections.

The authors would like to thank the IH{\'E}S and the Erwin Schr\"odinger Institute for providing excellent working conditions under which portions of this work were completed. The authors were both supported by the Simons Collaboration in Homological Mirror Symmetry, Simons Collaboration grant 385575 and the second author was supported by Laboratory of Mirror Symmetry NRUHSE, RF government grant, ag. No. 14.641.31.0001 during the preparation of this work.

\section{Background on dg categories}\label{sect:bg}

Here we will provide background on dg categories, their model structures and gluing constructions. We will assume that we work over a field $k$ of characteristic 0, which one is free to assume is algebraically closed, or even $\mathbb{C}$. For us, a dg category will be a (small) category $\mathscr{C}$ whose homomorphisms have the following structure.
\begin{enumerate}
\item $\hom_\mathscr{C}(a,b)$ is $\mathbb{Z}$-graded $k$-vector space, the graded components being denoted $\hom_\mathscr{C}^i(a,b)$.
\item There is a $k$-linear differential $d:\hom_\mathscr{C}^i(a,b) \rightarrow \hom_\mathscr{C}^{i+1}(a,b)$ with $d^2 = 0$.
\item If $f \in \hom_\mathscr{C}^{i}(a,b)$ and $g \in \hom_\mathscr{C}^j(c,a)$ then $f\cdot g$ contained in $\hom_\mathscr{C}^{i+j}(c,b)$.
\item Each $\hom_\mathscr{C}^0(a,a)$ contains a closed identity element $\mathrm{id}_a$ with the obvious properties.
\item If $f \in \hom_\mathscr{C}^i(b,c)$ and $g \in \hom_\mathscr{C}^j(a,b)$ then $d(f\cdot g) = df\cdot g + (-1)^{j} f \cdot dg$.
\end{enumerate}
\begin{ex}
The category of chain complexes over a field $k$ forms a dg category which we will denote $\mathrm{Ch}_k$. Let $a^\bullet$ and $b^\bullet$ be $k$-chain complexes, and let $\hom_{\mathrm{Ch}_k}^{\ell}(a^\bullet,b^\bullet) = \prod_{i\in\mathbb{Z}}\hom_k(a^{i}, b^{\ell+i})$ equipped with the differential which acts as
$$
f \in \hom_{\mathrm{Ch}_k}(a^\bullet,b^\bullet) \mapsto d\cdot f +(-1)^{\deg(f)}f\cdot d
$$
\end{ex}

A dg functor between dg categories $\mathscr{C}$ and $\mathscr{D}$ is a functor $F$ between the underlying $k$-linear categories of $\mathscr{C}$ and $\mathscr{D}$ so that the gradings on $\hom_\mathscr{C}(a,b)$ and $\hom_\mathscr{D}(F(a),F(b))$ agree, and so that $F(df) = dF(f)$. The collection of small dg categories over $k$ forms a category which we will denote $\mathrm{dgcat}_k$.

To each dg category $\mathscr{C}$ there is an associated homotopy category denoted $\HH^0\mathscr{C}$ or sometimes $[\mathscr{C}]$ so that $\mathrm{Ob}(\HH^0\mathscr{C}) = \mathrm{Ob}(\mathscr{C})$. We denote the object in $\HH^0\mathscr{C}$ associated to $a\in \mathrm{Ob}(\mathscr{C})$ by $[a]$. Homomorphisms in $\HH^0\mathscr{C}$ are given by $\hom_{\HH^0\mathscr{C}}([a],[b]) = \HH^0\hom_\mathscr{C}(a,b)$. Two objects $a$ and $b$ in $\mathscr{C}$ are called quasiisomorphic if the objects $[a]$ and $[b]$ are isomorphic in $\HH^0\mathscr{C}$.

The set of dg functors from $\mathscr{C}$ to $\mathscr{D}$, denoted $\mathrm{Fun}_{\mathrm{dg}}(\mathscr{C},\mathscr{D})$ forms a dg category whose homomorphisms are natural transformations. A (right) dg $\mathscr{C}$ module is a dg functor from $\mathscr{C}^\mathrm{op}$ to $\mathrm{Ch}_k$. The category of $\mathscr{C}$-modules will be denoted $\mathrm{mod}_\mathscr{C}$. There is a triangulated structure on the category $\mathrm{mod}_\mathscr{C}$ obtained from the triangulated structure on $\mathrm{Ch}_k$. Specifically, if $\mathsf{M}$ is a $\mathscr{C}$-module then we may define $\mathsf{M}[1]$ to be the functor so that
$$
\mathsf{M}[1](a) = \mathsf{M}(a)[1]
$$
and similarly, if we have a homomorphism of degree 0 between a pair of $\mathscr{C}$-modules $f: \mathsf{M} \rightarrow \mathsf{N}$, then we may define the cone of $f$ to be the module
$$
\cone(f)(a) = \cone(\mathsf{M}(a) \xrightarrow{f(a)} \mathsf{N}(a))
$$
There is a canonical functor from $\mathscr{C}$ to $\mathrm{mod}_\mathscr{C}$ called the Yoneda embedding,
$$
\mathsf{Y} : \mathscr{C} \longrightarrow \mathrm{mod}_\mathscr{C}, \qquad a \mapsto \hom_\mathscr{C}(-,a)
$$
which is a full and faithful embedding of dg categories by a dg version of the Yoneda lemma. The $\mathscr{C}$-module associated to an element $b \in \mathscr{C}$ under $\mathsf{Y}$ is denoted $\mathsf{Y}^b$.  A $\mathscr{C}$-module is called representable if there is some $b$ so that $\mathsf{M}\cong \mathsf{Y}^b$. If there is some $b$ so that $\mathsf{M}$ and $\mathsf{Y}^b$ are quasiisomorphic then we will say that $\mathsf{M}$ is quasirepresentable.

A functor $F$ between a pair of dg categories $\mathscr{C}$ and $\mathscr{D}$ is a {\it quasiequivalence} if for each pair of objects $a,b$ in $\mathscr{C}$, the map $\hom_\mathscr{C}(a,b) \rightarrow \hom_\mathscr{D}(F(a),F(b))$ is a quasiisomorphism of complexes and every object in $\mathscr{D}$ is quasiisomorphic to an object in the image of $F$. In other words, $F$ induces an equivalence between $\HH^0\mathscr{C}$ and $\HH^0\mathscr{D}$. We say that two categories $\mathscr{C}$ and $\mathscr{D}$ are quasiequivalent if there is a chain of quasiequivalences
between them. If $\mathscr{A}$ and $\mathscr{B}$ are a pair of dg categories, an $\mathscr{A}$-$\mathscr{B}$-bimodule is a right $\mathscr{A}^\mathrm{op}\otimes \mathscr{B}$ dg module. For instance, if $F_1:\mathscr{A} \rightarrow \mathscr{C}$ and $F_2: \mathscr{B} \rightarrow \mathscr{C}$ are dg functors, then there is a dg bimodule,
\[
\mathsf{S}_{F_1,F_2}(a,b) = \hom_\mathscr{C}(F_2(b),F_1(a)).
\]
A $\mathscr{A}$-$\mathscr{B}$ bimodule determines a functor from $\mathscr{A}^\mathrm{op}$ to $\mathrm{mod}_{\mathscr{B}^\mathrm{op}}$ and from $\mathscr{B}$ to $\mathrm{mod}_{\mathscr{A}}$. A quasifunctor is a $\mathscr{A}$-$\mathscr{B}$ bimodule $\mathsf{S}$ so that $\mathsf{S}(a,-)$ is quasirepresentable for any $a$ in $\mathscr{A}$. Such bimodules induce functors from $\HH^0\mathscr{A}$ to $\HH^0\mathscr{B}$. If we have a functor $F$ from $\mathscr{B}$ to $\mathscr{C}$ and $\mathscr{A}$ is quasiequivalent to $\mathscr{B}$ then there is a quasifunctor from $\mathscr{A}$ to $\mathscr{B}$ corresponding to $F$ obtained by taking tensor products of bimodules \cite{al}.

There is a model structure on $\mathrm{dgcat}_k$ called the Dwyer-Kan (DK) model structure. This model structure has weak equivalences given by quasiequivalences. Fibrations given by functors for which $F: \hom_\mathscr{C}(a,b) \rightarrow \hom_\mathscr{D}(F(a),F(b))$ is a surjection of complexes and so that if $F(a)$ is quasiisomorphic to $c$ in $\mathscr{D}$ then there is an object $c' \in \mathscr{C}$ so that $F(c') = c$ and this quasiisomorphism lifts to a quasiisomorphism between $a$ and $c'$. A consequence of this definition is that all dg categories are fibrant with respect to the DK model structure on $\mathrm{dgcat}_k$ (see \cite[Remark 2.14]{tab-drinfeld}).

\subsection{Pretriangulated dg categories}

Bondal and Kapranov have introduced pretriangulated dg categories in order to amend some of the deficiencies of triangulated categories.
\begin{defn}
A dg category $\mathscr{C}$ is \emph{pretriangulated} if for every $\mathsf{Y}^a \in \mathrm{mod}_\mathscr{C}$, the module $\mathsf{Y}^a[1]$ is quasirepresentable, and if $f : a \rightarrow b$ is closed of degree 0, then
\[
\cone( \mathsf{Y}^a \xrightarrow{\mathsf{Y}^f} \mathsf{Y}^b)
\]
is quasirepresentable. In other words, the homotopy category of $\mathscr{C}$ is triangulated and this triangulated structure is consistent with that of $\mathrm{mod}_\mathscr{C}$.
\end{defn}
Every dg category can be minimally embedded into a pretriangulated dg category. This construction is described by Bondal and Kapranov.
\begin{defn}
Given a dg category $\mathscr{A}$, \emph{ the category of twisted complexes over $\mathscr{A}$}, denoted $\mathrm{Tw}\mathscr{A}$, is the category whose objects are given pairs of $(\bigoplus_{i=1}^n a_{i}[d_i],m)$ where $\bigoplus_{i=1}^na_i[d_i]$ is a formal direct sum of formally shifted objects in $\mathscr{A}$ and $d_i$ are integers. The second piece of data $m$ is a strictly upper triangular matrix so that $m_{i,j} \in \hom_{\mathscr{A}}(a_i,a_j)[d_j-d_i]$, is homogeneous of degree 1 and vanishes if $i \geq j$. These objects must satisfy the Maurer-Cartan relation, $dm + m^2 = 0$. For a pair of objects $\alpha = (\bigoplus_{i=1}^k a_i[d_i],m)$ and $\beta = (\bigoplus_{i=1}^\ell b_i[e_i], n)$ we define $\hom_{\mathrm{Tw} \mathscr{A}}(\alpha,\beta)$ to be the space of matrices $g$ with entries $g_{i,j} \in \hom_{\mathscr{A}}(a_j,b_i)[e_i - d_j]$ and composition given by matrix multiplication. If $f$ is an element of $\hom_{\mathrm{Tw}(\mathscr{A})}^j(\alpha,\beta)$, then we define the differential acting on $f$ as follows,
$$
d_{\mathrm{Tw}(\mathscr{B})}f = d_\mathscr{A} f  + mf + (-1)^{j} fn
$$
where $d_\mathscr{A} f$ means we act element-wise on $f$ by the differential in $\mathscr{A}$. There is a full and faithful functor from $\mathscr{A}$ to $\mathrm{Tw}\mathscr{A}$ which sends an object $a \in \mathscr{A}$ to the object $(a,0)$ of $\mathrm{Tw}\mathscr{A}$. Pretriangulated categories are precisely those for which the embedding of $\mathscr{A}$ into $\mathrm{Tw}\mathscr{A}$ is a quasi-equivalence. One can give explicit representatives of cones in $\Tw\mathscr{A}$.

\end{defn}
If $F:\mathscr{A} \rightarrow \mathscr{B}$ is a dg functor, then there is an induced dg functor from $\Tw \mathscr{A}$ to $\Tw \mathscr{B}$ which we will denote by $F^{\Tw}$.

\begin{remark}
One may also construct the pretriangulated envelope of a dg category $\mathscr{C}$ by taking the category of all semi free dg modules over $\mathscr{C}$. Details may be found in \cite{al} or\cite{orl-glue}.

\end{remark}
\begin{ex}
If we let $\mathbf{k}$ be the category with one object $e$ so that $\hom_\mathbf{k}(e,e) = k \cdot \mathrm{id}_e$, then $\Tw \mathbf{k}$ is equivalent to the subcategory of $\mathrm{Ch}_k$ whose objects are bounded complexes of $k$-vector spaces. We will denote this category $\Perf_k$.
\end{ex}
A category $\mathscr{A}$ is said to be \emph{strongly pretriangulated} if it is dg equivalent to $\mathrm{Tw}\mathscr{A}$.

\subsection{Categories of coherent sheaves}\label{sect:cohsheaves}

Here we make several remarks on dg extensions of triangulated categories of coherent sheaves. We will denote by $\mathrm{coh}(X)$ the abelian category of coherent sheaves on $X$, and by $\mathrm{qcoh}(X)$ the category of quasicoherent sheaves on $X$. To each of these, we have triangulated derived categories $\DD(\mathrm{coh}(X))$ and $\DD(\mathrm{qcoh}(X))$. Our goal is to describe dg extensions of such categories. Good overviews can be found in \cite[Section 3.1]{orl-glue} and \cite[Section 3]{kl}. Our notation is adapted from \cite{orl-glue}.

There is a dg category $\mathscr{K}(\mathrm{qcoh}(X))$ of unbounded complexes of quasicoherent sheaves on $X$. A complex $\mathscr{S}^\bullet$ is called {\it h-flat} (or homotopy-flat) if the total complex of $\mathscr{S}^\bullet \otimes_{\mathscr{O}_X} \mathscr{C}^\bullet$ is acyclic for any acyclic complex $\mathscr{C}^\bullet$. We let $\mathscr{F}\mathrm{lat}(X)$ be the full subcategory of $\mathscr{K}(X)$ made up of h-flat complexes, and we let $\mathscr{A}\mathrm{c}_f(X)$ be the full subcategory of acyclic h-flat complexes, then the quotient dg category $\mathscr{F}\mathrm{lat}(X)/\mathscr{A}\mathrm{c}_f(X)$ is a dg enhancement of $\DD(\mathrm{qcoh}(X))$, which we will call $\DD_\mathrm{dg}(\mathrm{qcoh}(X))$.

If $X$ is noetherian and separated then $\DD^b(\mathrm{coh}(X))$ is equivalent to the subcategory of $\DD^b(\mathrm{qcoh}(X))$ made up of of complexes with coherent cohomology, which we will call $\DD^b_\mathrm{coh}(\mathrm{qcoh}(X))$. Thus we can define $\DD^b_\mathrm{dg}(\mathrm{coh}(X))$ to be the full subcategory of $\DD_\mathrm{dg}(\mathrm{qcoh}(X))$ whose class in $\DD(\mathrm{qcoh}(X))$ has bounded coherent cohomology. For brevity, we will use the notation $\DD^b_\mathrm{dg}(X)$ for $\DD^b_\mathrm{dg}(\mathrm{coh}(X))$.

There are many different dg enhancements of the category of coherent sheaves, but for quasiprojective schemes, all enhancements are quasiequivalent \cite[Theorem 8.13]{lo}.

The category $\mathrm{Perf}(X)$ is a full triangulated subcategory of $\DD^b(X)$, so we may define $\Perf(X)$ to be the full subcategory of objects in $\DD^b_\mathrm{dg}(X)$ whose corresponding objects are in $\mathrm{Perf}(X)$. Given a morphism $f: X \rightarrow Y$ of schemes, there is a functor $\mathbb{L}f^*: \DD^b_\mathrm{dg}(Y) \rightarrow \DD^b_\mathrm{dg}(X)$. There is a right adjoint (quasi) functor $\mathbb{R}f_* : \DD^b_\mathrm{dg}(X) \rightarrow \DD^b_\mathrm{dg}(Y)$. The functor $\mathbb{L}f^*$ extends to a dg functor from $\Perf(Y)$ to $\Perf(X)$. Tensoring with any $\mathscr{F}$ in $\mathscr{K}(\mathrm{qcoh}(X))$ also defines a dg functor from $\DD^b_{\dg}(X)$ to $\DD^b_{\dg}(X)$ denoted $(-)\otimes^\mathbb{L}\mathscr{F}$ which extends to a functor from $\Perf(X)$ to $\Perf(X)$. If $f$ is proper and has finite Tor dimension, then $\mathbb{R}f_*$ extends to a quasifunctor from $\Perf(X)$ to $\Perf(Y)$.

We will use the notation $\mathbb{R},\mathbb{L}$ to denote derived quasifunctors between dg enhancements and $\mathbf{R},\mathbf{L}$ derived functors between their triangulated homotopy categories.

\subsection{Various gluing constructions}

There are several related categorical constructions that we would like to discuss in this section. The first construction is due to Tabuada \cite{tabthesis} and is expanded upon by Orlov \cite{orl-glue}. We construct this category from the data of two dg categories $\mathscr{A}_1,\mathscr{A}_2$ and a $\mathscr{A}_1$-$\mathscr{A}_2$ dg bimodule, which we will denote $\phi$. We will denote the corresponding category as $\mathscr{A}_1 \sqcup_\phi \mathscr{A}_2$. We have
$$
\mathrm{Ob}(\mathscr{A}_1 \sqcup_\phi \mathscr{A}_2)  = \mathrm{Ob}(\mathscr{A}_1) \coprod \mathrm{Ob}(\mathscr{A}_2)
$$
and
$$
\hom_{\mathscr{A}_1 \sqcup_\phi \mathscr{A}_2}(a,b) = \begin{cases}\hom_{\mathscr{A}_1}(a,b) & \mbox{ if } a,b \in \mathrm{Ob}(\mathscr{A}_1)\\
 \hom_{\mathscr{A}_2}(a,b) & \mbox{ if } a,b \in \mathrm{Ob}(\mathscr{A}_2) \\
\phi(b,a) & \mbox{ if } a \in \mathrm{Ob}(\mathscr{A}_1), b \in \mathrm{Ob}(\mathscr{A}_2) \\
0 & \mbox{ otherwise}
 \end{cases}
$$
Composition is given by using the bimodule structure on $\phi$. Usually, we will be interested just in the bimodule defined as $(b,a) \mapsto \hom_{\mathscr{C}}(F_1(a),F_2(b))$ for functors $F_i : \mathscr{A}_i \rightarrow \mathscr{C}$. In this case we will write $\mathscr{A}_1 \sqcup_{F_1,F_2} \mathscr{A}_2$.

We will consider the pretriangulated envelope of $\mathscr{A}_1 \sqcup_\phi \mathscr{A}_2$. Kuznetsov and Lunts \cite{kl} give a concrete description of this category. Let $\mathscr{A}_1 \times_\phi \mathscr{A}_2$ be the dg category whose objects are $(a_1,a_2,\mu)$ where $a_i \in \mathscr{A}_i$ and $\mu \in \mathrm{Z}^0\phi(a_1,a_2)$ is a closed degree 0 element. Details of their construction can be found in \cite[Section 4]{kl}.

In the special case where there are two dg functors $F:\mathscr{A}_1 \rightarrow \mathscr{C}$ and $G: \mathscr{A}_2 \rightarrow \mathscr{C}$ and $\phi = \mathsf{S}_{F,G}$, we will give a precise definition.

\begin{defn}\label{def:bbb}
If $\mathscr{A}_1,\mathscr{A}_2$ and $\mathscr{C}$ are dg categories with $F:\mathscr{A}_1 \rightarrow \mathscr{C}$ and $G : \mathscr{A}_2 \rightarrow \mathscr{C}$ dg functors, then $\mathscr{A}_1 \times_{F,G} \mathscr{A}_2$ is the category with objects given by triples $(a,b,\mu)$ where $a \in \mathrm{Ob}(\mathscr{A}_1), b \in \mathrm{Ob}(\mathscr{A}_2)$ and $\mu$ is a closed degree 0 element of $\hom_\mathscr{C}(F(a),G(b))$. Homomorphisms between $(a,b,\mu)$ and $(c,d,\chi)$ are given by the mapping cone in $\mathrm{Ch}_k$ of the map
\[
\hom_{\mathscr{A}_1}(a,c) \oplus \hom_{\mathscr{A}_2}(b,d) \xrightarrow{\xi_* - \mu^*} \hom_{\mathscr{C}}(F_1(a),F_2(b)).
\]
Here $\mu^*$ means we compose the map $\hom_{\mathscr{A}_1}(b,d) \rightarrow \hom_\mathscr{C}(F_1(b),F_2(d))$ with $F_1(\mu)^*$, and we define $\xi_*$ similarly. Precisely, $\hom^i_{\mathscr{A}_1\times_{F_1,F_2}\mathscr{A}_2}((a,b,\mu),(c,d,\xi))$ is made up of triples
\[
(f,g,h) \in \hom^i_{\mathscr{A}_1}(a,b) \oplus \hom^i_{\mathscr{A}_2}(c,d) \oplus \hom^{i-1}_\mathscr{C}(F_1(a),F_2(d))
\]
with differential,
\[
d(f,g,h) = (df,dg,dh + (F_2(g)\cdot \mu - \xi\cdot F_1(f)))
\]
and
\[
(f',g',h')\cdot(f,g,h) = (f'\cdot f, g'\cdot g, h'\cdot f + (-1)^{\deg g'}g'\cdot h).
\]

\end{defn}
There is always a functor which we denote $c_{F_1,F_2} : \mathscr{A}_1 \times_{F_1,F_2} \mathscr{A}_2 \rightarrow \Tw \mathscr{C}$ sending $(a,b,\mu)$ to $\cone(\mu)$. We now list several properties of this construction.

\begin{proposition}[{Kuznetsov-Lunts \cite[Proposition 4.3]{kl}}]\label{ref:klprop1}
If $\mathscr{A}_1$ and $\mathscr{A}_2$ are pretriangulated, then $\mathscr{A}_1 \times_\phi \mathscr{A}_2$ is also pretriangulated.

\end{proposition}
\begin{proposition}[{Kuznetsov-Lunts \cite[Proposition 4.9]{kl}}]\label{ref:klprop2}
If there are quasiisomorphisms of functors $n_1: F_1' \rightarrow F_1$ and $n_2:F_2' \rightarrow F_2$ then there is a quasiequivalence,
\[
\gamma_{n_1,n_2} : \mathscr{A}_1 \times_{F_1',F_2'} \mathscr{A}_2 \rightarrow \mathscr{A}_1 \times_{F_1,F_2} \mathscr{A}_2
\]
so that $c_{F_1',F_2'}$ is quasiequivalent to $c_{F_1,F_2}\cdot \gamma_{n_1,n_2}$ as a bimodule.
\end{proposition}
In Proposition \ref{ref:klprop2}, we say that two functors are quasiisomorphic if the corresponding bimodules are, or equivalently, if there is a natural transformation between them which induces isomorphisms in the homotopy category.
\begin{proposition}[{Kuznetsov-Lunts \cite[Proposition 4.14]{kl}}]\label{ref:klprop3}
If $G: \mathscr{A} \rightarrow \mathscr{A}_1$ and $H:\mathscr{B} \rightarrow \mathscr{A}_2$ are quasiequivalences, then there is a quasiequivalence
\[
\eta_{G,H}:\mathscr{A} \times_{G\cdot F_1,H \cdot F_2} \mathscr{B} \rightarrow \mathscr{A}_1 \times_{F_1,F_2} \mathscr{A}_2.
\]
Furthermore, $c_{G\cdot F_1,H \cdot F_2}$ and $c_{F_1,F_2} \cdot \eta_{G,H}$ are quasiequivalent as bimodules.
\end{proposition}
The second statements in Propositions \ref{ref:klprop2} and \ref{ref:klprop3} are deduced from the proofs of \cite[Proposition 4.9, Proposition 4.14]{kl}
\begin{remark}
According to \cite[Corollary 4.5]{kl}, if $\mathscr{A}_1$ and $\mathscr{A}_2$ are pretriangulated, then $\HH^0(\mathscr{A}_1 \times_\phi \mathscr{A}_2)$  admits a semiorthogonal decomposition $\langle \HH^0 \mathscr{A}_1, \HH^0 \mathscr{A}_2 \rangle$. See Section \ref{sect:sod} for details on semiorthogonal decompositions.
\end{remark}
\begin{remark}
If $\mathscr{A}_1$ and $\mathscr{A}_2$ are pretriangulated dg categories, then $\mathscr{A}_1 \times_{\phi} \mathscr{A}_2$ is quasiequivalent to $\Tw (\mathscr{A}_1 \sqcup_\phi \mathscr{A}_2)$.
\end{remark}

There are two closely related constructions coming from Tabuada's construction \cite{tabthesis} of the path object in $\mathrm{dgcat}_k$. Our description of the homotopy equalizer can be found in \cite{sibilla} and our construction of the homotopy fiber product can be found in \cite{bbb}.

\begin{defn}
If we have two functors $F_1,F_2:\mathscr{A} \rightarrow \mathscr{C}$, the \emph{homotopy equalizer} $\mathscr{E}\mathrm{q}^h(F_1,F_2)$ is the category whose objects are pairs $(a,\mu)$ where $a \in \mathrm{Ob}(\mathscr{A})$ and $\mu \in \mathrm{Z}^0\hom_\mathscr{C}(F_1(a),F_2(a))$ so that $[\mu]$ is an isomorphism. The complexes $\hom^i_{\mathscr{E}\mathrm{q}^h(F_1,F_2)}((a,\mu), (b,\xi))$ are given by pairs
\[
(f,h) \in \hom^i_\mathscr{A}(a,b) \oplus \hom_\mathscr{C}^{i-1}(F_1(a),F_2(b))
\]
equipped with the differential $d(f,h) = (df, dh + (\xi\cdot F_1(f) - F_2(f) \cdot \mu))$, and so that composition computed as
\[
(f,h)\cdot (f',h') = (f\cdot f', F_2(f')\cdot h + (-1)^{\deg f'} h'\cdot F_1(f)).
\]
\end{defn}
\begin{defn}
The \emph{homotopy fiber product} of a pair of functors $F_1 :\mathscr{A}_1\rightarrow \mathscr{C}$ and $F_2 : \mathscr{A}_2 \rightarrow \mathscr{C}$, denoted $\mathscr{A}_1\times^h_\mathscr{C}\mathscr{A}_2$ is the homotopy equalizer of the functors $F_1 \times 0 :\mathscr{A}_1 \times \mathscr{A}_2\rightarrow \mathscr{C}$ and $0 \times F_2 : \mathscr{A}_1 \times \mathscr{A}_2 \rightarrow \mathscr{C}$.
\end{defn}

\begin{remark}\label{rmk:ordinaryfiber}
We will also use the ordinary fiber product of categories. If $F_i : \mathscr{A}_1\rightarrow \mathscr{C}$ are functors, then $\mathscr{A}_1 \times_{\mathscr{C}} \mathscr{A}_2$ is the full subcategory of $\mathscr{A}_1 \times^h_{\mathscr{C}}\mathscr{A}_2$ made up of objects $(a,b,\mu)$ where $\mu$ is an isomorphism in $\mathscr{C}$ between $F_1(a)$ and $F_2(b)$. If $F_1 : \mathscr{A}_1 \rightarrow \mathscr{C}$ is a cofibration, then $\mathscr{A}_1 \times^h_\mathscr{C} \mathscr{A}_2$ is quasiisomorphic to $\mathscr{A}_1 \times_\mathscr{C} \mathscr{A}_2$.

\end{remark}


It is important to note that the homotopy fiber product of the diagram
\[
\mathscr{A}_1 \xrightarrow{F_1} \mathscr{C} \xleftarrow{F_2} \mathscr{A}_2
\]
is a full subcategory of $\mathscr{A}_1 \times_{F_1,F_2}\mathscr{A}_2$ made up of objects $(a,b,\mu)$ so that $[\mu]$ is an isomorphism in $\HH^0\mathscr{C}$, or equivalently, $\cone(\mu)$ is homotopically trivial. Therefore, Propositions  \ref{ref:klprop2} and \ref{ref:klprop3} have analogues for homotopy fiber products, as one would expect.
\begin{proposition}
If $G: \mathscr{A}_1' \rightarrow \mathscr{A}_1$ and $H:\mathscr{A}_2' \rightarrow \mathscr{A}_2$ are quasiequivalences, then there is quasiequivalence
\[
\eta_{G,H} : \mathscr{A}_1' \times_{\mathscr{C}}^h \mathscr{A}_2' \rightarrow \mathscr{A}_1 \times_{\mathscr{C}}^h \mathscr{A}_2.
\]
\end{proposition}
\begin{proposition}
If there are natural transformations $n_1 : F_1' \rightarrow F_1$ and $n_2 :F_1' \rightarrow F_2$ which are quasiisomorphisms of bimodules then there is a quasiequivalence $\gamma_{n_1,n_2}$ between the homotopy fiber products of $F_1',F_2'$ and $F_1,F_2$. 
\end{proposition}

\begin{remark}
One can check that the homotopy equalizer of a pair of functors between pretriangulated categories is a pretriangulated category by following the proof of \cite[Lemma 4.3]{kl}. Details will be provided in Proposition \ref{prop:istriang}.
\end{remark}
We may then slightly generalize \cite[Lemma 4.2]{bbb} to determine when a pair of elements of $\mathscr{A}_1 \times_{F_1,F_2}\mathscr{A}_2$ are homotopy equivalent.

\begin{lemma}\label{lemma:isomrec}
A pair of objects $\alpha = (a_1,a_2,\mu)$ and $\beta= (b_1,b_2,\tau)$ are isomorphic in $\HH^0(\mathscr{A}_1 \times_{F_1,F_2} \mathscr{A}_2)$ if and only if there are morphisms $\xi_i: a_i \rightarrow b_i$ so that $[\xi_1]$ and $[\xi_2]$ are isomorphisms and the diagram
$$
\begin{CD}
F_1(a_1) @> \mu >> F_2(a_2) \\
@VF_1(\xi_1) VV @VF_2(\xi_2) VV \\
F_1(b_1) @> \tau >> F_2(b_2)
\end{CD}
$$
commutes up to homotopy.
\end{lemma}
\begin{proof}
That the diagram above commutes up to homotopy means that there is a degree $-1$ element $g$ of $\hom_\mathscr{C}(F_1(a_1),F_2(b_2))$ so that
$$
d_\mathscr{C}g = F_2(\xi_2) \cdot \mu - \tau \cdot F_1(b_2).
$$
Therefore, by definition, $(\xi_1,\xi_2,g)$ is a closed, degree 0 homomorphism between $\alpha$ and $\beta$. One can follow the proof of \cite[Lemma 4.2]{bbb} to see that $(\xi_1,\xi_2,g)$ is a homotopy isomorphism if and only if $\xi_1$ and $\xi_2$ are homotopy isomorphisms.
\end{proof}

\section{Perverse schobers and perverse sheaves of categories}\label{sect:psc}

In this section, we will describe the work of Kapranov and Schechtman \cite{ks1} on perverse schobers and give the definition of the category of global sections. We will also describe a general construction of $K$-coordinatized perverse sheaves of categories on Riemann surfaces and we define their monodromy and categories of global sections.

\subsection{Perverse sheaves on a disc}

Perverse sheaves in on a complex manifold $M$ are elements of the heart of a t-structure on the constructible bounded derived category. In the case where $M$ is the disc then a perverse sheaf $\mathscr{F}$ is a (shifted) local system away from a finite set $\Sigma$ of points. Locally around a point in $\Sigma$, $\mathscr{F}$ has a concrete linear algebraic description.
\begin{theorem}[Galligo, Grainger and Maisonobe {\cite[Theorem II.2.3]{ggm}}]\label{thm:ggm}
Around a point $p \in \Sigma$, the data of a perverse sheaf is equivalent to the category of quadruples $(\phi,\psi,\mathrm{var},\mathrm{can})$ where $\phi$ and $\psi$ are finite dimensional vector spaces and $\mathrm{can} \colon \psi \rightarrow \phi, \mathrm{var}\colon \phi \rightarrow \psi$ are homomorphisms so that
\begin{equation} \label{eq:quivereq}
\mathrm{id}_\phi + \mathrm{can} \cdot \mathrm{var} \qquad \text{and} \qquad \mathrm{id}_\psi + \mathrm{var}\cdot \mathrm{can}
\end{equation}
are invertible. There is a similar description for perverse sheaves on $\mathbb{C}$ with stratification $\Lambda$ given by the coordinate axes and their intersections.
\end{theorem}
The method of proof of \cite[Theorem II.2.3]{ggm} gives us more than this. Let $K = \mathbb{R}_{\geq 0}$ be the intersection of the positive real axis in $\mathbb{C}$ with a disc $D$ centered around $0$ and let $\Lambda_K$ be the stratification of $K$ with one-dimensional stratum $\Lambda_{K,1} = \mathbb{R}_{> 0}$ and zero-dimensional stratum $\Lambda_{K,0} = \{0\}$. Then one builds a functor $\mathscr{R}_K$ from $\mathrm{Perv}(D,\Sigma)$ to $\mathrm{Constr}(K,\Lambda_K)$ which is full and faithful, and whose image is the set of constructible sheaves on $K$ with fiber $\phi$ on $\Lambda_{K,0}$ and $\psi$ at each point of $\Lambda_{K,1}$. The morphism $\mathrm{var}$ is the variation map, and $\mathrm{can}$ is the canonical map from the theory of perverse sheaves. If $\mathscr{F}$ is a perverse sheaf, then the corresponding vector space $\phi$ is the vanishing cycles sheaf of $\mathscr{F}$, $\phi$ is the nearby cycles sheaf of $\mathscr{F}$.

Kapranov and Schechtman have generalized this idea \cite{ks1,ks2}. They show that if $X$ is a Riemann surface $S$ with boundary and with a set of points $\Sigma$, then for an appropriate choice of skeleton $K$ of $S$ there is a description of $\mathrm{Perv}(S,\Sigma)$, in terms of the combinatorics of $K$.

\begin{figure}
\begin{tikzpicture}
\fill (0,0) circle[radius=2pt] ;
\draw[very thick] (0,0) to (.71,.71);
\draw[very thick] (0,0) to (1,0);
\draw[very thick] (0,0) to (.71,-.71);
\draw[very thick] (0,0) to (0,-1);
\draw[very thick] (0,0) to (-.71,.71);
\draw[very thick] (0,0) to (-1,0);
\draw[very thick] (0,0) to (-.71,-.71);
\node at (0,.8) {$\dots$};
\end{tikzpicture}\caption{A skeleton $K$ near the point $s$. \label{fig:wheel}}
\end{figure}
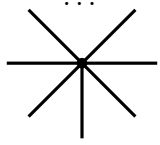

The most difficult part of their construction is to understand the restriction map from $\mathrm{Perv}(D,\Sigma)$ to $\mathrm{Constr}(K)$ around a point where $K$ looks like a wheel with $n$-spokes, as in Figure \ref{fig:wheel}. We refer to \cite{ks2} for details. 

A perverse sheaf on $(D,\Sigma)$, where $D$ is a closed disc in $\mathbb{C}$ and $\Sigma$ is a collection of points can be constructed from the following data. First, fix a skeleton $K$ given as follows.
\begin{enumerate}
\item A choice of a point $s \notin \Sigma$.
\item A choice of a path $\gamma_\infty$ from $s$ to the boundary of $D$ which does not pass through any point $p \in \Sigma$.
\item A choice of ordering of the points $p \in \Sigma$. Label the points in $\Sigma$ as $p_1,\dots, p_k$ based on this ordering.
\item A choice of continuous paths $\gamma_{i}$ from $s$ to each $p_i \in \Sigma$ so that $\gamma_i$ near $s$ are ordered counterclockwise as $\gamma_1,\dots, \gamma_k, \gamma_\infty$. We assume no paths intersect except at $s$.
\end{enumerate}
Then let $K=\gamma_\infty \cup \bigcup_{p\in \Sigma}\gamma_{p}$, with stratification given by points $\Lambda_{K,0}= \Sigma' = s \cup \Sigma$ and $\Lambda_{K,1} = K \setminus \Sigma'$. A perverse sheaf on $(D,\Sigma)$ is then equivalent to a constructible sheaf $\mathscr{S}$ on $(K,\Sigma')$ with fiber $\psi$ over any point not in $\Sigma'$ and fibers $\phi_i$ at each point in $\Sigma$, along with maps $\mathrm{var}_i : \phi_i \rightarrow \psi$ and $\mathrm{can}_i : \psi \rightarrow \phi_i$ satisfying Equation \ref{eq:quivereq}.
\begin{remark}\label{rmk:Kk}
Since the isotopy type of the skeleton $K$ depends only on the number of points in $\Sigma$, we will let $K_k$ be the skeleton associated to a set of $k$ points in $\Delta$, since we will sometimes refer only to the skeleton, forgetting its embedding in $\Delta$.
\end{remark}

\begin{figure}
\begin{tikzpicture}
\draw[dashed, very thick] (0,0) circle[radius=1 in] ;
\fill (-0.6,1.5) circle[radius= 2pt];
\fill (-0.6,-1.5) circle[radius= 2pt];
\fill (-1,-0.5) circle[radius= 2pt];
\fill (-1,0.5) circle[radius= 2pt];
\fill (1,1.5) circle[radius=2pt];

\fill (1,0) circle[radius= 2pt];

\draw[very thick] (2.5,0) to (1,0);
\draw[very thick] (1,0) .. controls (0.2,1.25) .. (-0.6,1.5);
\draw[very thick] (1,0) .. controls (0.2,-1.25) .. (-0.6,-1.5);
\draw[very thick] (1,0) .. controls (0.2,-0.5) .. (-1,-0.5);
\draw[very thick] (1,0) .. controls (0.2,0.5) .. (-1,0.5);
\draw[very thick] (1,0) .. controls (1.2,.5) .. (1,1.5);

\node at (1.4,1.5){$p_1$};
\node at (-0.9,1.5){$p_2$};
\node at (-1.3,0.5){$p_3$};
\node at (-1.3,-0.5){$p_4$};
\node at (-0.9,-1.5){$p_5$};

\node at (1.1,-0.2){$s$};

\end{tikzpicture}\caption{The skeleton of a perverse schober.}
\end{figure}
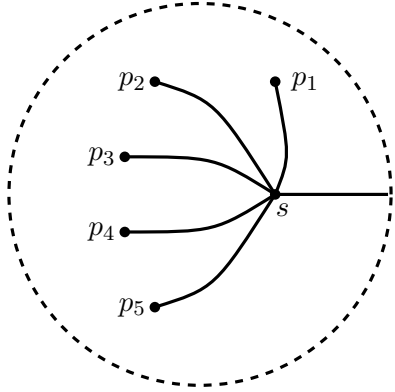

\subsection{Perverse schobers}\label{sect:pschob}
Let $D$ be the closed unit disc in $\mathbb{C}$ and let $K = \mathbb{R}_{\geq 0} \cap D$. Kapranov and Schechtman define a $K$-coordinatized perverse schober on $(D,0)$ to be the data of two idempotent closed pretriangulated dg categories $\mathscr{A}$ and $\mathscr{C}$ along with a spherical functor $F: \mathscr{A} \rightarrow \mathscr{C}$.
\begin{defn}[Anno and Logvinenko {\cite[Definition 1.1]{al}}]
Let $F$ be a quasifunctor $F$ from $\mathscr{A}$ to $\mathscr{B}$ with right and left adjoint quasifunctors $R$ and $L$. The quasifunctors $F$ is interpreted as an $\mathscr{A}$-$\mathscr{B}$ bimodule, and $R$ and $L$ are interpreted as $\mathscr{B}$-$\mathscr{A}$ bimodules. We define
\begin{align*}
T = \Cone(FR \rightarrow \mathrm{id}_\mathscr{B}),&\quad C' = \cone(LF \rightarrow \mathrm{id}_{\mathscr{A}}) \\
C = \cone(\mathrm{id}_\mathscr{A} \rightarrow RF)[-1] , & \quad T' = \cone(\mathrm{id}_\mathscr{B} \rightarrow FL)[-1]
\end{align*}
Here composition of functors is intepreted as tensor product of bimodules and the cones are taken in categories of $\mathscr{B}$-$\mathscr{B}$ and $\mathscr{A}$-$\mathscr{A}$ bimodules where appropriate. The quasifunctor $F$ is called \emph{spherical} if $T$ induces an autoequivalence of $\HH^0\mathscr{B}$ and the composition
$$
R \rightarrow RFL \rightarrow CL[1]
$$
induces a quasiisomorphism of dg bimodules between $R$ and $CL[1]$. Under these conditions, $T,T'$ and $C,C'$ form mutually quasiinverse pairs.
\end{defn}
The condition that $T$ is an autoequivalence is analogous to the condition that $\mathrm{id}_\phi + \mathrm{can} \cdot \mathrm{var}$ be invertible. In the case of perverse sheaves, there is only one map  between $\psi$ and $\phi$, whereas for a spherical functor, we require that there be two functors $R$ and $L$ from $\mathscr{C}$ to $\mathscr{A}$. The condition that $R \cong CL[1]$ should be read as saying that these two maps, while not identical, are at least compatible.

\begin{defn}
A \emph{$K$-coordinatized perverse schober} on $(D,\Sigma)$ is a choice of a skeleton $K$ as above, along with the data of a category $\mathscr{C}$, an ordered collection of categories $\mathscr{A}_1,\dots, \mathscr{A}_k$ corresponding to each point $p_i \in \Sigma$ and spherical functors $F_i : \mathscr{A}_i \rightarrow \mathscr{C}$. We will call this data $\mathscr{S}_K(\mathscr{A}_i,F_i)$. We will call the category $\mathscr{C}$ the \emph{fiber category} of $\mathscr{S}_K(\mathscr{A}_i,F_i)$.
\end{defn}
This data will give rise to what is essentially a constructible sheaf of categories on $K$ if we can determine an appropriate category as a fiber over the point $s$. This will be described in the next section.
\begin{remark}
If we assume that $\mathscr{C}$ has the property that the Serre functor of $\HH^0\mathscr{C}$ denoted $\mathsf{S}_{[\mathscr{C}]}$ is equal to the shift functor $ [m]$ for some integer $m$ and that $C$ induces the Serre functor on $\HH^0\mathscr{A}$ up to shift, then $R$ and $L$ are quasiequivalent up to shift. This occurs if $F$ admits a weak right Calabi--Yau structure \cite{kps}. An alternate definition of perverse schobers assuming that $F$ admits a weak right Calabi--Yau structure may also be interesting. We will see similar structures in Section \ref{sect:sod1}.
\end{remark}
\subsection{The category of global sections of a perverse schober}\label{sect:gsections}

Let $\mathscr{C}$ be a pretriangulated dg category, then define $A_2(\mathscr{C}) = \mathscr{C} \times_{\mathrm{id},\mathrm{id}} \mathscr{C}$. We will often replace this with the quasiequivalent category $\Tw \mathscr{C} \times_{\mathrm{id},\mathrm{id}} \Tw\mathscr{C}$. This category admits two obvious functors $f_1$ and $f_2$ to $\Tw \mathscr{C}$ sending $(a,b,\mu)$ to the objects $a$ and $b$ of $\Tw \mathscr{C}$ respectively. Since 
\[
\hom^i_{A_2(\mathscr{C})}((a,b,\mu),(c,d,\xi)) = \hom^i_\mathscr{C}(a,c) \oplus \hom^i_\mathscr{C}(b,d) \oplus \hom^{i-1}_\mathscr{C}(a,d)
\]
we have natural projections onto $\hom^i_\mathscr{C}(a,c)$ and $\hom^i_\mathscr{C}(b,d)$ which describe how these functors act on homomorphisms. There's another functor $f_3$ from $A_2(\mathscr{C})$ to $\Tw \mathscr{C}$ given by sending $(a,b,\mu)$ to $\cone(\mu)$. The object $\cone(\mu)$ is represented as a twisted complex by 
\[
\left( a \oplus b[-1], \left(\begin{matrix} 0 & \mu \\ 0 & 0  \end{matrix} \right) \right)
\]
A homomorphism $(r,s,t)$ of degree $i$ is sent to 
\[
\left( \begin{matrix} r & t \\ 0 & s \end{matrix} \right).
\]
We define the category $A_3(\mathscr{C})$ to be the homotopy fiber product of the diagram;
\[
\begin{tikzcd}
A_3(\mathscr{C}) \ar[r,"g"] \ar[d,"h"] & A_2(\mathscr{C}) \ar[d,"f_3"]\\
A_2(\mathscr{C}) \ar[r,"f_1"]  & \Tw \mathscr{C}
\end{tikzcd}
\]
From this we obtain four functors $f_1,\dots, f_4$ from $A_3(\mathscr{C})$ to $\Tw\mathscr{C}$ as, respectively, $f_1\cdot g, f_2 \cdot g, f_2\cdot h, f_3 \cdot h$. Applying this construction recursively, we define categories $A_k(\mathscr{C})$ for any $k$ along with functors $f_i :A_k(\mathscr{C}) \rightarrow \Tw \mathscr{C}$ for $i=1,\dots, k$. The category $A_k(\mathscr{C})$ is the category that plays the role of the stalk of the constructible sheaf of categories on $K$ at the point $s$. Therefore the following definition is natural.
\begin{defn}\label{defn:Gl}
The \emph{category of global sections} of a perverse schober $\mathscr{S}_K(\mathscr{C},\mathscr{A}_p,F_p)$, denoted by $\Gamma \mathscr{S}_K(\mathscr{C},\mathscr{A}_p,F_p)$ is the homotopy fiber product of the following diagram,
\begin{equation}\label{eq:hfib}
\mathscr{A}_1 \times \dots \times \mathscr{A}_k \xrightarrow{F_1 \times \dots \times F_k}
\Tw \mathscr{C} \times \dots \times \Tw \mathscr{C} \xleftarrow{f_1 \times \dots \times f_k} A_k(\mathscr{C}).
\end{equation}
The functor $f_{k+1} :A_k(\Tw \mathscr{C}) \rightarrow \Tw \mathscr{C}$ induces a functor $\Gamma \mathscr{S}_{K}(\mathscr{A}_i,F_i) \rightarrow \Tw \mathscr{C}$ which we denote $s_\infty$.
\end{defn}

The following lemma shows that for computational purposes, the word homotopy is not essential here.

\begin{lemma}\label{lemma:DK}
for any dg category $\mathscr{C}$, the functor
\[
\mathscr{C} \times_{\mathrm{id},\mathrm{id}}\mathscr{C} \xrightarrow{f_1 \times f_2}\mathscr{C} \times \mathscr{C}
\]
is a fibration in the DK model stucture.
\end{lemma}
\begin{proof}
We have that $f_1 \times f_2$ acts on objects and homomorphisms as
\[
(a,b,\mu) \mapsto (a,b), \qquad (r_1,r_2,g) \mapsto (r_1,r_2)
\]
Since $\hom_{(\Tw\mathscr{C})^2}((a,b),(c,d)) = \hom_{\Tw \mathscr{C}}(a,c) \times \hom_{\Tw\mathscr{C}}(b,d)$, any homomorphism can be lifted to a homomorphism of $A_2(\mathscr{C})$. We just need to show that if $(a,b,\mu) \in A_2(\mathscr{C})$ and there $(\xi,\tau): (a,b) \rightarrow (a',b')$ for $a,b,a',b'$ in $\Tw\mathscr{C}$ descends to an isomorphism in the homotopy category, then there is an object $(a',b',\mu')$ which is quasiisomorphic to $(a,b,\mu)$ in $A_2(\mathscr{C})$. If $\xi'$ is a quasiinverse of $\xi$, then we get a morphism $\mu' = \tau\cdot \mu\cdot \xi'$ from $a'$ to $b'$. We may apply Lemma \ref{lemma:isomrec} to see that $(a',b',\mu')$ is quasiisomorphic to $(a,b,\mu)$.
\end{proof}
Therefore, if there are only two categories $\mathscr{A}_1,\mathscr{A}_2$, the category $\Gamma\mathscr{S}_K(\mathscr{A}_i,F_i)$ can be computed as the ordinary fiber product of the diagram in Equation \ref{eq:hfib} instead of the homotopy fiber product. It is not hard to show that this just recovers $\mathscr{A}_1\times_{F_1,F_2} \mathscr{A}_2$.
\begin{proposition}\label{prop:bcase}
If $|\Sigma| = 2$ and $\mathscr{S}_K(\mathscr{A}_i,F_i)$ is a perverse schober then $\Gamma\mathscr{S}_K(\mathscr{A}_i,F_i)$ is quasiequivalent to $\mathscr{A}_1 \times_{F_1,F_2} \mathscr{A}_2$.
\end{proposition}
This quasiequivalence (which we denote by $Q$) can be written as follows. Objects of $\Gamma \mathscr{S}_K(\mathscr{A}_i,F_i)$ are given by
\[
((a,b,\mu), (r,s), (g_{a,r},g_{b,s}))
\]
where $(a,b,\mu) \in \mathrm{Ob}(A_2(\mathscr{C})), (r,s) \in \mathrm{Ob}(\mathscr{A}_1) \times \mathrm{Ob}(\mathscr{A}_2)$, and $g_{a,r} \in \hom^0_\mathscr{C}(a,F_1(r))$ and $g_{b,s} \in \hom^0_\mathscr{C}(b,F_2(s))$ are quasiisomorphisms. Objects $(c,d,\xi) \in \mathrm{Ob}(\mathscr{A}_1\times_{F_1,F_2}\mathscr{A}_2)$ are mapped by $Q$ to 
\[
((F_1(c),F_2(d),\xi), (c,d), (\mathrm{id}_{F_1(c)}, \mathrm{id}_{F_2(c)}))
\]
The action of $Q$ on homomorphisms is defined in a straightforward way. We obtain the following diagram.
\begin{equation}\label{eq:composing}
\begin{tikzcd} 
\mathscr{A}_1 \times_{F_1,F_2} \mathscr{A}_2 \arrow[drr, bend left, "x"] \arrow[ddr, bend right, "y"] \arrow[dr, "{Q}" description] & & 
\\ &  \Gamma\mathscr{S}_K(\mathscr{C},\mathscr{A}_i,F_i) \arrow[r, "p"] \arrow[d, "q"] & \mathscr{A}_1 \times \mathscr{A}_2 \arrow[d, "F_1 \times F_2"] 
\\ & A_2(\mathscr{C}) \arrow[r, "f_1 \times f_2"] & \mathscr{C} \times \mathscr{C} 
\end{tikzcd}
\end{equation}
where $x,y$ and $p,q$ are the natural maps. This diagram commutes except in the bottom right square, where it only commutes up to homotopy. Composing $y$ with $f_3 : A_2(\Tw \mathscr{C}) \rightarrow \Tw\mathscr{C}$ recovers the map $c_{F_1,F_2}$. In other words, $s_\infty \cdot Q = c_{F_1,F_2}$.

We will provide some notation for repeated gluing of dg categories.
\begin{defn}\label{defn:gluedcat}
Let us assume that we have $F_i: \mathscr{A}_i \rightarrow \Tw\mathscr{C}$ for $i = 1, \dots, k$. Then, as before, we may construct the category $\mathscr{A}_1 \times_{F_1,F_2} \mathscr{A}_2$ which comes equipped with a functor $c_{F_1,F_2} : \mathscr{A}_1 \times_{F_1,F_2} \mathscr{A}_2 \rightarrow \Tw \mathscr{C}$ sending $(a,b,\mu)$ to $\cone(\mu)$. We then have a pair of functors, $c_{F_1,F_2}, F_3$ to $\Tw \mathscr{C}$, so we may construct the category $(\mathscr{A}_1 \times_{F_1,F_2}\mathscr{A}_2) \times_{c_{F_1,F_2},F_3} \mathscr{A}_3$. Repeating this process, we obtain a category glued together from $\mathscr{A}_1,\dots, \mathscr{A}_k$. We will denote this category $\mathscr{G}\ell(\mathscr{A}_i,F_i)$. Let $c_{F_1,\dots, F_n}$ denote the final cone functor to $\Tw \mathscr{C}$.
\end{defn}
Now we will prove the following theorem.
\begin{theorem}\label{thm:glsect}
There is a quasiequivalence $Q_{F_1,\dots,F_k}:\mathscr{G}\ell(\mathscr{A}_i,F_i) \rightarrow  \Gamma\mathscr{S}_{K_k}(\mathscr{A}_i,F_i)$ so that $s_\infty \cdot Q_{F_1,\dots, F_k} = c_{F_1,\dots, F_k}$.
\end{theorem}
\begin{proof}
We proceed by induction. Assume that for $\mathscr{A}_1,\dots, \mathscr{A}_{k-1}$ and $F_1,\dots, F_{k-1}$ there is a quasiequivalence $Q_{F_1,\dots,F_{k-1}}:\mathscr{G}\ell(\mathscr{A}_i,F_i) \rightarrow \Gamma\mathscr{S}_{K_{k-1}}(\mathscr{A}_i, F_i)$. The case where $n=2$ is just Proposition \ref{prop:bcase}.

We now want to compute the homotopy limit of the diagram in Equation \ref{eq:bigdiag}.
\begin{equation}\label{eq:bigdiag}
\begin{tikzcd}
&&  && \mathscr{A}_1 \times \dots \times \mathscr{A}_{k-1} \ar[d,"F_1\times \dots \times F_{k-1}"]\\
& A_{k}(\mathscr{C}) \ar[r] \ar[d]& A_{k-1}(\mathscr{C}) \ar[rr,"f_1\times\dots \times f_{k-1}"] \ar[d,"f_k"] && \underbrace{\Tw\mathscr{C} \times \dots \times \Tw \mathscr{C}}_{(k-1)\text{ copies}} \\
& A_2(\mathscr{C}) \ar[r,"f_1"] \ar[d,"f_2"] & \Tw \mathscr{C} & & \\
\mathscr{A}_k \ar[r,"F_k"] & \Tw \mathscr{C} & & &
\end{tikzcd}
\end{equation}
where the square is homotopy cartesian. By assumption, the homotopy fiber product of the top right span is $\mathscr{G}\ell(\mathscr{A}_i,F_i)$. Therefore, this reduces to finding the homotopy limit of the following diagram.
\[
\begin{tikzcd}
& & \mathscr{G}\ell(\mathscr{A}_i,F_i) \ar[d] \ar[dd, bend left = 50, "c_{F_1,\dots, F_{k-1}}"] \\
& A_k(\mathscr{C}) \ar[r] \ar[d] & A_{k-1}(\mathscr{C}) \ar[d] \\
& A_2(\mathscr{C}) \ar[r,"f_1"] \ar[d,"f_2"] & \Tw \mathscr{C}  \\
\mathscr{A}_k \ar[r,"F_k"] & \Tw \mathscr{C} &
\end{tikzcd}
\]
Equivalently, since the square is homotopy cartesian, this reduces to the homotopy limit of the following diagram.
\[
\begin{tikzcd}
& \mathscr{G}\ell(\mathscr{A}_i,F_i) \times \mathscr{A}_k \ar[d,"c_{F_1,\dots, F_{k-1}}\times F_k"] \\
A_2(\mathscr{C}) \ar[r,"f_1\times f_2"] & \Tw \mathscr{C} \times \Tw \mathscr{C}
\end{tikzcd}
\]
By Proposition \ref{prop:bcase}, there is a quasiequivalence $Q$ between $\mathscr{G}\ell(\mathscr{A}_i,F_i) \times_{c_{F_1,\dots, F_{k-1}},F_k} \mathscr{A}_k$ and the homotopy fiber product of this diagram. By definition, then the homotopy limit of the diagram in Equation \ref{eq:bigdiag} is equivalent to $\mathscr{G}\ell(\mathscr{A}_i,F_i)$ and the natural map to $A_2(\mathscr{C})$ composed with $f_3$ recovers $c_{F_1,\dots, F_k}$ as required.

\end{proof}

\subsection{More general perverse sheaves of categories}\label{sect:gpsc}

We now define perverse sheaves of categories in greater generality. Let $S$ be a connected, oriented, compact topological surface with $n$ boundary components and $k$ marked points $\Sigma$. In the following, by embedded graph we mean a finite collection of vertices $\mathrm{Vert}(K)$ made up of points in the interior of $S$ and a finite set of edges $\mathrm{Ed}(K)$. Edges will be embeddings $g_e$ of the closed interval $[0,1]$ into $S$ so that one of $g_e(0)$ or $g_e(1)$ is a vertex and the other is either a vertex or contained in a boundary component. We assume $g_e((0,1))$ is contained in the interior of $S$ and that two edges $e_1$ and $e_2$ intersect only in elements of $\mathrm{Vert}(K)$. 

A perverse sheaf of categories starts with the following data.
\begin{enumerate}
\item[($*$)] A spanning graph $K$ on $S$. That is, a graph $K$ embedded into $S$ to which $S$ is homotopic. We further require that that each $p_i \in \Sigma$ is a univalent vertex of $K$ and that there are no bivalent vertices.
\item[($**$)] A category $\mathscr{A}_v$ for each vertex $v$ of $K$ and a category $\mathscr{C}_e$ for each edge of $K$. If $v$ is not in $\Sigma$ then $\mathscr{A}_v = A_n(\mathscr{C})$, where $n+1$ is the valency of $v$. All $\mathscr{C}_e$ are equivalent to some fixed category $\mathscr{C}$.
\item[($***$)] If $e$ is incident to $v$, then there is a functor $F_{v,e}:\mathscr{A}_v \rightarrow \mathscr{C}$. If $v$ is $(n+1)$-valent, then for some counterclockwise ordering of edges emanating from $v$, denoted $e_1,\dots, e_{n+1}$, we have $F_{e_i,v} =\phi_{e_i,v}\cdot f_i$ for $\phi_{e_i,v}$ some quasiautoequivalence of $\mathscr{C}$. Here $f_i$ are functors defined in the previous section. If $v \in \Sigma$ then $F_{v,e}$ is a spherical functor. 
\end{enumerate}

\begin{remark}\label{rmk:2edgy}
In the case where $e$ has both ends adjacent to the same vertex, we must adjust (3) so that we have one functor $F_{v,e}^1$ and $F_{v,e}^2$ for each adjacency, with corresponding autoequivalences $\phi_{v,e}^1$ and $\phi_{v,e}^2$ of $\mathscr{C}$. 
\end{remark}
\begin{remark}
This definition is essentially a simplified restatement of the characterization of perverse sheaves on $S$ with singularities in $\Sigma$ as in \cite{ks2} with vector spaces replaced by dg categories.
\end{remark}
This allows us to define something resembling a constructible presheaf $\mathscr{S}_k$ of dg categories on $K$ defined in terms of the following sets, which we will describe in the case where $K$ has no loops. If $v \in \mathrm{Vert}(K)$, then let $U_v$ the union of $v$ and the interior of the adjacent $1$-cells. For $e \in \mathrm{Ed}(K)$ let $U_e$ be the interior of the 1-cell corresponding to $e$. To the edges, we assign the categories $\Gamma(\mathscr{S}_K,U_e) = \mathscr{C}$, and to vertices we let $\Gamma(\mathscr{S}_K,U_v) = \mathscr{A}_v$, the functors $F_{v,e}$ determine $\Gamma(\mathscr{S}_K,U_v) \rightarrow \Gamma(\mathscr{S}_K,U_e)$ for $e$ an edge adjacent to $v$.

A $K$-coordinatized perverse sheaf of categories should be thought of as a categorification of a perverse sheaf on $S$ which is a local system when restricted to $S \setminus \Sigma$. Therefore, we should have a $K$-coordinatized local system of categories on the complement of $\Sigma$. We define this as follows. For each $p \in \Sigma$ choose a small open disc centered at $p$ denoted $D_p$, then let $S^\circ = S \setminus \cup_{p\in \Sigma} D_p$. We may construct a skeleton $K^\circ$ of $S^\circ$ as follows. If $e_p$ is the edge adjacent to $p$ and $q$ its other endpoint, and $p'$ is a point in $e_p$ which is not in  $D_p$, then we replace $e_p$ with an edge $e_{p'}$, the portion of $e_p$ between $q$ and $p'$. Then we attach one more edge $e_{F_p}$ to $p'$ with both endpoints at $p'$ and travelling in a small loop around $D_p$. Call the resulting skeleton $K^\circ$. We can equip $K^\circ$ with a $K^\circ$ coordinatized perverse sheaf of categories $\mathscr{S}^\circ$ which has the same data as $\mathscr{S}$ away from the new vertices $p'$. Each new vertex $p'$ is trivalent, and therefore, we must define three functors from $A_2(\mathscr{C})$ to $\mathscr{C}$. Let $F_{p',e_{p'}} = f_3$, let $F_{p',e_{F_p}}^1 = T_{F_p}\cdot f_1$ be the functor corresponding to the adjacency between $p'$ and $e_{F_p}$ to the immediate right of $e_{p'}$. Let $F_{p',e_{F_p}}^2 = f_2$ be the functor corresponding to the second adjacency between $p'$ and $e_{F_p}$. Here $T_{F_p}$ is the spherical twist associated to $F_p$.

A cycle $C$ is an ordered sequence of vertices $v_1,\dots, v_n$ with $v_n = v_1$ and an edge $e_{i,i+1}$ adjacent to both of $v_i$ and $v_{i+1}$. To any cycle in $K^\circ$, we define the monodromy around $C$ to be
\begin{equation}\label{eq:monodromy}
\mathsf{Mon}(C) = \phi_{v_{n},e_{n-1}}^{-1} \cdot \phi_{v_{n-1},e_{n-1}}\cdot \dots \cdot \phi_{v_2,e_1}^{-1} \cdot \phi_{v_1,e_1}
\end{equation}
which is a quasiautoequivalence of $\mathscr{C}$. Each cycle starting at a chosen vertex $v_1$ induces an element of $\pi_1(S^\circ,v_1)$, and since $S^\circ$ is homotopic to $K^\circ$, we see that every class in $\pi_1(S^\circ,v_1)$ can be represented this way. 
\begin{defn}
A \emph{$K$-coordinatized perverse sheaf of categories} on $(S,\Sigma)$ is a choice of a skeleton $K$ of $S$, a pretriangulated dg category $\mathscr{C},$ a collection of pretriangulated dg categories $\mathscr{A}_v$ for all $v \in \mathrm{Vert}(K)$ and functors $F_v : \mathscr{A}_v \rightarrow \mathscr{C}$ for each $v \in \mathrm{Vert}(K)$ satisfying the conditions given in ($*$), ($**$) and ($***$), so that $\mathsf{Mon}$ induces a representation of $ \pi_1(S^\circ,v_1)$ in the group of autoequivalences of $\HH^0\mathscr{C}$.
\end{defn}
\begin{ex}
A $K$-coordinatized perverse schober gives rise to a $K$-coordinatized perverse sheaf of categories whose autoequivalences $\phi_{v,e}$ are all trivial. The monodromy around any cycle is a product of spherical twists. 
\end{ex}
There is a diagram of categories associated to $\mathscr{S}_K$, given by the categories $\mathscr{A}_v$ for all $v \in \mathrm{Vert}(K)$ and $\mathscr{C}_e$ for $e \in \mathrm{Edge}(K)$ and the functors $F_{v,e}$ between them. The homotopy limit of this diagram is denoted $\Gamma\mathscr{S}_K$, and called the category of global sections of $\mathscr{S}_K$.
\begin{remark}
Such a structure should arise from a symplectic fibration over $S$ with smooth fibers away from $\Sigma$, and its category of global sections should be the partially wrapped Fukaya category of $S$. We allow that there be an indeterminate number of edges emanating towards each boundary component of $S$, which corresponds to allowing more general partial wrappings in the corresponding Fukaya category. 
\end{remark}

\section{Sheaves of categories on the cylinder, $\mathbb{P}^1$ bundles and a result of Segal}\label{sect:recon}

The goal of this section is to give two examples where perverse sheaves of categories appear naturally. First, we will show that a construction of Segal \cite{seg} is related to certain perverse sheaves of categories. Next we show that if $X$ is a smooth variety and $\mathscr{L}$ is a line bundle on $X$, then perverse sheaves of categories capture the derived category of coherent sheaves on $\DD^b_\mathrm{dg}(\mathbb{P}_X(\mathscr{O}_X \oplus \mathscr{L}))$. In both cases, we use perverse sheaves of categories on the cylinder $S^1 \times [0,1]$.

\subsection{Perverse sheaves of categories and spherical functors}\label{sect:Segal}

In \cite{seg}, Segal shows that if $\Phi$ is an autoequivalence of a triangulated category $\mathscr{D}$ which admits a dg enhancement $\mathscr{D}_\mathrm{dg}$ then there is a triangulated category $\overline{\mathscr{D}_\Phi}$ and a functor $j^* : \mathscr{D} \rightarrow \overline{\mathscr{D}_\Phi}$ which is spherical and whose cotwist is the autoequivalence $\Phi$. In this section, we will show how this construction can be recast in terms of perverse sheaves of categories.

The first statement that we must prove is the following proposition.
\begin{proposition}\label{prop:istriang}
If $\mathscr{A}$ and $\mathscr{C}$ are pretriangulated dg categories and $F_1,F_2 :\mathscr{A} \rightarrow \mathscr{C}$ are functors, then the homotopy equalizer $\mathscr{E}\mathrm{q}^h(F_1,F_2)$ is pretriangulated.
\end{proposition}
\begin{proof}
Recall that objects of the homotopy equalizer of a dg category can be represented by pairs $(a,\mu)$ where $a \in \mathscr{A}$ and $\mu : F_1(a) \rightarrow F_2(a)$ is degree 0 closed morphism whose induced morphism in $\HH^0 \mathscr{C}$ is an isomorphism. Our proof follows that of \cite[Proposition 4.3]{kl}.

If $a[1]$ is an element of $\mathscr{A}$ isomorphic to $\mathsf{Y}^a[1]$ then it is clear that $(a[1],\mu)$ is isomorphic to $\mathsf{Y}^{(a,\mu)}[1]$. Hence the shift functor makes sense. Morphisms between $(a,\mu)$ and $(b,\xi)$ in $\mathscr{E}\mathrm{q}^h(F_1,F_2)$ of degree $i$ are pairs of objects $(f,f') \in \hom^i(a,b) \oplus \hom^{i-1}(F_1(a),F_2(b))$, and $(f,f')$ is closed if and only if $df = 0$ and
\[
df' = \xi \cdot F_1(f) + (-1)^i F_2(f) \cdot \mu.
\]
Let $(f,f')$ be a closed degree 0 homomorphism from $(a,\mu)$ to $(b,\xi)$. Then we claim that the cone of $(f,f')$ is represented by the pair $c = \cone(f)$ along with a morphism that we construct as follows. According to \cite[Remark 3.1]{kl} we have the following morphisms and identities
\[
a[1] \xrightarrow{i} c \xrightarrow{p} a[1], \qquad b \xrightarrow{q} c \xrightarrow{s} b
\]
satisfying the conditions
\[
pi = \mathrm{id}_{a[1]}, \quad sq = \mathrm{id}_b, \quad ip+qs = \mathrm{id}_c,\quad dq = dp = 0, \quad di = qf, \quad ds = -fp.
\]
We define $\gamma$ to be
\[
\gamma = F_2(i) \mu F_1(p) + F_2(q) \xi F_1(s) + F_2(q) f' F_1(p)
\]
By construction, the diagram
\[
\begin{CD}
F_1(a) @>F_1(f)>> F_1(b) @>>> F_1(c) @>>> F_1(a[1]) \\
@V\mu VV @V\xi VV @V\gamma VV @V\mu[1] VV \\
F_1(a) @>F_2(f)>> F_2(b) @>>> F_2(c) @>>> F_2(a[1])
\end{CD}
\]
commutes up to homotopy in each square. Following the computations in \cite[Proposition 4.3]{kl}, this is closed and is of degree 0. The corresponding diagram in the homotopy category has rows which are distinguished triangles, the squares commute, and the first two vertical arrows and the last arrow are isomorphisms. Therefore $[\gamma]$ is an isomorphism. One then uses \cite[Remark 3.1]{kl} to see that $(c,\gamma)$ represents the cone of $(f,f')$ in $\mathscr{E}\mathrm{q}^h(F_1,F_2)$.
\end{proof}
\begin{remark}
This also implies that the homotopy fiber product of pretriangulated dg categories is again pretriangulated, since the homotopy fiber product is expressed as a homotopy equalizer.
\end{remark}

We will now explain Segal's construction. Segal \cite{seg} begins with a triangulated category $\mathscr{D}$ and an autoequivalence $\Phi$ of $\mathscr{D}$. He then constructs a category $\mathscr{D}_\Phi$ which is generated by objects $j^*a$ for all $a \in \mathscr{D}$, which satisfy
\[
\hom_{\mathscr{D}_\Phi}(j^*a,j^*b) = \hom_\mathscr{D}(x,y) \oplus \hom_\mathscr{D}(x, \Phi^{-1}y[1]).
\]
If $(f,f') \in \hom_{\mathscr{D}_\Phi}(j^*a,j^*b)$ and $(g,g') \in \hom_{\mathscr{D}_\Phi}(j^*b,j^*c)$, then
\[
(g,g') \cdot (f,f') = (g\cdot f, g' \cdot f + \Phi^{-1}(g) \cdot f').
\]
Segal shows that $j$ has right and left adjoints $j_*$ and $j^!$ respectively, is spherical, and that its spherical twist functor is $\Phi$. Our goal in this section is to show that if $\mathscr{D}$ has a dg extension $\mathscr{D}_\mathrm{dg}$, then there is a perverse sheaf of categories $\mathscr{S}_p(\mathscr{D}_\mathrm{dg},\Phi)$ whose category of global sections contains a dg extension of $\mathscr{D}_\Phi$ as a full subcategory. For simplicity, we will assume that $\Phi$ is an autoequivalence of $\mathscr{D}_{\dg}$ with inverse $\Phi^{-1}$.

Let $K_p$ be the skeleton of $S^1 \times [0,1]$ depicted in Figure \ref{fig:Sp}, that is, a graph with one trivalent vertex $v$ and two edges $e_c$ and $e_o$. The edge $e_c$ has both ends equal to $v$ and $e_o$ has one end connected to $v$ and the other is contained in one boundary component of $S \times [0,1]$. 
We let $\mathscr{S}_p(\mathscr{D}_\mathrm{dg},\Phi)$ be the perverse sheaf of categories on $K_p$ so that $\mathscr{A}_v = A_2(\mathscr{D}_\mathrm{dg}), \mathscr{C}_{e_v} = \mathscr{C}_{e_o} = \mathscr{D}_\mathrm{dg}, F_{v,e_c}^1 = f_1, F_{v,e_c}^2 = \Phi \cdot f_2$ (see Remark \ref{rmk:2edgy}) and so that $F_{v,e_o} = f_3$. This means that the category of global sections $\Gamma \mathscr{S}_p(\mathscr{D}_\mathrm{dg},\Phi)$ is the category $\mathscr{E}\mathrm{q}^h(f_1,\Phi\cdot f_2)$ where,
\[
f_1: A_2(\mathscr{D}_\mathrm{dg}) \rightarrow \mathscr{D}_\mathrm{dg}, \quad (a,b,\mu) \mapsto a, \qquad \Phi\cdot f_2 : A_2(\mathscr{D}_\mathrm{dg}) \rightarrow \mathscr{D}_\mathrm{dg}, \quad (a,b,\mu) \mapsto \Phi(b).
\]
There is a full subcategory of $\mathscr{E}\mathrm{q}^h(f_1,\Phi\cdot f_2)$ made up of objects,
\[
\{((a,\Phi^{-1}(a),0),\mathrm{id}_{a}) : a \in \mathscr{D}_\mathrm{dg}\}
\]
We will denote this category $\mathscr{D}_{\Phi,\mathrm{dg}}$.

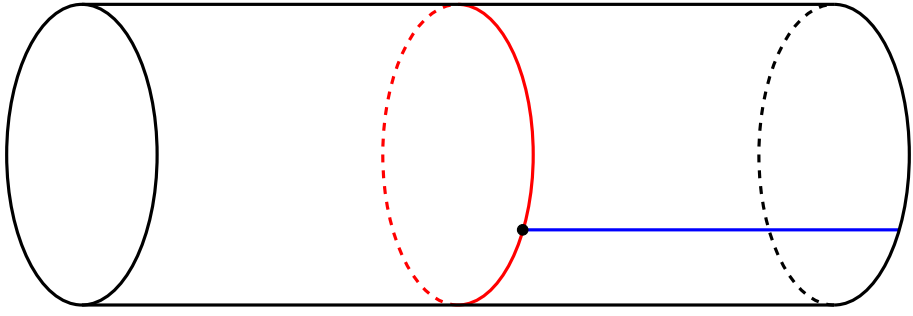
\begin{figure}
    \begin{tikzpicture}
\draw[fill = black] (0.86,-1) circle (1pt);

\draw[blue,very thick] (0.86,-1) to (5.86,-1);

	\draw[very thick] (5,-2) arc (270:450:1cm and 2cm);
    \draw[dashed,very thick] (5,-2) arc (270:90:1cm and 2cm);
	\draw[very thick] (-5,0) ellipse (1cm and 2cm);

\draw[red,very thick] (0,-2) arc (270:450:1cm and 2cm);
\draw[dashed,red, very thick] (0,-2) arc (270:90:1cm and 2cm);
\draw[very thick] (-5,2) to (5,2);
\draw[very thick] (-5,-2) to (5,-2);

\draw[fill = black] (0.86,-1) circle (2pt);

\end{tikzpicture}\caption{\label{fig:Sp} The skeleton $K_p$.}
\end{figure}

\begin{proposition}
The category $\mathscr{D}_{\Phi,\mathrm{dg}}$ is a dg enhancement of $\mathscr{D}_\Phi$.
\end{proposition}
\begin{proof}
In this proof, for simplicity, we will write $\mathscr{D}$ in place of $\mathscr{D}_\mathrm{dg}$. Let $\alpha = ((a,\Phi^{-1}(a),0),\mathrm{id}_a)$ and $\beta= ((b,\Phi^{-1}(b),0),\mathrm{id}_b)$. An element of $\hom_{\mathscr{E}\mathrm{q}^h(f_1,\Phi\cdot f_2)}(\alpha,\beta)$ is written as $f = ((s,r,t),u)$ where $s \in \hom_\mathscr{D}^i(a,b), r\in \hom_\mathscr{D}^i(\Phi^{-1}(a),\Phi^{-1}(b)), t \in \hom^{i-1}_{\mathscr{D}}(a, \Phi^{-1}(b))$ and $u \in \hom^{i-1}_\mathscr{D}(a,b[1]) = \hom^i_\mathscr{D}(a,b)$. We check that
\[
d((s,t,r),u) = ((ds,dt,dr),du - (s - \Phi(t)))
\]
So if $f$ is closed of degree $0$ then $s,r$ are closed, $t$ is closed and $du = s - \Phi(r)$. This means that $s$ and $\Phi(r)$ are homtopy equivalent, hence $r$ and $\Phi^{-1}(s)$ are homotopy equivalent, so that $d \Phi^{-1}(u) = \Phi^{-1}(s) - r$. Now we will check that
\begin{align*}
d((0,\Phi^{-1}(u),0),0) &= ((0 ,\Phi^{-1}(s) - r,0), -u) \\ &=( (s,\Phi^{-1}(s),t),0) - ((s, r, t),u).
\end{align*}
Therefore, every closed morphism is, up to sign, quasiequivalent to a morphism of the form $((s,\Phi^{-1}(s),t),0)$. Such morphisms are automatically closed since $s$ and $t$ are closed and are exact if and only if $s$ and $t$ are exact. Therefore, $\HH^0\mathscr{D}_{\Phi,\mathrm{dg}}$ has morphisms represented by closed degree 0 morphisms $s : a \rightarrow b$, and $r:a \rightarrow \Phi^{-1}(b)[1]$. Composition can be computed inside of $\mathscr{D}_{\Phi,\mathrm{dg}}$ to be;
\[
((s,\Phi^{-1}(s),t),0)\cdot((r,\Phi^{-1}(r),q),0) = ((s\cdot r, \Phi^{-1}(s\cdot t), r \cdot t + \Phi^{-1}(s) \cdot q), 0).
\]
Therefore composition of morphisms agrees with that of $\mathscr{D}_\Phi$ so $\mathscr{D}_{\Phi,\mathrm{dg}}$ is a dg enhancement of $\mathscr{D}_\Phi$.
\end{proof}

Two of Segal's functors then have natural interpretations in terms of our construction. The functor $j^*$ is comes from the functor
\[
a \mapsto ((a,\Phi^{-1}(a),0),\mathrm{id}_a) \in \mathscr{D}_{\Phi,\mathrm{dg}},
\]
and its right adjoint $j_*$ comes from the functor inherited from $f_3$ from $A_2(\mathscr{D}_{\mathrm{dg}})$ to $\Tw\mathscr{D}_\mathrm{dg}$, sending
\[
((a,b,\mu),\tau) \mapsto \Cone(\mu)
\]
This functor should be thought of as the pullback to a fiber functor, from the category of global sections of $\mathscr{S}_p(\mathscr{D}_\mathrm{dg},\Phi)$ to the stalk along a fiber of the edge of $K_p$ which intersects a boundary component of $S \times [0,1]$.

\subsection{$\mathbb{P}^1$ bundles and sheaves of categories}\label{sect:Pbub}

In this section, we will deal with the bounded derived category of coherent sheaves on varieties of the form $\mathbb{P}_X(\mathscr{O}_X\oplus \mathscr{L})$ for some line bundle $\mathscr{L}$ on a variety $X$. It is well known \cite{stz ,sibilla, kontsevich} that if $X = \mathrm{pt}$ then there is a perverse sheaf of categories on the skeleton in Figure \ref{fig:Sphi} whose category of global sections is $\DD^b(\mathbb{P}^1)$.
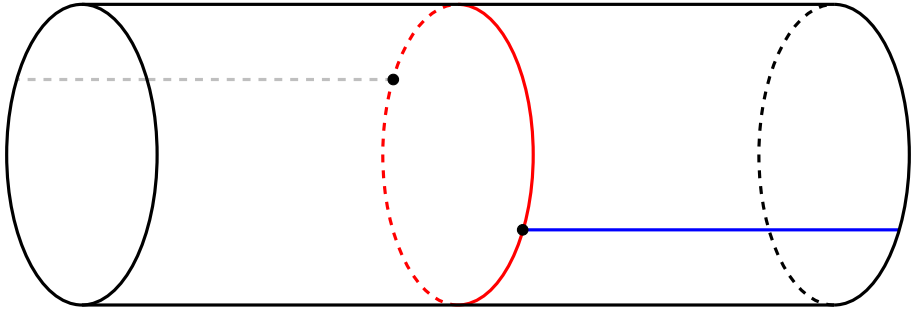
\begin{figure}
    \begin{tikzpicture}

\draw[fill = black] (0.86,-1) circle (1pt);
\draw[fill = black] (-0.86,1) circle (1pt);

\draw[blue,very thick] (0.86,-1) to (5.86,-1);
\draw[lightgray,dashed,very thick] (-0.86,1) to (-4.14,1);
\draw[lightgray, very thick,dashed] (-4.14,1) to (-5.86,1);

	\draw[very thick] (5,-2) arc (270:450:1cm and 2cm);
    \draw[dashed,very thick] (5,-2) arc (270:90:1cm and 2cm);
	\draw[very thick] (-5,0) ellipse (1cm and 2cm);

\draw[red,very thick] (0,-2) arc (270:450:1cm and 2cm);
\draw[dashed,red, very thick] (0,-2) arc (270:90:1cm and 2cm);
\draw[very thick] (-5,2) to (5,2);
\draw[very thick] (-5,-2) to (5,-2);

\draw[fill = black] (0.86,-1) circle (2pt);
\draw[fill = black] (-0.86,1) circle (2pt);
\end{tikzpicture}\caption{\label{fig:Sphi}The skeleton $K_\phi$.}
\end{figure}
We place the category $A_2(\Perf_k)$ at the two vertices, and the category $\Perf_k$ at each point along an edge. The category $\Gamma\mathscr{S}_\phi$ is the homotopy fiber product $A_2(\Perf_k) \times^h_{\Perf_k \times \Perf_k} A_2(\Perf_k)$ taken over the functors $f_1 \times f_2$ for both copies of $A_2(\Perf_k)$.

The idea is now that we let a pretriangulated dg category $\mathscr{C}$ be the general fiber of our perverse sheaf of categories on the same complex $K_\phi$, but this time, we will assume that there is monodromy around the central $S^1$. This is encoded by defining the transition functor along one of the edges to be an auto equivalence $\Phi^{-1}$ of $\mathscr{C}$. We have four edges $g_1,g_2$ and $h_1,h_2$ and two vertices $v_1,v_2$. The edge $g_1$ is adjacent to $v_1$ and has its second vertex in one of the boundary components. Similarly, $g_2$ is adjacent to $v_2$ and has its second vertex in the other boundary component. The edges $f_1$ and $f_2$ are adjacent to both $v_1$ and $v_2$. Both vertices are trivalent, so $\mathscr{A}_{v_1} = \mathscr{A}_{v_2} = A_2(\mathscr{C})$ and we let $F_{v_1,g_1} = F_{v_2,g_2} = f_3$, $F_{v_1,h_1} = F_{v_2,h_1} = f_1$ and $F_{v_1,h_2} = \Phi\cdot f_2$ and $F_{v_2,h_2} = f_2$. Call the resulting sheaf of categories $\mathscr{S}_\phi(\mathscr{C},\Phi)$. The category of global sections of $\mathscr{S}_\phi(\mathscr{C},\Phi)$ is the homotopy fiber product of
\[
 A_2(\mathscr{C}) \xrightarrow{f_1 \times \Phi\cdot f_2} \mathscr{C} \times \mathscr{C} \xleftarrow{f_1 \times f_2} A_2(\mathscr{C}).
\]
There are two natural subsets of the set of objects of $\Gamma\mathscr{S}_\phi(\mathscr{C},\Phi)$ written as
\[
S_1= \{ ((a,0,0), (a,0,0),\mathrm{id}_a) : a \in \mathscr{C}\},\quad  S_2=\{ ((0,b,0),(0,\Phi(b),0),\mathrm{id}_b) : b \in \mathscr{C}\}
\]
We see that if $A,A' \in S_1$ then $\hom_{\Gamma\mathscr{S}_\phi(\mathscr{C},\Phi)}(A,A')$ is the cone of the map
\begin{equation}\label{eq:conn}
\hom_{A_2(\mathscr{C})^2}((a,0,0) \times (a,0,0),(a',0,0)\times (a',0,0))  \longrightarrow \hom_{(\Tw \mathscr{C})^2}(a\times 0, a'\times 0 ).
\end{equation}
Applying definitions, we see that it is equivalent to the cone of
\[
\hom_{\mathscr{C}}(a,a') \oplus \hom_{\mathscr{C}}(a,a') \xrightarrow{(\mathrm{id} , -\mathrm{id})^T} \hom_{\mathscr{C}}(a,a').
\]
This cone is quasiisomorphic as a complex to the diagonal subcomplex of $\hom_\mathscr{C}(a,a')\oplus \hom_\mathscr{C}(a,a')$. Let us let $\mathscr{C}_1'$ be the subcategory of $\Gamma\mathscr{S}_\phi(\mathscr{C},\Phi)$ whose objects are $S_1$ and whose morphisms are given by the diagonal subcomplex. It is easy to see that this indeed forms a subcategory. Define similarly the subcategory $\mathscr{C}_2'$ of $\Gamma\mathscr{S}_\phi(\mathscr{C},\Phi)$. Let $\mathscr{C}_1$ and $\mathscr{C}_2$ be the full subcategories of $\Gamma\mathscr{S}_\phi(\mathscr{C},\Phi)$ on objects $S_1$ and $S_2$. Clearly, $\mathscr{C}_i$ is quasiequivalent to $\mathscr{C}_i'$ for $i = 1,2$.
\begin{proposition}\label{prop:ncP1}
Let $A = ((a,0,0),(a,0,0),\mathrm{id}_a) \in \mathrm{Ob}(\mathscr{C}_1)$ and $B = ((0,b,0),(0,b,0),\mathrm{id}_b) \in \mathrm{Ob}(\mathscr{C}_2)$. The following statements are true.
\begin{enumerate}
\item $\hom_{\Gamma\mathscr{S}_\phi(\mathscr{C},\Phi)}(B,A) = 0$.
\item $\hom_{\Gamma\mathscr{S}_\phi(\mathscr{C},\Phi)}(A,B)$ is isomorphic to $\hom_{\mathscr{C}}(a,b\oplus \Phi(b))$.
\item Every element of $\Gamma\mathscr{S}_\phi(\mathscr{C},\Phi)$ is quasiisomorphic to $\cone(f)$ for $A \in \mathrm{Ob}(\mathscr{C}_{1})$ and $B \in \mathrm{Ob}(\mathscr{C}_{2})$ and $f$ some closed element of $\hom^0_{\Gamma\mathscr{S}_\phi(\mathscr{C},\Phi)}(A,B)$.
\end{enumerate}
Therefore, if $\mathsf{G}$ is the $\mathscr{C}$-$\mathscr{C}$ bimodule sending $b \otimes a$ to $\hom_\mathscr{C}(a,b \oplus \Phi(b))$, there is a quasiequivalence between $\mathscr{C} \times_\mathsf{G}\mathscr{C}$ and $\Gamma\mathscr{S}_\phi(\mathscr{C},\Phi)$.
\end{proposition}
\begin{proof}
Seeing that condition (1) of the proposition holds is a direct check. To check condition (2), we suppose that $A \in \mathscr{C}_{1}$ and $B \in \mathscr{C}_{2}$. Then we see that
\begin{align*}
\hom_{\Gamma\mathscr{S}_\phi(\mathscr{C},\Phi)}(A,B) &= \cone(\hom_{\mathscr{C}}(a,b) \oplus \hom_{\mathscr{C}}(a,\Phi(b)) \xrightarrow{} \hom_{\mathscr{C}}(a,0) \oplus \hom_{\mathscr{C}}(0,b))\\
& = \hom_{\mathscr{C}}(a, b \oplus \Phi(b)).
\end{align*}
Therefore, we have a quasifull and faithful embedding of $\mathscr{C} \sqcup_{\mathsf{G}} \mathscr{C}$ into $\Gamma\mathscr{S}_\phi(\mathscr{C},\Phi)$ where $\mathsf{G}$ is the $\mathscr{C} \times \mathscr{C}$ bimodule sending $(a,b)$ to $\hom_{\mathscr{C}}(a,b \oplus \Phi(b))$.

If we can show that (3) holds, then we know that $\Tw(\mathscr{C} \sqcup_\mathsf{G} \mathscr{C})$ embeds fully and faithfully into $\Tw \Gamma\mathscr{S}_\phi(\mathscr{C},\Phi)$, which is quasiequivalent to $\Gamma\mathscr{S}_\phi(\mathscr{C},\Phi)$. Therefore, (3) implies the final statement of the proposition. Let us take then an object in $\Gamma \mathscr{S}_\phi(\mathscr{C},\Phi)$
\[
((a,b,\mu), (a',b',\mu'), \xi_a,\xi_b)
\]
where $\xi_a : a \rightarrow a'$ and $\xi_b: \Phi(b) \rightarrow b'$ are quasiisomorphisms in $\mathscr{C}$ and $(a,b,\mu)$ and $(a',b',\mu')$ are objects in $A_2(\mathscr{C})$. We may construct a homotopy commutative diagram
\[
\begin{CD}
a @>\tau>> \Phi(b) \\
@V\xi_aVV @A\xi_b'AA\\
a' @>\mu'>> b'
\end{CD}
\]
where $\xi_b'$ is a lift of $[\xi_b]^{-1} \in \mathrm{H}^0\hom_{[\mathscr{C}]}([b'],[\Phi(b)])$ to $\hom_\mathscr{C}(b',\Phi(b))$. Therefore, $(a,\Phi(b),\tau)$ is quasiisomorphic to $(a',b',\mu')$ by \cite[Lemma 4.2]{bbb}. Now we have that $(\mu,\tau,0)$ is a closed degree 0 morphism in $\hom_{\Gamma \mathscr{S}_\phi (\mathscr{C},T)}(A,B)$ where $A = ((a,0,0),(a,0,0),\mathrm{id})$ and $B = ((0,b,0),(0,\Phi(b),0),\mathrm{id})$. Following the construction of cones in \cite[Lemma 4.3]{kl}, the cone of $(\mu,\tau,0)$ is represented by the object $((a,b,\mu),(a,\Phi(b),\tau),\mathrm{id}_a,\mathrm{id}_b)$. Thus (3) follows and hence we have that $\mathscr{C}_1$ and $\mathscr{C}_2$ form a semiorthogonal decomposition of $\Gamma \mathscr{S}_\phi(\mathscr{C},\Phi)$.
\end{proof}
If $a,b\in \mathrm{Ob}(\mathscr{C})$ and $\mu = \mu_1\oplus \mu_2 \in \hom_{\mathscr{C}}(a,b) \oplus \hom_\mathscr{C}(a,\Phi(b))$ then the quasiequivalence mentioned in the statement of Proposition \ref{prop:ncP1} is given by 
\[
\Xi: (a,b,\mu) \mapsto ((a,b,\mu_1),(a,\Phi(b),\mu_2),\mathrm{id}_a,\mathrm{id}_{\Phi(b)})
\]
We wish to determine the implications of Proposition \ref{prop:ncP1} in geometry. We let $\mathscr{C} = \DD^b_\mathrm{dg}(X)$ and $\Phi_\mathscr{L} = (-)\otimes^\mathbb{L} \mathscr{L}$ for some line bundle $\mathscr{L}$ on $X$. 
\begin{theorem}\label{thm:orlov}
The categories $\Gamma \mathscr{S}_\phi(\DD^b_\mathrm{dg}(X),\Phi_\mathscr{L})$ and $\DD^b_\mathrm{dg}(\mathbb{P}_X({\mathscr{O}_X\oplus\mathscr{L}}))$ are quasiequivalent.
\end{theorem}
\begin{proof}
According to Orlov \cite[Example 3.9]{orl-glue}, if $\mathscr{E}$ is a rank 2 vector bundle on a variety $X$, then $\DD^b_\mathrm{dg}(\mathbb{P}_X(\mathscr{E}))$ is equivalent to the gluing of $\DD^b_\mathrm{dg}(X)$ to itself along the bimodule
\[
\mathsf{S}_\mathscr{E}(b,a) \cong \hom_{\DD^b_\mathrm{dg}(X)}(a,b\otimes^\mathbb{L} \mathscr{E}).
\]
This can be proved as follows. One has that $\DD^b(\mathbb{P}_X(\mathscr{E}))$ admits a semiorthogonal decomposition
\[
\langle \mathbf{L}p^*\DD^b(X), \mathbf{L}p^*\DD^b(X) \otimes^\mathbf{L} \mathscr{O}_\mathbb{P}(1) \rangle
\]
where $p: \mathbb{P}_X(\mathscr{E}) \rightarrow X$ is the natural projection map and $\mathscr{O}_\mathbb{P}(1)$ is the relative hyperplane bundle. Therefore, there is a quasiequivalence between $\DD^b_\mathrm{dg}(X) \times_{\mathsf{S}_\mathbb{P}} \DD^b_\mathrm{dg}(X)$ and $\DD^b_\mathrm{dg}(\mathbb{P}_X(\mathscr{E}))$ where 
\[
\mathsf{S}_\mathbb{P}(b,a) = \hom_{\DD^b_\mathrm{dg}(\mathbb{P}_X(\mathscr{E}))}(\mathbb{L}p^*a, \mathbb{L}p^*b \otimes^\mathbb{L}\mathscr{O}_{\mathbb{P}}(1)).
\]
This follows from \cite[Proposition 4.10]{kl}. Then, by \cite{orl-monoid}, we have that the cohomology of $\mathsf{S}_\mathbb{P}(b,a)$ is equal to that of $\mathsf{S}_\mathscr{E}(b,a)$, and they are quasiisomorphic bimodules. Finally \cite[Lemma 4.7]{kl} then says that $\DD^b_\mathrm{dg}(X) \times_{\mathsf{S}_\mathbb{P}} \DD^b_\mathrm{dg}(X)$ is quasiequivalent to $\DD^b_\mathrm{dg}(X) \times_{\mathsf{S}_\mathscr{E}} \DD^b_\mathrm{dg}(X)$.

If $\mathscr{E} = \mathscr{O}_X \oplus \mathscr{L}$, then $\mathsf{S}_\mathscr{E}$ is precisely the gluing functor $\mathsf{G}$ described in Proposition \ref{prop:ncP1}. We then deduce the result.
\end{proof}
\begin{remark}
More generally, if $\mathscr{C}$ is a dg category and $\Phi$ is an autoequivalence, then $\Gamma\mathscr{S}_\phi(\mathscr{C},\Phi)$ should be thought of as a noncommutative $\mathbb{P}^1$ bundle over $\mathscr{C}$. These are analogues of holomorphic families of categories with fibers equivalent to $\DD^b(\mathbb{P}^1)$ over the base $\mathscr{C}$. 
\end{remark}
\begin{remark}
This notion of noncommutative $\mathbb{P}^1$ bundles does not coincide with that of Van den Bergh \cite{vdbncP1} in general. That being said, Van den Bergh constructs the noncommutative $\mathbb{P}^1$ bundle $Q_\mathscr{E}$ using a sheaf bimodule $\mathscr{E}$. One may extend $\mathscr{E}$ to a dg $\DD^b_{\dg}(X)$-$\DD^b_{\dg}(X)$ bimodule $\mathsf{E}$. We expect that $\DD^b_{\dg}(\mathrm{qgr}(Q_\mathscr{E}))$ is quasiequivalent to $\DD^b_{\dg}(X) \times_\mathsf{E} \DD^b_{\dg}(X)$.

\end{remark}

\begin{remark}
The philosophy of perverse sheaves of categories is that their global sections represent partially wrapped Fukaya categories \cite{sylv} of fibrations over, in our case, Riemann surfaces. Theorem \ref{thm:orlov} is a good demonstration of this philosophy. Assume that we have a symplectic manifold $(M,\omega)$ and a symplectomorphism $\Psi$ of $(M,\omega)$. Then we may build a symplectic fibration over $S^1 \times \mathbb{R}$ which is a symplectic fiber bundle $(V,\sigma)$ with fiber $(M,\omega)$ and whose symplectic monodromy around the meridian of $S^1 \times \mathbb{R}$ is $\Psi$. Proposition \ref{prop:ncP1} suggests that the partially wrapped Fukaya category $\mathscr{WF}_s(V,\sigma)$ with two stops, one at a fiber in each boundary component, satisfies
\[
\DD^\pi \mathscr{WF}_s(V,\sigma) = \DD^\pi \mathscr{F}(M,\omega) \times_{\mathsf{S}_\Psi} \DD^\pi \mathscr{F}(M,\omega)
\]
where $\mathsf{S}_\Psi$ is the $\mathscr{F}(M,\omega)$-$\mathscr{F}(M,\omega)$ bimodule which satisfies
\[
\mathsf{S}_\Psi(b,a) = \hom_{\mathscr{F}(M,\omega)}(a, b \oplus \Psi(b)).
\]
A consequence of this is that if $(M,\sigma)$ is homologically mirror to an algebraic variety $X$, and the symplectomorphism $\Psi$ acts on $\mathscr{W}\mathscr{F}_s(M,\omega)$ in a way that is mirror to $(-)\otimes^\mathbb{L}\mathscr{L}$ for some line bundle $\mathscr{L}$ on $X$, then $(V,\sigma)$ should be the homological mirror to $\mathbb{P}_X(\mathscr{O}_X \oplus \mathscr{L})$. Similarly, the category $\DD^\pi \mathscr{WF}_s(V,\sigma)$ with $s$ denoting a stop at a fiber in one of the two boundary components should be equivalent to $\Perf(L)$ where $L$ is the total space of $\mathscr{L}$.

\end{remark}
%
%

\section{Perverse schobers and semiorthogonal decompositions}\label{sect:sod}

In this section, we will discuss how semiorthogonal decompositions and perverse schobers interact. In Section \ref{sect:sod1} we show that semiorthogonal decompositions combined with certain spherical functors give rise naturally to perverse schobers. In Section \ref{Sect:mutation}, we describe how mutation of semiorthogonal decompositions interacts with Kapranov and Schechtman's braid group action on perverse schobers.

\subsection{Semiorthogonal decompositions and Serre functors}\label{sect:sod1}


In this section, we will assume that all triangulated categories have finite dimensional spaces of homomorphisms, and that dg categories are enhancements of such triangulated categories. Let us recall a couple of concepts. If $\mathscr{T}$ is a triangulated category, the Serre functor $\mathsf{S}$ is an autoequivalence of $\mathscr{T}$ which has the property that there are functorial isomorphisms $\phi_{a,b} :\hom_\mathscr{T}(a,b) \rightarrow \hom_\mathscr{T}(b,\mathsf{S}(a))^\vee$ and so that the pairing $(\phi_{\mathsf{S}(b),\mathsf{S}(a)}^{-1})^\vee \cdot \phi_{a,b}$ is the morphism $\hom_\mathscr{T}(a,b) \rightarrow \hom_\mathscr{T}(\mathsf{S}(a),\mathsf{S}(b))$ induced by $\mathsf{S}$. Serre functors are unique if they exist \cite{bk2}. We will write $\mathsf{S}_\mathscr{T}$ if we want to emphasize the ambient category.

A triangulated subcategory $\mathscr{B}$ of $\mathscr{T}$ is called admissible if the embedding functor $\alpha : \mathscr{B} \rightarrow \mathscr{T}$ has right and left adjoints $\alpha^!$ and $\alpha^*$. To an admissible subcategory $\mathscr{B}$ there are admissible right and left orthogonal subcategories of $\mathscr{T}$ denoted $\mathscr{B}^\perp$ and $^\perp\mathscr{B}$ respectively and defined
$$
\mathscr{B}^\perp = \{ t \in \mathscr{T} : \hom_\mathscr{T}(b,t) = 0, \forall \,\, b \in \mathscr{B}\}, \qquad \prescript{\perp}{}{\mathscr{B}}= \{ t \in \mathscr{T} : \hom_\mathscr{T}(t,b) = 0, \forall\,\, b \in \mathscr{B}\}
$$
An ordered collection $(\mathscr{B}_1,\dots ,\mathscr{B}_k)$ of admissible subcategories forms a semiorthogonal decomposition of $\mathscr{T}$ if $\mathscr{B}_i \subseteq \mathscr{B}_j^\perp$ if $i < j$ and every element of $\mathscr{T}$ can be obtained by repeatedly taking cones of morphisms between elements of $\mathscr{B}_1,\dots, \mathscr{B}_k$. If $\mathscr{B}$ is an admissible subcategory of $\mathscr{T}$, then there are semiorthogonal decompositions $(\mathscr{B}, \prescript{\perp}{}{\mathscr{B}})$ and $(\mathscr{B}^\perp,\mathscr{B})$.

If $\mathscr{A}$ is a dg extension of $\mathscr{T}$, and $\mathscr{T}$ admits a semiorthogonal decomposition, we have dg extensions $\mathscr{A}_i$ of each $\mathscr{B}_i$ defined to be the full subcategory of $\mathscr{A}$ whose objects correspond to objects in $\mathscr{B}_i$. By a slight abuse of notation, we will say that $(\mathscr{A}_1,\dots, \mathscr{A}_k)$ is a semiorthogonal decomposition of $\mathscr{A}$ in this case.
The following result of Addington tells us how spherical functors interact with semiorthogonal decompositions. This extends a result of Seidel and Thomas \cite{st}.

\begin{proposition}[Addington {\cite[Proposition 1.1]{add}}]\label{prop:add}
Let $\mathscr{A}$ and $\mathscr{C}$ be pretriangulated dg categories whose homotopy categories have Serre functors $\mathsf{S}_{[\mathscr{A}]}$ and $\mathsf{S}_{[\mathscr{C}]}$. Assume that there is a spherical functor $F: \mathscr{A} \rightarrow \mathscr{C}$ with cotwist $C$ which induces $\mathsf{S}_{[\mathscr{A}]}[\ell]$ for some integer $\ell$. If $\alpha:\mathscr{B} \rightarrow \mathscr{A}$ is an admissible functor then $F\alpha$ is spherical and its cotwist induces $\mathsf{S}_{[\mathscr{B}]}[\ell]$.
\end{proposition}

Therefore if we have a functor $F$ satisfying the conditions of Proposition \ref{prop:add}, and if $\mathscr{A}$ admits a semiorthogonal decomposition $ (\mathscr{A}_1,\dots, \mathscr{A}_k)$, with admissible functors $\alpha_j: \mathscr{A}_j \rightarrow \mathscr{A}$ then we get an ordered sequence of spherical functors $F_j = F\alpha_j : \mathscr{A}_j \rightarrow \mathscr{C}$.
\begin{corollary}
In the situation of Proposition \ref{prop:add}, there is a natural $K$-coordinatized perverse schober $\mathscr{S}_K(\mathscr{C},\mathscr{A}_i,F_i)$ for an appropriate choice of $K$.
\end{corollary}
We would like to check that the category of global sections of this $\mathscr{S}(\mathscr{C},\mathscr{A}_i,F_i)$ recovers $\mathscr{A}$. We will say a subcategory $\mathscr{B}$ of a pretriangulated dg category $\mathscr{A}$ is admissible if $\HH^0\mathscr{B}$ is an admissible subcategory of $\HH^0\mathscr{A}$. We will define $\prescript{\perp}{}{\mathscr{B}}$ and $\mathscr{B}^\perp$ as the full subcategories of $\mathscr{A}$ whose objects correspond to objects of $\prescript{\perp}{}{\HH^0\mathscr{B}}$ and $\HH^0\mathscr{B}^\perp$ respectively.
\begin{lemma}\label{lemma:glue}
Let $\mathscr{A}$ be a pretriangulated dg category and let $\mathscr{B}$ be an admissible subcategory. Let $a \in \mathscr{B}$ and $b$ in $\mathscr{B}^\perp$, and assume that there's a spherical functor $F : \mathscr{A} \rightarrow \mathscr{C}$ with right adjoint $R$ so that the cotwist $C$ of $F$ on $\mathscr{A}$ is a shift of the Serre functor $\mathsf{S}_{[\mathscr{A}]}$. Then we have a quasiisomorphism of complexes
$$
\hom_\mathscr{A}(a,b) \cong {\hom_\mathscr{C}}(F(a),F(b)).
$$
\end{lemma}
\begin{proof}
The condition that the cotwist $C$ of $F$ is a shift of the Serre functor says that we have a distinguished triangle,
$$
b \longrightarrow RF(b) \longrightarrow C(b).
$$
Applying the triangulated functor $\hom_\mathscr{A}(a,-)$ from $\mathscr{A}$ to $\mathrm{Ch}_k$ to this exact triangle, we get a distinguished triangle
$$
\hom_{\mathscr{A}}(a,b) \longrightarrow \hom_\mathscr{A}(a, RF(b)) \longrightarrow \hom_\mathscr{A}(a, C(b))
$$
By adjunction, we have that the second term is quasiisomorphic to $\hom_\mathscr{C}(F(a), F(b))$ and the third term is quasiisomorphic to $\hom_\mathscr{A}(b[k], a)$ by the fact that $C$ is equivalent to the Serre functor up to shift. Since $a \in \mathscr{B}$ and $b \in \mathscr{B}^\perp$, this vanishes up to homotopy and hence we get the required quasiisomorphism of complexes.
\end{proof}
Let us assume that there are dg categories $\mathscr{C}$ and $\mathscr{A}$ and that there is a spherical functor $F : \mathscr{A} \rightarrow \mathscr{C}$ with cotwist $C$ which is equivalent to $\mathsf{S}_{[\mathscr{A}]}[k]$. Let us also assume that there is an admissible subcategory $\mathscr{B}$ of $\mathscr{A}$. Then we have an immediate corollary to Lemma \ref{lemma:glue}.
\begin{corollary}\label{cor:qi}
There is a quasiisomorphism between the $\mathscr{B}$-$\mathscr{B}^\perp$ dg bimodules defined as
$$
\varphi: b \otimes a \in (\mathscr{B}^\perp)^\mathrm{op} \otimes \mathscr{B} \mapsto \hom_\mathscr{A}(a,b), \qquad \varpi: b \otimes a \in  (\mathscr{B}^\perp)^\mathrm{op} \otimes \mathscr{B} \mapsto \hom_\mathscr{C}(F(a),F(b)).
$$
\end{corollary}
There is a quasiequivalence, \cite[Proposition 4.10]{kl} $c_\varphi: \mathscr{B} \times_\varphi \mathscr{B}^\perp \rightarrow \Tw \mathscr{A}$ which sends a triple $(a,b,g)$ to $\cone(g)$ in $\Tw \mathscr{A}$. There is also a natural functor $c_\varpi: \mathscr{B} \times_\varpi \mathscr{B}^\perp \rightarrow \Tw \mathscr{C}$ which operates the same way, sending $(a,b,\mu)$ to $\cone(\mu)$. Finally, there is a functor $f: \mathscr{B} \times_\varphi \mathscr{B}^\perp \rightarrow \mathscr{B} \times_\varpi \mathscr{B}^\perp$ sending $(a,b,g)$ to $(a,b, F(g))$, which is a quasiequivalence of categories \cite[Lemma 4.7]{kl}. We obtain the following commutative diagram of categories
\begin{equation}\label{comm}
\begin{CD}
\mathscr{B} \times_\varphi \mathscr{B}^\perp @>f>> \mathscr{B} \times_\varpi \mathscr{B}^\perp \\
@Vc_\varphi VV @Vc_\varpi VV \\
\Tw \mathscr{A} @>F^{\Tw} >> \Tw\mathscr{C} \\
@AAA @AAA \\
\mathscr{A} @>F >> \mathscr{C}
\end{CD}
\end{equation}
The vertical maps on the bottom are the natural embeddings, which are quasiequivalences since we assume that $\mathscr{A}$ and $\mathscr{C}$ are pretriangulated. The chain of functors along the left hand side and top of Equation \ref{comm} provide a quasiequivalence between $\mathscr{B} \times_\varphi \mathscr{B}^\perp$ and $\mathscr{A}$. This chain of quasiequivalences also identifies $c_\varphi$ with $F$ as bimodules. We record this as a corollary;
\begin{corollary}\label{cor:obvs}
Keeping the notation of Corollary \ref{cor:qi}, the categories $\mathscr{B} \times_\varphi \mathscr{B}^\perp$ and $\mathscr{A}$ are quasiequivalent. Furthermore $F $ and $c_\varpi$ are identified under the corresponding chain of quasiequivalences.
\end{corollary}
Using Propositions \ref{ref:klprop2} and \ref{ref:klprop3}, we have that for any dg category $\mathscr{D}$ and functor $G:\mathscr{D} \rightarrow \mathscr{C}$, there is a chain of quasiequivalences between $\mathscr{D} \times_{G,F}\mathscr{A}$ and $\mathscr{D} \times_{G,c_\varphi} (\mathscr{B}\times_\varphi \mathscr{B}^\perp)$ so that $c_{G,c_\varphi}$ is quasiequivalent as a dg bimodule to $c_{G,F}$. We will now prove the following result. 

\begin{theorem}\label{thm:recon}
Let $F:\mathscr{A} \rightarrow \mathscr{C}$ be a spherical functor whose cotwist induces the Serre functor on $\HH^0\mathscr{A}$ (up to shift) and assume that $\HH^0\mathscr{A}$ admits a semiorthogonal decomposition into categories $\HH^0\mathscr{A}_1,\dots, \HH^0\mathscr{A}_k$ where $\mathscr{A}_1,\dots, \mathscr{A}_k$ are full subcategories of $\mathscr{A}$. Let $\alpha_i$ be the full and faithful embedding of $\mathscr{A}_i$ into $\mathscr{A}$, and let $F_i = F\alpha_i$. Then the category $\mathscr{A}$ is quasi-equivalent to $\mathscr{G}\ell(\mathscr{A}_i,F_i)$. Furthermore, the natural functor $c_{F_1,\dots, F_n} : \mathscr{G}\ell (\mathscr{A}_i,F_i) \rightarrow \mathscr{C}$ is quasiisomorphic to $F$.
\end{theorem}
\begin{proof}
Let us proceed by induction. We have that $\mathscr{A}_1,\dots, \mathscr{A}_k$ forms a semiorthogonal decomposition for $\mathscr{A}$ and $\alpha_j:\mathscr{A}_j \rightarrow \mathscr{A}$ are the corresponding embeddings and $F :\mathscr{A} \rightarrow \mathscr{C}$ is the spherical functor to $\mathscr{C}$. Define subcategories $\mathscr{A}_{1,j}$ of $\mathscr{A}$ to be the full subcategory whose objects are those of $\mathscr{A}_1,\dots, \mathscr{A}_j$. We see that $\HH^0\mathscr{A}_{1,j}$ admits a semiorthogonal decomposition $\HH^0\mathscr{A}_1,\dots, \HH^0\mathscr{A}_j$, hence the composition of the embedding functor and the functor $F$ give a functor $F_{1,j}:\mathscr{A}_{1,j} \rightarrow \mathscr{C}$ which is spherical with cotwist $\mathsf{S}_{\mathscr{A}_{1,j}}[\ell]$. It follows from Corollary \ref{cor:obvs} that $\mathscr{A}_{1,j+1}$ is quasiisomorphic to $\mathscr{A}_{1,j} \times_{F_{1,j},F_{j+1}} \mathscr{A}_{j+1}$. Furthermore, the functor $c_\varpi$ coincides with $F_{1,j+1}$ under this quasiisomorphism. The theorem follows then by recursively applying this observation.
\end{proof}
As a corollary to Theorem \ref{thm:recon} and Theorem \ref{thm:glsect} we have the following result. 
\begin{corollary}\label{cor:extw}
Let $\mathscr{A}$ be a pretriangulated category which admits a semiorthogonal decomposition $\mathscr{A} = (\mathscr{A}_1,\dots, \mathscr{A}_k)$ and a spherical functor $F : \mathscr{A} \rightarrow \mathscr{C}$ whose cotwist induces the Serre functor on $\HH^0\mathscr{A}$ up to shift. The perverse schober $\mathscr{S}_{K}(\mathscr{C},\mathscr{A}_i,F_i)$ has category of global sections quasiequivalent to $\mathscr{A}$.
\end{corollary}

%
%

\subsection{Mutations and $K$-coordinatization}\label{Sect:mutation}

\begin{figure}
\begin{tikzpicture}
\node at (0.3,-0.3){$v$};
\node at (1.5,1){$p_{i+1}$};
\node at (-1.3,1){$p_i$};
\fill (0,0) circle[radius=2pt] ;
\fill (1,1) circle[radius=2pt] ;
\fill (-1,1) circle[radius=2pt] ;
\draw[very thick] (0,0) .. controls (0.7,0.5) .. (1,1);
\draw[very thick] (0,0) .. controls (-0.7,0.5) .. (-1,1);
\draw[very thick] (0,0) to (0,-0.5);
\end{tikzpicture}\qquad \qquad
\begin{tikzpicture}
\node at (0.3,-0.3){$v$};
\node at (1.5,1){$p_{i+1}$};
\node at (-1.3,1){$p_i$};
\fill (0,0) circle[radius=2pt] ;
\fill (1,1) circle[radius=2pt] ;
\fill (-1,1) circle[radius=2pt] ;
\draw[very thick] (0,0) .. controls (0.7,0.5) .. (1,1);
\draw[very thick] (0,0) .. controls (4.5,0.5) and (0,3.5) .. (-1,1);
\draw[very thick] (0,0) to (0,-0.5);
\end{tikzpicture}\caption{\label{figure:mutation}Action of $\sigma_i$ on $K$.}
\end{figure}
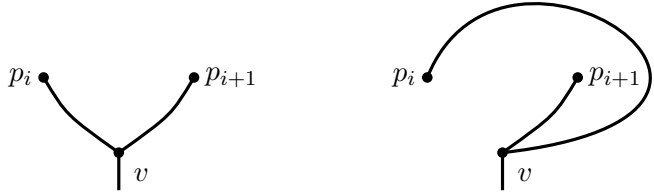

Kapranov and Schechtman \cite[Section 2.A]{ks1} define an action of $\mathrm{Br}_{|\Sigma|}$ on the set of all $K$-coordinatized perverse schobers on $(D,\Sigma)$. This acts on skeleta $K$, since $\mathrm{Br}_{|\Sigma|} = \pi_0 \mathrm{Diff}^+(D,\Sigma,s)$ (see \cite[2.B]{ks1} for details or Figure \ref{figure:mutation} for an illustration).

This group is generated by elements $\sigma_i$. We define an action of $\mathrm{Br}_{|\Sigma|}$ on the set of all perverse schobers by saying that $\sigma_i$ takes a schober $\mathscr{S}_K(\mathscr{C},\mathscr{A}_j,F_j)$ and produces $\mathscr{S}_{\sigma_i(K)}(\mathscr{C},\mathscr{A}'_j, F_j')$
so that $F_j' = F_j$ and $\mathscr{A}_j' = \mathscr{A}_j$ if $j \neq i,i+1$, $\mathscr{A}_{i+1}' = \mathscr{A}_i, \mathscr{A}_i' = \mathscr{A}_{i+1}$, and that $F_i' = F_{i+1}, F_{i+1}' = T_{i+1}' F_i$ where $T_{i+1}$ is the spherical twist associated to $F_i$. The action of $\sigma_i$ on $K$ is depicted in Figure \ref{figure:mutation}.

Assume we have a (dg enhancement of a) triangulated category $\mathscr{T}$ which is equipped with a semiorthogonal decomposition $\langle \mathscr{A}_1,\dots, \mathscr{A}_{n}\rangle$ and a spherical functor $F: \mathscr{T} \rightarrow \mathscr{C}$, and induced spherical functors $F_i:\mathscr{A}_i \rightarrow \mathscr{C}$. Then we may {\it mutate} this semiorthogonal decomposition so as to change the order of $\mathscr{A}_i$ and $\mathscr{A}_{i+1}$. This then produces new spherical functors $F_i' : \mathscr{A}_i \rightarrow \mathscr{C}$ and a new $K$-coordinatized perverse schober by Corollary \ref{cor:extw}. We will prove that the relation between $F_i$ and $F_i'$ is precisely as described by Kapranov and Schechtman. For the rest of this section, we will operate exclusively with triangulated categories.

Mutations are defined as follows. Assume that we have a semiorthogonal decompositions $\mathscr{T} = \langle \mathscr{A},\prescript{\perp}{}{\mathscr{A}} \rangle$ and $\langle \mathscr{A}^\perp, \mathscr{A} \rangle$. Then there is an embedding functor $\alpha : \mathscr{A} \rightarrow \mathscr{T}$ with right and left adjoints $\alpha^!$ and $\alpha^*$. The right and left mutation functors of $\mathscr{T}$ are given by
\[
\alpha \alpha^!(a) \rightarrow a \rightarrow \mathit{L}_{\mathscr{A}}(a), \qquad R_\mathscr{A}(a) \rightarrow a \rightarrow \alpha \alpha^*(a).
\]
The functors $L_\mathscr{A}$ and $R_\mathscr{A}$ send $\mathscr{T}$ to $\mathscr{T}$, and specifically map $\mathscr{A}$ onto the subcategories $\mathscr{A}^\perp$ and $\prescript{\perp}{}{\mathscr{A}}$ respectively. Furthermore, they give mutually inverse equivalences between $\mathscr{A}^\perp$ and $\prescript{\perp}{}{\mathscr{A}}{}{}$. If we have a semiorthogonal decomposition $\mathscr{T} = \langle \mathscr{A}_1,\dots ,\mathscr{A}_k\rangle$ we may perform mutations at a specific component of this semiorthogonal decomposition, by noting that we also have semiorthogonal decompositions
\[
\langle \mathscr{A}_1,\dots, \langle \mathscr{A}_i, \mathscr{A}_{i+1} \rangle, \dots , \mathscr{A}_k \rangle
\]
Therefore, we can define
\[
L_{\mathscr{A}_i}  \langle \mathscr{A}_1,\dots , \mathscr{A}_k \rangle
\]
to be the semiorthogonal decomposition obtained by replacing the semiorthogonal summand $\langle \mathscr{A}_i,\mathscr{A}_{i+1}\rangle$ with $\langle L_{\mathscr{A}_i}\mathscr{A}_{i+1},\mathscr{A}_i \rangle$, and similarly, we define
\[
R_{\mathscr{A}_{i+1}} \langle \mathscr{A}_1,\dots, \mathscr{A}_k \rangle
\]
to be the semiorthogonal decomposition obtained by replacing the semiorthogonal summand $\langle \mathscr{A}_i, \mathscr{A}_{i+1} \rangle$ with the new semiorthogonal summand $\langle \mathscr{A}_{i+1}, R_{\mathscr{A}_{i+1}}(\mathscr{A}_i) \rangle$.
\begin{proposition}\label{prop:mutat}
Assume that we have a semiorthogonal decomposition $\mathscr{T} = \langle \mathscr{A}, \mathscr{B} \rangle$ and a spherical functor $F:\mathscr{T} \rightarrow \mathscr{C}$ so that $C_F = \mathsf{S}_{\mathscr{T}}[\ell]$ for some integer $\ell$. Let $\alpha: \mathscr{A} \rightarrow \mathscr{T}$ and $\beta: \mathscr{B} \rightarrow \mathscr{T}$ be the natural embeddings and let $F_\alpha : \mathscr{A} \rightarrow \mathscr{C}$ and $F_\beta : \mathscr{B} \rightarrow \mathscr{C}$ be the spherical functors from $\mathscr{A}$ and $\mathscr{B}$ obtained from Proposition \ref{prop:add}. Define $T_\alpha, T_\beta,C_\alpha$ and $C_\beta$ as the corresponding twist and cotwist functors. Then
\[
F L_\mathscr{A} \beta \cong T_\alpha  F_\beta, \qquad F R_\mathscr{B}  \alpha \cong T^{'}_\beta F_\alpha.
\]
\end{proposition}
\begin{proof}
It is known \cite[Remark 2.11]{kperry} that there are isomorphisms of functors
\[
L_\mathscr{A}\beta \cong \mathsf{S}_\mathscr{T} \beta \mathsf{S}_\mathscr{B}^{-1}, \qquad R_\mathscr{B} \alpha \cong \mathsf{S}_\mathscr{T}^{-1} \alpha \mathsf{S}_\mathscr{A}.
\]
The rest of the proposition is provided by compatibility conditions between spherical functors, Serre functors and twist and cotwist functors. By construction, we have that $C_F \cong \mathsf{S}_\mathscr{T}$. The fact that $F$ is spherical implies that $FC_F \cong T_FF[-2]$ (see \cite[Section 1.3]{add}). Consequently, $F\mathsf{S}_\mathscr{T} \cong FC_F[-\ell] \cong T_FF[\ell-2]$. Therefore,
\[
FL_\mathscr{A} \beta \cong T_F F\beta \mathsf{S}_\mathscr{B}^{-1}[\ell-2] \cong T_F  F_\beta \mathsf{S}_\mathscr{B}^{-1}[\ell-2]
\]
Since $F_\beta$ is also spherical with cotwist $\mathsf{S}_\mathscr{B}[\ell]$ by Proposition \ref{prop:add}, it follows that $\mathsf{S}_\mathscr{B}^{-1}[\ell-2] \cong C_\beta'[-2]$. We also have that $F_\beta C_\beta' \cong T_\beta F_\beta[2]$ (see \cite[Section 1.3]{add}). Thus,
\[
F_\beta \mathsf{S}_\mathscr{B}^{-1}[\ell-2] \cong F_\beta C_\beta'[-2] = T_\beta' F_\beta.
\]
It follows that $FL_\mathscr{A}\beta \cong T_FT_\beta' F_\beta$. According to \cite[Theorem 11]{asp-add}, we have $T_F \cong T_\alpha T_\beta$. Hence $FL_\mathscr{A} \beta \cong T_\alpha F_\beta$.

The second statement in the proposition is proved similarly.
\end{proof}

\begin{corollary}
If $\mathscr{S}_K(\mathscr{C},\mathscr{A}_i,F_i)$ is the perverse schober obtained from the construction in Corollary \ref{cor:extw}, then $\sigma_j\mathscr{S}_K(\mathscr{C},\mathscr{A}_i,F_i)$ and $\mathscr{S}_K(\mathscr{C},\mathscr{A}_i,F_i)$ have quasiequivalent categories of global sections.
\end{corollary}
\begin{remark}
If $F_i : \mathscr{A}_i \rightarrow \mathscr{C}$ is an ordered collection of spherical functors with a weak right relative Calabi--Yau structure then the category $\mathscr{C}$ is Calabi--Yau and the cotwist of $F_i$ is the Serre functor on $\mathscr{A}_i$. Work of Katzarkov, Pandit and Spaide \cite{kps} shows that if $F_i$ have weak right Calabi--Yau structures, then the category of global sections $\mathscr{A} = \Gamma\mathscr{S}_K(\mathscr{C},\mathscr{A}_i,F_i)$ of the corresponding perverse schober also admits a functor $F:\mathscr{A} \rightarrow \mathscr{C}$ which has a weak right Calabi--Yau structure.  Therefore, we may apply Proposition \ref{prop:mutat} to see that Kapranov and Schechtman's action on a perverse schober constructed from functors with weak right relative Calabi--Yau structures leaves the category of global sections invariant.

\end{remark}

\section{Type II degenerations of Calabi--Yau varieties}\label{sect:typeII}

In this section, we draw a connection between certain perverse sheaves of categories on the punctured disc and categories of perfect complexes on type II degenerations of Calab-i-Yau varieties. This incorporates information from Section \ref{sect:recon} and \ref{sect:sod}.

\subsection{Sheaves of categories and blowing up}\label{sect:scatblow}

In the next section, we will let $V$ be a smooth projective variety, $S$ a smooth, effective anticanonical divisor in $V$ and $Z$ a smooth codimension one subvariety of $S$. We let $\widetilde{V}$ be the blow up of $V$ in $Z$. We would like to prove that there are perverse sheaves of categories which encode the data of $\widetilde{V}$. Since $Z$ is contained in a smooth anticanonical divisor $S$, the proper transform of $S$ is isomorphic to $S$ and is still anticanonical in $\widetilde{V}$. Let us take $j:Z\hookrightarrow S$ and let $k: S \hookrightarrow V$. We also let $\widetilde{k}$ be the embedding $\widetilde{S} \hookrightarrow \widetilde{V}$, so if $\pi:\widetilde{V} \rightarrow V$ is the blow up map, then $\pi\cdot \widetilde{k} = k$.
\begin{proposition}\label{prop:todiv}
There is a quasiequivalence of dg categories between $\DD^b_{\dg}(\widetilde{V})$  and
\[
\DD^b_{\dg}(V) \times_{\mathbb{L}k^*, \mathbb{R}j_*} \DD^b_{\dg}(Z).
\]
so that the quasifunctor $c_{\mathbb{L}k^*,\mathbb{R}j_*}$ to $\DD^b_{\dg}(S)$ is quasiisomorphic as a bimodule to $\mathbb{L}\widetilde{k}^*$. 
\end{proposition}
\begin{proof}
From \cite{orl-monoid}, we have a semiorthogonal decomposition,
\[
\DD^b(\widetilde{V}) = \langle \mathbf{L}\pi^*\DD^b(V), \mathbf{R}q_*\mathbf{L}p^*\DD^b(Z) \rangle,
\]
where $p : \mathbb{P}_Z(N_{Z/V}) \rightarrow Z$ is the obvious $\mathbb{P}^1$ fibration and $q : \mathbb{P}_Z(N_{Z/V}) \hookrightarrow \widetilde{V}$ is the induced closed embedding.

Therefore, by \cite{kl} there is a quasiequivalence between $\DD^b_{\dg}(\widetilde{V})$ and $\DD^b_{\dg}(V) \times_{\mathbb{L}\pi^*,\mathbb{R}p^*\mathbb{L}q_*} \DD^b_{\dg}(Z)$. We can use Corollary \ref{cor:obvs} to see that this category is quasiequivalent to
\[
\DD^b(V) \times_{\mathbb{L}\widetilde{k}^*\mathbb{L}\pi^*, \mathbb{L}\widetilde{k}^* \mathbb{R}q_* \mathbb{L}p^*} \DD^b_{\dg}(Z)
\]
since $S$ is anticanonical in $\widetilde{V}$. Then we have that $\pi \cdot \widetilde{k} = k$, so $\mathbb{L}\widetilde{k}^* \mathbb{L}\pi^* = \mathbb{L}k^*$. Furthermore, we have a cartesian diagram
\[
\begin{tikzcd}
Z \ar[r,"\sigma"] \ar[d,"j"]& \mathbb{P}_Z(N_{Z/V})\ar[d,"q"] \\
S \ar[r,"\widetilde{k}"]& \widetilde{V}
\end{tikzcd}
\]
In \cite[Corollary 2.26]{kuzhyp}, Kuznetsov constructs objects $K_{\widetilde{k},q}$ and $K^{j, \sigma}$ of $\DD^b(\mathbb{P}_Z(N_{Z/V}) \times S)$ whose corresponding Fourier-Mukai transforms from $\DD^b(\mathbb{P}_Z(N_{Z/V})) \rightarrow \DD^b(Z)$ represent $\mathbf{L}\widetilde{k}^*\mathbf{R}q_*$ and $\mathbf{R}j_* \mathbf{L}\sigma^*$ respectively. In {\it loc. cit.}, Kuznetsov shows that there is an isomorphism $K_{\widetilde{k},q} \rightarrow K^{j,\sigma}$. This implies that there is a quasiisomorphism between the dg functors $\mathbb{R}j_* \mathbb{L}\sigma^*$ and $\mathbb{L}\widetilde{k}^* \mathbb{R}q_*$ as dg bimodules. Therefore, since $p\cdot \sigma = \mathrm{id}_Z$, we have that $\DD^b_{\dg}(\widetilde{V})$ is quasiequivalent to
\[
\DD^b_{\dg}(V)\times_{\mathbb{L}k^*, \mathbb{R}j_*}\DD^b_{\dg}(Z)
\]
as expected. By construction and Propositions \ref{ref:klprop2} and \ref{ref:klprop3}, the induced quasifunctor to $\DD^b_{\dg}(S)$ is quasiisomorphic to $\mathbb{L}\widetilde{k}^*$ as a bimodule.
\end{proof}
This means that the process of blowing up along a smooth anticanonical divisor not only can be recast as categorical gluing, but it can be recast as taking the category of global sections of a perverse schober, since the functors $\mathbb{L}k^*$ and $\mathbb{R}j_*$ are spherical. 
\begin{remark}
We may apply this proposition recursively. Let $Z_1,\dots , Z_n$ be smooth codimension 1 subvarieties in $S$, $j_i: Z_i \hookrightarrow V$ the corresponding closed embeddings, and let $\mathrm{Bl}_{Z_1,\dots, Z_n}(V)$ denote the blow up of $V$ in the proper image of $Z_1$, followed by the blow up in the proper image of $Z_2$ etc.. Then there is a perverse schober on $K_{n+1}$ associated to the spherical functors $\mathbb{L}k^*:\DD^b_{\dg}(V) \rightarrow \DD^b_{\dg}(S)$ and $\mathbb{R}j_i^* : \DD^b_{\dg}(Z_i) \rightarrow \DD^b_{\dg}(S)$ whose category of global sections is quasiequivalent to $\DD^b_{\dg}(\mathrm{Bl}_{Z_1,\dots, Z_n}(V))$.

\end{remark}
Now we note that this construction can be simplified further in the case where $S$ decomposes into a pair of disjoint smooth connected components one of which contains $Z$. The prototype of this situation is when $V = \mathbb{P}^1 \times S$, and $S$ is a Calabi--Yau manifold. In this case $(\{0 \} \times S) \sqcup (\{ \infty\} \times S)$ is a smooth, effective anticanonical divisor in $V$ composed of two disjoint subvarieties.
\begin{proposition}\label{prop:twosections}
Let $S_1$ and $S_2$ be a pair of smooth disjoint divisors in a variety $V$ so that $S_1 \sqcup S_2$ is the vanishing locus of a section of $-K_V$, and let $k_\ell : S_\ell \hookrightarrow V$ be the corresponding closed embeddings. Let 
$Z$ be a smooth, effective divisor in $S_1$ with closed embedding map $i$, and let $\widetilde{V} = \mathrm{Bl}_Z(V)$. Then there is a quasiequivalence between $\DD^b_{\dg}(\widetilde{V})$ and $\DD^b_{\dg}(V) \times_\mathsf{S}\DD^b_{\dg}(Z)$ where
\[
\mathsf{S}(b,a) = \hom_{\mathrm{mod}_{\DD^b_{\dg}(S_1)}}(\mathbb{L}k_1^*b, \mathbb{R}i_*a).
\]
Furthermore, the functor $\DD^b_{\dg}(V) \times_\mathsf{S} \DD^b_{\dg}(Z) \rightarrow \Tw \mathrm{mod}_{\DD^b_{\dg}(S_1)}$ sending
\[
(a,b,\mu) \mapsto \cone(\mu)
\]
is quasi-isomorphic to $\mathbb{L}\widetilde{k}_1^*$ and $(a,b,\mu) \mapsto \mathbb{L}{k}_2^*b$ is equivalent to $\mathbb{L}\widetilde{k}_2^*$, where $\widetilde{k}_\ell : S_\ell \hookrightarrow \widetilde{V}$ are the obvious embeddings.
\end{proposition}
\begin{proof}
The first part follows from Proposition \ref{prop:todiv}; we have that $\DD^b_{\dg}(\widetilde{V})$ is quasiequivalent to $\DD^b_{\dg}(V) \times_{\mathsf{S}'}\DD^b_{\dg}(Z)$ where
\[
\mathsf{S}'(a,b) = \hom_{\DD^b_{\dg}(S_1) \times \DD^b_{\dg}(S_2)}(\mathbb{L}k_1^*(b) \times \mathbb{L}k_2^*(b), \mathbb{R}i_*(a) \times 0),
\]
which is naturally isomorphic to $\mathsf{S}$. Since $Z \subset S_1$ is smooth of codimension 1, the proper transforms of $S_1$ and $S_2$, which are isomorphic to $S_1$ and $S_2$, are anticanonical in $V$. According to Corollary \ref{cor:obvs}, the functor
\[
\mathbb{L}\widetilde{k}_1^* \times \mathbb{L}\widetilde{k}_2^* : \DD^b_{\dg}(\widetilde{V}) \longrightarrow  \DD^b_{\dg}(S_1) \times \DD^b_{\dg}(S_2)
\]
sends $(a,b,\mu')$ for $\mu' \in \mathsf{S}'(b,a)$ to $\cone(\mu')$. Since $\mu'$ can be written as $\mu \times 0$ for $\mu \in \mathsf{S}(b,a)$, it follows then from Propositions \ref{ref:klprop2} and \ref{ref:klprop3} that $\mathbb{L}\widetilde{k}_1^*$ is quasiisomorphic to the functor sending $(a,b,\mu')$ to $\cone(\mu)$ and $\mathbb{L}\widetilde{k}_2^*$ is quasiisomorphic to the functor sending $(a,b,\mu')$ to $\mathbb{L}k_2^*a$.
\end{proof}

We can represent this situation as the diagram of categories which commutes up to homotopy;
\[
\begin{tikzcd}
\DD^b_{\dg}(\widetilde{V}) \ar[rd, "Q"] \ar[rrrd,bend left=15, "\mathbb{L}\widetilde{k}_2^*"] \ar[dddr, bend right, "\mathbb{L}\widetilde{k}_1^*"] & & & \\
& \DD^b_{\dg}(V) \times_{\mathbb{L}k^*,\mathbb{R}i_*} \DD^b_{\dg}(Z) \ar[r] \ar[d] & \DD^b_{\dg}(V) \times \DD^b_{\dg}(Z) \ar[d,"\mathbb{L}k_1^* \times \mathbb{R}i_*"] \ar[r,"\mathbb{L}k_2^* \times 0"]& \DD^b_{\dg}(S_2) \\
& A_2(\DD^b_{\dg}(S_1)) \ar[r,"f_1\times f_2"] \ar[d,"f_3"] & \DD^b_{\dg}(S_1) \times \DD^b_{\dg}(S_2) & \\
& \DD^b_{\dg}(S_1) & & 
\end{tikzcd}
\]
Where the square is homotopy cartesian, and arrows should be interpreted as quasifunctors and $Q$ is a quasifunctor which induces an equivalence in the homotopy category.

\subsection{Sections of $\mathbb{P}^1$ bundles}

Recall that in Section \ref{sect:Pbub}, when we constructed $\Gamma\mathscr{S}_\phi(\mathscr{C},\Phi)$, we use $f_1$ and $f_2$ to glue $A_2(\mathscr{C})$ to $A_2(\mathscr{C})$. Thus we are left with a pair of functors $g_1$ and $g_2$ inherited from $f_3$ which is depicted in the following diagram.
\[
\begin{tikzcd}
\mathscr{C} \times_{\mathsf{S}_\Phi} \mathscr{C} \ar [rd,"\widetilde{Q}"] &  & & \\
& \Gamma\mathscr{S}_\phi(\mathscr{C},\Phi) \ar[r,"m"] \ar[d,"n"]  \ar[rr, bend left , "g_2"] \ar[dd, bend right=60, "g_1"]& A_2(\mathscr{C})\ar[r,"f_3"] \ar[d,"f_1 \times \Phi\cdot f_2"]  & \Tw \mathscr{C} & \\
&A_2(\mathscr{C}) \ar[r,"f_1 \times f_2"] \ar[d,"f_3"] & \Tw \mathscr{C} \times \Tw \mathscr{C} &  \\
&\Tw \mathscr{C} && 
\end{tikzcd}
\]
so that $g_1 = f_3 \cdot n$ and $g_2 = f_3 \cdot m$. The functor $\widetilde{Q}$ is a quasiequivalence. The two copies of $\mathscr{C}$ in $\mathscr{C}\times_{\mathsf{S}_\Phi} \mathscr{C}$ map to $\mathscr{C}_1'$ and $\mathscr{C}_2'$ as defined before Proposition \ref{prop:ncP1}. It is easy to see how $g_1$ and $g_2$ behave on the level of the categories $\mathscr{C}'_1$ and $\mathscr{C}'_2$.

On the level of objects, these maps send
\[
g_1, g_2 : ((a,0,0) , (a,0,0), \mathrm{id}_a \times 0)) \mapsto a[1],a[1] , \quad g_1, g_2 : ((0,b,0) , (0,b,0), 0 \times \mathrm{id}_{b})) \mapsto \Phi(b),b.
\]
The behaviour with respect to homomorphisms is a little bit more subtle. We know that if $A,A' \in \mathscr{C}_1'$ are written as $((a,0,0),(a,0,0), \mathrm{id}_a\times 0)$ and $((a',0,0),(a',0,0),\mathrm{id}_{a'}\times 0)$ then 
\[
\hom_{\Gamma \mathscr{S}_\phi(\mathscr{C},\Phi)}(A,A') = \cone(\hom_\mathscr{C}(a,a')^2 \xrightarrow{(\mathrm{id}, -\mathrm{id})^T} \hom_\mathscr{C}(a,a'))
\]
The complex $\hom_{\mathscr{C}_1'}(A,A')$ is the diagonal subcomplex of $\hom_\mathscr{C}(a,a')^2$. One can check directly that $g_1$ and $g_2$ act on this subcomplex as the identity, sending $(f,f)$ to $f$. A similar statement is true for $\mathscr{C}_2'$.

Now we know that if $A$ is as above and $B \in \mathscr{C}_2$ is written as $((0,b,0),(0,b,0), 0 \times \mathrm{id}_b)$, then by the proof of Proposition \ref{prop:ncP1}, $\hom_{\Gamma\mathscr{S}(\mathscr{C},\Phi)}(A,B)$ is equal to 
\[
\hom_\mathscr{C}(a,b\oplus \Phi(b)) = \hom_\mathscr{C}(a,b) \oplus \hom_\mathscr{C}(a,\Phi(b)).
\]
The functor $g_1$ induces projection onto the first component, and $g_2$ induces projection onto the second component. Therefore, $f_3\cdot m \cdot \widetilde Q$ and $f_3 \cdot n \cdot \widetilde Q$ are the functors sending $(a,b, \mu \oplus \xi)$ to 
\[
\cone(\mu: a \rightarrow b) , \qquad \qquad \cone(\xi : a \rightarrow \Phi(b)).
\]
In the case considered in Theorem \ref{thm:orlov}, we can ask whether $g_1$ and $g_2$ have any geometric significance. There are two filtrations of $\mathscr{E} = \mathscr{O}_S \oplus \mathscr{L}$, by sub-bundles coming from the obvious embedding of $\mathscr{O}_S$ and $\mathscr{L}$ into $\mathscr{E}$, thus there are two sections of $\mathbb{P}_S(\mathscr{E})$ which we call $S_\mathscr{O}$ and $S_\mathscr{L}$, and we let $\sigma_\mathscr{O}$ and $\sigma_\mathscr{L}$ be their embeddings into $\mathbb{P}_S(\mathscr{E})$.
\begin{proposition}\label{prop:bundleblow}
Assume that $S$ is compact, smooth and Calabi--Yau. The functors $\mathbb{L}\sigma_\mathscr{O}^*$ and $g_1$ are quasiisomorphic as bimodules, as are $\mathbb{L}\sigma_\mathscr{L}^*$ and $g_2$.
\end{proposition}
\begin{proof}
First, we remark that because of the conditions that we have placed on $S$, $S_\mathscr{O} \cup S_\mathscr{L}$ is a smooth anticanonical divisor on $\mathbb{P}_S(\mathscr{E})$. Let us denote by $p : \mathbb{P}_S(\mathscr{E}) \rightarrow S$ the natural projection map. We have a semiorthogonal decomposition,
\[
\DD^b(\mathbb{P}_S(\mathscr{O}_S \oplus \mathscr{L})) = \langle \mathbb{L}p^*\DD^b(S), \mathbb{L}p^* \DD^b(S) \otimes^\mathbb{L} \mathscr{O}_S(1) \rangle.
\]
Therefore, by the discussion preceding Corollary \ref{cor:obvs}, we get that $\DD^b_{\dg}(\mathbb{P}_S(\mathscr{O}\oplus \mathscr{L}))$ is quasiequivalent to
\[
\DD^b_\mathrm{dg}(S_\mathscr{L}) \times_{\mathsf{S}_\mathbb{P}}\DD^b_\mathrm{dg}(S_\mathscr{O})
\]
where 
\[
\mathsf{S}_\mathbb{P}(b,a) = \hom_{\DD^b_\mathrm{dg}(\mathbb{P}_S(\mathscr{E}))}(\mathbb{L}p^*a, \mathbb{L}p^*b \otimes^\mathbb{L}\mathscr{O}_{\mathbb{P}}(1)).
\]
According to Corollary \ref{cor:obvs}, this bimodule is isomorphic to
\begin{align*}
\mathsf{S}_{S_\mathscr{O}\cup S_\mathscr{L}}(a,b) = &\hom_{\DD^b_\mathrm{dg}(S_\mathscr{O})}(\mathbb{L}\sigma_\mathscr{O}^*\mathbb{L}p^*(b), \mathbb{L}\sigma_\mathscr{O}^*(\mathbb{L}p^*(a) \otimes^\mathbb{L} \mathscr{O}_\mathbb{P}(1)))\\
\oplus &\hom_{\DD^b_\mathrm{dg}(S_\mathscr{L})}(\mathbb{L}\sigma_\mathscr{L}^*\mathbb{L}p^*(b), \mathbb{L}\sigma_\mathscr{L}^*(\mathbb{L}p^*(a) \otimes^\mathbb{L} \mathscr{O}_\mathbb{P}(1))).
\end{align*}
Since $p\cdot \sigma_\mathscr{O} = p \cdot \sigma_\mathscr{L} = \mathrm{id}_S$ and $\mathbb{L}\sigma_\mathscr{O}^*\mathscr{O}_\mathbb{P}(1) \cong \mathscr{O}_S$ and $\mathbb{L}\sigma_\mathscr{O}^*\mathscr{O}_\mathbb{P}(1) \cong \mathscr{L}$ by \cite[V, Proposition 2.6]{hartshorne}, it follows that this bimodule is isomorphic to
\[
\mathsf{S}'_{S_\mathscr{O}\cup S_\mathscr{L}}(a,b) = \hom_{\DD^b_\mathrm{dg}(S)}(b,a) \oplus \hom_{\DD^b_\mathrm{dg}(S)}(b,a\otimes^\mathbb{L} \mathscr{L}). 
\]
Using Corollary \ref{cor:obvs}, the functor $\mathbb{L}\sigma_\mathscr{O}^* \times \mathbb{L}\sigma_\mathscr{L}^*$ is quasiisomorphic as a bimodule to the functor sending $(a,b, \mu \oplus \xi)$ to $\cone(\mu) \times \cone(\xi)$. This proves the result.
\end{proof}

\subsection{Gluing components together}

There are certain geometric situations in which one obtains normal crossings varieties whose components are either varieties with smooth anticanonical divisors or which are $\mathbb{P}^1$ bundles over a Calabi--Yau variety. Notably, these appear as particularly nice degenerations of Calabi--Yau varieties. We will now introduce a particularly nice class of examples of such degenerations.
\begin{defn}[A. Tyurin \cite{tyur}, N.-H. Lee\cite{lee}]
A \emph{Tyurin degeneration} of Calabi--Yau varieties is a smooth projective family $g:\mathscr{X} \rightarrow \Delta$ over an analytic disc $\Delta$ centered at 0 with variable $t$ which satisfies the following criteria.
\begin{enumerate}
\item The fiber over $t \neq 0 \in \Delta$ is a smooth projective Calabi--Yau manifold.\footnote{We follow the convention that $X$ is Calabi--Yau if its canonical bundle is trivial and $h^{i,0}(X) = 0$ for $0 < i < \dim X$.} 
\item The fiber over $0$ is normal crossings and consists of a pair of divisors $Y_0$ and $Y_1$ so that $Y_0 \cap Y_1 = Z$ is an anticanonical subvariety in $Y_0$ and $Y_1$.
\end{enumerate}
\end{defn}
It goes back to work of Friedman \cite{fried} that this definition implies that $N_{Z/Y_0} \otimes N_{Z/Y_1} \cong \mathscr{O}_Z$. We let $\mathscr{X}'$ be the $n$-fold cover of $\mathscr{X}$ totally ramified along $g^{-1}(0)$. Then $\mathscr{X}'$ is no longer smooth but has a $cA_n$ singularity along the preimage $Z'$ of $Z =Y_1 \cap Y_2$ in $\mathscr{X}'$. We may resolve $\mathscr{X}'$ by blowing up $(n-1)$ times at $Z$. The resulting variety has $n+1$ components, $Y'_0,\dots, Y'_n$ so that $Y'_0 = Y_0$ and $Y_n' = Y_1$, and $Y_1',\dots, Y_{n-1}' \cong \mathbb{P}_Z(\mathscr{O}_Z \oplus N_{Z/Y_0} )$. If we let $Z_i (\cong Z) = Y_i' \cap Y_{i+1}'$ then the fact that this is the fiber over 0 of a smooth family of varieties means that $N_{Z_i/Y_i'} \otimes N_{Z_i/Y_{i+1}'} \cong \mathscr{O}_{Z_i}$. 
\begin{defn}\label{def:simptII}
Let $Y= Y_0\cup \dots \cup Y_n$ be a normal crossings union of smooth varieties. Then $Y$ is said to be  a \emph{simplified type II degeneration of Calabi--Yau varieties} if
\begin{enumerate}
\item $Y_i \cap Y_j = 0$ if and only if $|i-j| \geq 1$. 
\item If $Z_i = Y_{i-1} \cap Y_{i}$ then for each $i \neq 0, n$, $Y_i=\mathbb{P}_{Z_i}(\mathscr{O}\oplus N_{Z_i/Y})$. 
\item For each $i \neq 0, n$, the images of $Z_i$ and $Z_{i+1}$ in $Y_i$ are disjoint sections of the projection map onto $Z_i$.
\item For each $i$,
\[
N_{Z_i/Y_i} \otimes N_{Z_i/Y_{i-1}} \cong \mathscr{O}_{Z_i}.
\]
This means that $Y$ is d-semistable in the language of Friedman \cite{fried}, or log smooth in the language of Kawamata and Namikawa \cite{kn}.
\end{enumerate}
\end{defn}
\begin{remark}
The name ``type II degeneration'' is motivated by Kulikov's classification of degenerations of K3 surfaces  \cite{kulikov} (see also \cite{perspink}). The word ``simplified'' is used because a proper generalization of Kulikov's definition would require only that $Y_i$ for $i \neq 0,  n$ be crepant birational to $\mathbb{P}_{Z_i}(\mathscr{O} \oplus N_{Z_i/Y_i})$.
\end{remark}
\begin{remark}
It is clear from the definition that for all $i, j$, the varieties $Z_i$ and $Z_j$ are biregular.
\end{remark}

Calling this a degeneration of Calabi--Yau varieties is justified by a result of Kawamata and Namikawa \cite{kn} which says that such varieties can be smoothed to Calabi--Yau varieties. 

According to \cite[Proposition 4.7]{stz} if $C$ is a normal crossings curve composed of two components $C_1$ and $C_2$, and $C_1 \cap C_2$ is a single point, then $\mathscr{P}\mathrm{erf}(C)$ is quasiequivalent to the homotopy fiber product of $\mathrm{D}^b_{\dg}(C_1)$ and $\mathrm{D}^b_{\dg}(C_2)$ along the derived pullback functors $\mathbf{L}f^* : \mathrm{D}^b_{\dg}(C_i) \rightarrow \mathrm{D}^b_{\dg}(C_1 \cap C_2)$. Since simplified type II degenerations are normal crossings, and their singular locus locally looks like a product of $C_1 \cup C_2$ and affine space, \cite[Proposition 4.7]{stz} generalizes easily to imply that there is a diagram of categories whose homotopy limit is quasiequivalent to $\Perf(Y)$
\[
\begin{tikzcd}
\DD^b_{\dg}(Y_0) \ar[rd] & \DD^b_{\dg}(Y_1) \ar[rd] \ar[d] & \dots & \DD^b_{\dg}(Y_{n-1}) \ar[d] \ar[rd] & \DD^b_{\dg}(Y_n) \ar[d] \\
& \DD^b_{\dg}(Z_1) & \DD^b_{\dg}(Z_2) &  \DD^b_{\dg}(Z_{n-1}) & \DD^b_{\dg}(Z_{n}) & 
\end{tikzcd}
\]
The arrows in this diagram denote the pullbacks to $Z_{i}$ or $Z_{i+1}$ in $Y_i$. We will often drop the indices and use the notation $Z$ to refer to $Z_i$. We may use Theorem \ref{thm:orlov} and Proposition \ref{prop:bundleblow} to replace $\DD^b_{\dg}(Y_i)$ for $i \neq 0,n$ with $\Gamma \mathscr{S}_{K_\phi}(\DD^b_{\dg}(Z_i), (-)\otimes^\mathbb{L} N_{Z_i/Y_i})$, and the restriction functors by $g_1$ and $g_2$ respectively. Denote this category $\Gamma \mathscr{S}_{K_\phi}$. Using Corollary \ref{cor:extw}, if $\DD^b_{\dg}(Y_0)$ and $\DD^b_{\dg}(Y_n)$ have semiorthogonal decompositions $\mathscr{A}_1,\dots, \mathscr{A}_{k_1}$ and $\mathscr{B}_1,\dots, \mathscr{B}_{k_2}$ and $F_i : \mathscr{A}_i \rightarrow \DD^b_{\dg}(Z_0)$ and $G_j : \mathscr{B}_j \rightarrow \DD^b_{\dg}(Z_{n-1})$ are the spherical functors obtained by composing the embeddings of $\mathscr{A}_i$ into $\DD^b_{\dg}(Y_0)$ or of $\mathscr{B}_j$ into $\DD^b_{\dg}(Y_n)$ with the pullback map to $Z$, then we may replace $\DD^b_{\dg}(Y_1)$ with $\Gamma \mathscr{S}_{K_{k_1}}(\mathscr{A}_i,F_i)$ and the pullback to $\DD^b_{\dg}(Z_0)$ by $s_\infty$, and similarly we can replace $\DD^b_{\dg}(Y_n)$ with $\Gamma\mathscr{S}_{K_{k_2}}(G_i,\mathscr{B}_i)$ and the pullback to $\DD^b_{\dg}(Z_{n-1})$ by $s_\infty$ (see Section \ref{sect:gsections} for notation). We denote these categories $\Gamma \mathscr{S}_{K_{k_1}}$ and $\Gamma \mathscr{S}_{K_{k_2}}$ respectively. Therefore, we have the following diagram of categories whose homotopy limit is quasiequivalent to $\Perf(Y)$.
\[
\begin{tikzcd}
\Gamma\mathscr{S}_{k_1} \ar[rd,"s_\infty"] & \Gamma\mathscr{S}_{K_\phi} \ar[rd,"g_2"] \ar[d, "g_1"] & \dots & \Gamma \mathscr{S}_{K_\phi} \ar[d,"g_1"] \ar[rd,"g_2"] & \Gamma\mathscr{S}_{K_{k_2}} \ar[d,"s_\infty"] \\
& \DD^b_{\dg}(Z) & \DD^b_{\dg}(Z) &  \DD^b_{\dg}(Z) & \DD^b_{\dg}(Z) & 
\end{tikzcd}
\]
We may then reinterpret this in the following way.
\begin{theorem}\label{thm:simpII}
To a simplified type II degeneration of Calabi--Yau varieties $Y = Y_0\cup \dots \cup Y_n$ and a choice of semiorthogonal decompositions of $Y_0$ and $Y_n$ there is a perverse sheaf of categories on $S^2$ with $n$ boundary components whose category of global sections is $\Perf(Y)$.
\end{theorem}

This sheaf of categories is constructed by first gluing $(n-1)$ copies of $K_\phi$ together along their exterior edges to get a graph $K_{(n-1)\phi}$. This graph is naturally a spanning graph on $S^2$ with $n$ boundary components and we may equip it with a $K_{(n-1)\phi}$-coordinatized perverse sheaf of categories using precisely the same data as was used to describe $\mathscr{S}_{K_\phi}(\DD^b_{\dg}(Z), (-)\otimes^\mathbb{L} N_{Z/Y_i})$. This is given as follows. $K_{(n-1)\phi}$ has $2(n-1)$ vertices called $v_{1,j},v_{2,j}$ for $j =1,\dots (n-1)$ and $3(n-1) + 1$ edges which we denote $g_1,\dots, g_n$ and $f_{1,j}, f_{2,j}$ for $j = 1,\dots, (n-1)$. For $i=2,\dots, n-1$ let $g_i$ be adjacent to both $v_{2,i}$ and $v_{1,i+1}$, and let $g_1$ be adjacent to $v_{1,1}$ with its other end in a boundary component and $g_n$ adjacent to $v_{2,n}$ with its other end in another boundary component. Then $f_{1,i}$ and $f_{2,i}$ are adjacent to both $v_{1,i}$ and $v_{2,i}$. All of the vertices $v_{j,\ell}$ are trivalent, so we assign to them the category $A_2(\DD^b_{\dg}(Z))$. We have functors
\begin{align*}
F_{v_{1,1},g_1}  = F_{v_{2,n},g_n}& = F_{v_{2,i},g_i} = F_{v_{i+1,1},g_i} = f_3, \quad F_{v_{1,j}, f_{1,j}} = N_{Z/Y_i} \otimes^\mathbb{L} f_2, \\
&F_{v_{2,j},f_{2,j}} = F_{v_{1,j},f_{2,j}} = f_1, \quad F_{v_{2,j}, f_{1,j}} = f_2.
\end{align*}
The category of global sections of $\mathscr{S}_{K_{(n-1)\phi}}$ is equivalent to $\Perf(Y_1 \cup \dots \cup Y_{n-1})$. To $\mathscr{S}_{K_{k_1}}(\mathscr{A}_i,F_i)$ and $\mathscr{S}_{K_{k_2}}(\mathscr{B}_i,F_i)$, there are graphs with vertices $c, u_1,\dots, u_{k_1}$ and $d, w_{1}, \dots, w_{k_2}$ where $c$ is $(k_1+1)$-valent, $d$ is $(k_2+1)$-valent and $u_i,w_i$ are univalent. There are edges $r_1,\dots,r_{k_1}, r_\infty$ and $q_1,\dots, q_{k_2},q_\infty$ so that $r_i$ connects $c$ to $u_i$ if $i \neq \infty$ and $q_i$ connects $d$ to $w_i$ if $i \neq \infty$. We have that $\mathscr{A}_c = A_{k_1}(\DD^b_{\dg}(Z))$ and $\mathscr{A}_d = A_{k_2}(\DD^b_{\dg}(Z))$, $\mathscr{A}_{u_i} = \mathscr{A}_i$ and $\mathscr{A}_{w_i} = \mathscr{B}_i$ and $F_{u_i,r_i} = F_i$ and $F_{w_i,q_i} = G_i$. Finally, $F_{c,r_i} = f_i$ and $F_{d,q_i} = f_i$ if $i \neq \infty$, and $F_{c,r_\infty} = f_{k_1 + 1}$ and $F_{d,q_\infty} = f_{k_2 + 1}$.

Attach $r_\infty$ to $g_1$ and $q_\infty$ to $g_n$ to obtain the graph $K_Y$. This comes equipped with a $K_Y$-coordinatized perverse sheaf of categories whose category of global sections is quasiequivalent to $\Perf(Y)$ by the argument given above.

\begin{figure}
    \begin{tikzpicture}[scale=0.8]

\draw[very thick] (3.86,-1) to (5.86,-1);
\draw[very thick] (3.86,1) to (1.86,1);

\draw[very thick] (1.86,-1) to (-0.14,-1);

	\draw[very thick] (5,-2) arc (270:450:1cm and 2cm);
    \draw[dashed,very thick] (5,-2) arc (270:90:1cm and 2cm);

	\draw[very thick] (1,-2) arc (270:450:1cm and 2cm);
    \draw[dashed,very thick] (1,-2) arc (270:90:1cm and 2cm);

\draw[very thick] (3,-2) arc (270:450:1cm and 2cm);
\draw[dashed, very thick] (3,-2) arc (270:90:1cm and 2cm);

\draw[very thick] (-6,-2) arc (270:450:1cm and 2cm);
\draw[dashed, very thick] (-6,-2) arc (270:90:1cm and 2cm);

\draw[very thick] (-4,-2) arc (270:450:1cm and 2cm);
\draw[dashed, very thick] (-4,-2) arc (270:90:1cm and 2cm);

	\draw[very thick] (-8,0) ellipse (1cm and 2cm);

	\draw[dashed] (-2,0) ellipse (1cm and 2cm);
\draw[dashed] (-1,0) ellipse (1cm and 2cm);

\draw[very thick] (-8,2) to (-2,2);
\draw[very thick] (-8,-2) to (-2,-2);

\draw[decorate,decoration={brace,amplitude=10pt}, very thick]
(3,-2.2) -- (-6,-2.2);
\node at (-1.5,-3){$(n-1)$\text{-times}};

\draw[very thick] (3,0) circle (5pt);
\draw[very thick] (1,0) circle (5pt);
\draw[very thick] (-2,0) circle (5pt);
\draw[very thick] (-4,0) circle (5pt);

\draw[very thick] (5,2) to (-1,2);
\draw[very thick] (5,-2) to (-1,-2);

\draw[fill = black] (1.86,-1) circle (2pt);
\draw[fill = black] (1.86,1) circle (2pt);

\draw[fill = black] (3.86,-1) circle (2pt);
\draw[fill = black] (3.86,1) circle (2pt);

\draw[fill = black] (-3.14,-1) circle (2pt);
\draw[fill = black] (-3.14,1) circle (2pt);
\draw[fill = black] (-5.14,-1) circle (2pt);
\draw[fill = black] (-5.14,1) circle (2pt);

\draw[very thick] (-5.14,-1) to (-3.14,-1);
\draw[very thick] (-3.14,1) to (-1.14,1);
\draw[very thick] (-5.14,1) to (-7.14,1);

\end{tikzpicture}\caption{\label{fig:nSphi}The skeleton $K_{(n-1)\phi}$.}
\end{figure}
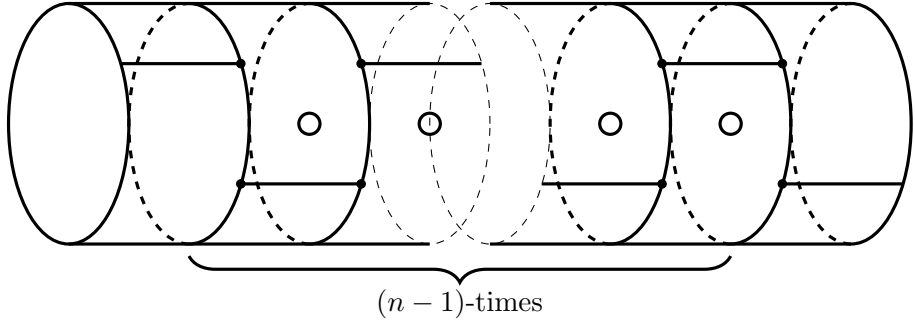

\begin{figure}
    \begin{tikzpicture}[scale=0.8]

\draw[fill = black] (5.3,-1) circle (2pt);
\draw[very thick] (5.3,-1) .. controls (5.4,-0.1).. (4.9,1);
\draw[very thick] (5.3,-1) .. controls (5.7,-0.1).. (5.7,1.3);
\draw[very thick] (5.3,-1) .. controls (6,-0.5).. (6.5,0.5);
\draw[very thick] (5.3,-1) .. controls (5.6,-1.3).. (6.2,-1.2);
\draw[very thick] (5.3,-1) .. controls (5.3,-1.3).. (5.1,-1.7);

\draw[fill = black] (5.1,-1.7) circle (2pt);
\draw[fill = black] (6.2,-1.2) circle (2pt);
\draw[fill = black] (6.5,0.5) circle (2pt);
\draw[fill = black] (5.7,1.3) circle (2pt);
\draw[fill = black] (4.9,1) circle (2pt);

\draw[very thick] (3.86,-1) to (5.3,-1);
\draw[very thick] (3.86,1) to (1.86,1);

\draw[very thick] (1.86,-1) to (-0.14,-1);

	\draw[very thick] (5,-2) arc (270:450:2cm and 2cm);

	\draw[very thick] (1,-2) arc (270:450:1cm and 2cm);
    \draw[dashed,very thick] (1,-2) arc (270:90:1cm and 2cm);

\draw[very thick] (3,-2) arc (270:450:1cm and 2cm);
\draw[dashed, very thick] (3,-2) arc (270:90:1cm and 2cm);

\draw[very thick] (-6,-2) arc (270:450:1cm and 2cm);
\draw[dashed, very thick] (-6,-2) arc (270:90:1cm and 2cm);

\draw[very thick] (-4,-2) arc (270:450:1cm and 2cm);
\draw[dashed, very thick] (-4,-2) arc (270:90:1cm and 2cm);

	\draw[very thick] (-8,-2) arc (270:90:2cm and 2cm);

	\draw[dashed] (-2,0) ellipse (1cm and 2cm);
\draw[dashed] (-1,0) ellipse (1cm and 2cm);

\draw[very thick] (-8,2) to (-2,2);
\draw[very thick] (-8,-2) to (-2,-2);

\draw[decorate,decoration={brace,amplitude=10pt}, very thick]
(3,-2.2) -- (-6,-2.2);
\node at (-1.5,-3){$(n-1)$\text{-times}};

\draw[very thick] (3,0) circle (5pt);
\draw[very thick] (1,0) circle (5pt);
\draw[very thick] (-2,0) circle (5pt);
\draw[very thick] (-4,0) circle (5pt);

\draw[very thick] (-9.7,0) circle (5pt);

\draw[very thick,dashed] (6.7,0) circle (5pt);

\draw[very thick] (5,2) to (-1,2);
\draw[very thick] (5,-2) to (-1,-2);

\draw[fill = black] (1.86,-1) circle (2pt);
\draw[fill = black] (1.86,1) circle (2pt);

\draw[fill = black] (3.86,-1) circle (2pt);
\draw[fill = black] (3.86,1) circle (2pt);

\draw[fill = black] (-3.14,-1) circle (2pt);
\draw[fill = black] (-3.14,1) circle (2pt);
\draw[fill = black] (-5.14,-1) circle (2pt);
\draw[fill = black] (-5.14,1) circle (2pt);

\draw[very thick] (-5.14,-1) to (-3.14,-1);
\draw[very thick] (-3.14,1) to (-1.14,1);
\draw[very thick] (-5.14,1) to (-7.14,1);

\draw[fill = black] (-7.14,1) circle (2pt);

\draw[very thick] (-7.14,1).. controls (-7.5,1.5) .. (-8,1.8);
\draw[very thick] (-7.14,1).. controls (-8,1.2) .. (-9,1.1);
\draw[very thick] (-7.14,1).. controls (-8,0.5) .. (-9.2,0.2);
\draw[very thick] (-7.14,1).. controls (-8,-1.2) .. (-9,-1);
\draw[very thick] (-7.14,1).. controls (-7,-0.5) .. (-7.14,-1.5);

\draw[fill=black] (-8,1.8) circle (2pt);
\draw[fill=black] (-9,1.1) circle (2pt);
\draw[fill=black] (-9.2,0.2) circle (2pt);
\draw[fill=black] (-9,-1) circle (2pt);
\draw[fill=black] (-7.14,-1.5) circle (2pt);

\end{tikzpicture}\caption{\label{fig:nY} The skeleton $K_{Y}$.}
\end{figure}
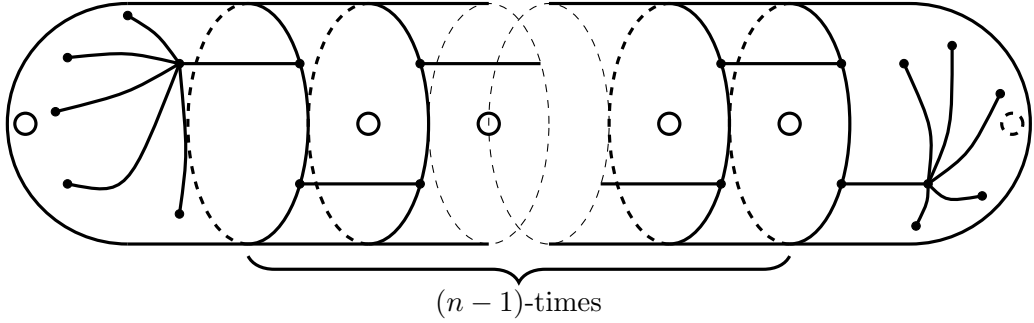

We may note that, until now, the d-semistability condition, that $N_{Z_i/Y_i} \otimes N_{Z_i/Y_{i+1}} = \mathscr{O}_{Z_i}$ played no role; indeed Theorem \ref{thm:simpII} carries through for any chain of varieties satisfying (1) and (2) of Definition \ref{def:simptII}. The d-semistability condition translates in terms of perverse sheaves of categories the the following statement. 

To each boundary component of $S$, we have a small loop going around it. Therefore, for each boundary component, there is a cycle in $K^\circ$ (see Section \ref{sect:gpsc} for definition) which descends to that loop in $\pi_1(S^\circ, v)$ for some vertex of $K^\circ$.

\begin{theorem}\label{thm:mono}
If $Y$ is a normal crossings variety satisfying (1) and (2) of Definition \ref{def:simptII}, then the perverse sheaf of categories constructed in Theorem \ref{thm:simpII} has trivial monodromy around each boundary component of $S$ if and only if $Y$ satisfies (3), which is to say that it can be deformed to a smooth Calabi--Yau variety by \cite{kn}.
\end{theorem}
Before proceeding to the proof of this result, we will record a standard result. A section of a projective bundle $\mathbb{P}_X(\mathscr{E})$ is determined by a surjective map of vector bundles
\[
\mathscr{E} \longrightarrow \mathscr{L}
\]
where $\mathscr{L}$ is a line bundle.
\begin{proposition}
If $\mathscr{E}$ is a rank 2 vector bundle on a smooth projective variety $X$, and if $D$ is a section of  $\mathbb{P}_X(\mathscr{E})$ determined by a short exact sequence of vector bundles
\[
0 \longrightarrow \mathscr{L}_1 \longrightarrow \mathscr{E} \longrightarrow\mathscr{L}_2 \longrightarrow 0
\]
then $N_{D/\mathbb{P}_X(\mathscr{E})} \cong \mathscr{L}_2 \otimes \mathscr{L}^{-1}_1$.
\end{proposition}
Therefore, if $D_1$ and $D_2$ are two sections of $P = \mathbb{P}_Z(\mathscr{O} \oplus \mathscr{L})$ then $N_{D_1/P} = N_{D_2/P}^{-1}$. Now we proceed with the proof of Theorem \ref{thm:mono}.
\begin{proof}[Proof of Theorem \ref{thm:mono}]
There are three situations that we must consider. First, we look at the easiest case which is that of the interior punctures, that is, punctures $1,\dots, n-2$ in Figure \ref{fig:nY}. In this case, we have that our loop is made up of four vertices and five edges. We exhibit this as
\[
v_{1,j} \xrightarrow{f_{1,j}} v_{2,j} \xrightarrow{f_{2,j}} v_{1,j} \xrightarrow{g_{j}} v_{2,j+1} \xrightarrow{f_{2,j+1}} v_{1,j+1} \xrightarrow{f_{1,j+1}} v_{2,j+1} \xrightarrow{g_{j}} v_{1,j}
\]
Since the functor $\phi_{v_{1,j},f_{1,j}} = (-) \otimes^\mathbb{L}N_{Z_{j+1}/Y_j}$ and $\phi_{v_{1,j+1},f_{1,j+1}} = (-) \otimes^\mathbb{L} N_{Z_{j+1}/Y_{j+1}}$, and all other transition functors in this cycle are trivial, it follows that monodromy around this loop is trivial. 

Now we make the following observation. Let $K_n$ be the skeleton associated to a perverse schober with $n$ marked points. Then we recall that $K_n^\circ$ is the skeleton $K_n$ with each univalent vertex $v_p$ for $p \in \Sigma$ replaced by a trivalent vertex $v_p'$ with a loop. See Section \ref{sect:pschob} for details and notation. Counterclockwise monodromy around this loop is given by the spherical twist associated to $F_p:\mathscr{A}_p\rightarrow \mathscr{C}$. Now let $p_1,\dots, p_n$ be the set $\Sigma$, so that the edges $q_i$ are oriented counterclockwise around $c$ and $q_\infty$ is directly counterclockwise to $q_1$. Then define cycles $C_i$
\[
c \xrightarrow{q_i} v_{p_i'} \xrightarrow{e_{F_{p_i}}} v_{p_i'}\xrightarrow{q_i} c  
\] 
and a cycle $C_\infty$ by the concatenation $C_n \cdot \dots \cdot C_1$. Then it follows directly from the definition that monodromy around $C_\infty$ is simply 
\[
T_{F_n}\cdot \dots \cdot T_{F_1}.
\]
Let $k: Z_1\hookrightarrow Y_0$ be the closed embedding, assume that $\alpha_i : \mathscr{A}_i \rightarrow \DD^b_{\dg}(Y_0)$ form a semiorthogonal decomposition and let $F_i = \mathbb{L}k^* \cdot \alpha_i$. Then it follows from work of Addington and Aspinwall \cite[Theorem 11]{asp-add} that $T_{F_n}\cdot \dots \cdot T_{F_1}$ is the spherical twist associated to $\mathbb{L}k^*$. Furthermore, this spherical twist is, up to shift, just $(-)\otimes^\mathbb{L} N_{Y_0/Z_1}$. Now if we have a simplified type II degeneration of Calabi--Yau varieties, then monodromy around boundary components at either end of Figure \ref{fig:nY} is given by concatenation with $C_\infty$ and the cycle
\[
c \xrightarrow{q_\infty} v_{1,1} \xrightarrow{f_{1,1}} v_{1,2} \xrightarrow{f_{1,2}} v_{1,1} \xrightarrow{q_\infty} c
\]
Monodromy around this cycle is $(-)\otimes^\mathbb{L}N_{Z_1/Y_1}$. Therefore monodromy around the concatenation of these cycles is $\otimes^\mathbb{L} (N_{Z_1/Y_1} \otimes^\mathbb{L} N_{Z_1/Y_0}) = (-) \otimes^\mathbb{L} \mathscr{O}_{Z_1}$, which is trivial. Similarly we can deal with the monodromy around the final boundary component. The remaining case occurs when $n = 1$, that is, there are no ruled components of $Y$. In this case, a very similar argument suffices.
\end{proof}

\subsection{The case of K3 surfaces}

When we assume that the dimension of $\mathscr{X}$ is $3$, hence the dimension of $\mathscr{X}_0$ is 2, our results can be made more general. We have the following definition.
\begin{defn}
Let $\mathscr{X}$ be a K\"ahler manifold of dimension 3 equipped with a proper morphism $\pi:\mathscr{X} \rightarrow \Delta$. We say that $\mathscr{X}$ is a \emph{type II degeneration of K3 surfaces} if the following conditions hold.
\begin{enumerate}
\item The bundle $\omega_{\mathscr{X}}$ is trivial.
\item All fibers $\pi^{-1}(t)$ are smooth K3 surfaces if $t\neq 0$.
\item The fiber over $0$ is a union of surfaces $S_0 \cup \dots \cup S_n$ so that if $i = 0,n$ then $S_i$ is a smooth rational surface, and if $i \neq 0,n$ then $S_i$ is a smooth ruled surface over an elliptic curve. Furthermore, $E_i = S_i \cap S_{i+1}$ is a smooth elliptic curve and $S_i \cap S_j = \emptyset$ if $|i-j| \geq 1$.
\end{enumerate}
\end{defn} 
Note that we have changed notation slightly in order to emphasize that we are dealing with surfaces and curves instead of varieties of arbitrary dimension.

The fact that the total space is Calabi--Yau implies that $E_i\cup E_{i+1}$ is an anticanonical divisor in $S_i$ if $i \neq 0,1$ and $E_1$ and $E_n$ are anticanonical in $S_0$ and $S_n$ respectively. Generally \cite{perspink,kulikov}, there are three types of semistable degenerations of K3 surfaces, types I, II and III. Type I is essentially a degeneration to a smooth K3 surface and type III degenerations are normal crossings degenerations whose dual complex is a triangulation of $S^2$. 

There are well-known birational modifications of type II degenerations of K3 surfaces called type I modifications. See e.g. \cite{bmd} for details. 
\begin{defn}
Let $\pi:\mathscr{X} \rightarrow \Delta$ be a type II degeneration of K3 surfaces. Let $C$ be a smooth $(-1)$ curve in an irreducible component $S_i$ of $\pi^{-1}(0)$ which intersects $E_{i-1}$ or $E_i$ in a single point $p$. Then $C$ has normal bundle $\mathscr{O}(-1) \oplus \mathscr{O}(-1)$ in $\mathscr{X}$, thus it can be flopped in $\mathscr{X}$ to produce a type II degeneration $\pi':\mathscr{X}' \rightarrow \Delta$ of K3 surfaces. The fiber over $0$ of $\pi'$ is identical to that of $\pi$ except that $S_i'$ is the contraction of $S_i$ along $C$ and $S'_{i-1}$ or $S'_i$ is the blow up of $S_{i-1}$ or $S_i$ in the point $p$ depending on whether $C$ intersects $E_{i-1}$ or $E_i$. This is called a \emph{type I modification of the fiber over 0}.

\end{defn}
There are type II and III modifications of degenerations of K3 surfaces, but they do not play a role in type II degenerations of K3 surfaces. The following result is no doubt standard. We thank Alan Thompson for taking the time to explain its proof to us.
\begin{proposition}\label{prop:tImod}
Let $\mathscr{X} \rightarrow \Delta$ be a type II degeneration of K3 surfaces. By repeatedly performing type I modifications, we may construct a birational model of $\mathscr{X}$ for which the components $S_1,\dots, S_{n-1}$ are not simply ruled surfaces over an elliptic curve but in fact projective bundles over an elliptic curve.
\end{proposition}
\begin{proof}
Let $S_{n-1}$ be a component of the fiber over $0$, then $E_{n-1}$ and $E_{n}$ are sections of the ruling on $S_{n-1}$. If $S_{n-1}$ is not a $\mathbb{P}^1$ bundle over $E \cong E_1 \cong E_2$, then there is a fiber $R$ of the ruling on $S_{n-1}$ which is reducible. We will now show that the component $R'$ of $R$ which intersects $E_{n-1}$ is a $(-1)$ curve. 

The fibers of the ruling on $S_{n-1}$ must be of constant arithmetic genus $0$. The arithmetic genus is greater than or equal to the sum of the arithmetic genera of the irreducible components of each fiber, thus all irreducible components of each fiber must have arithmetic genus 0. Any irreducible singular curve has arithmetic genus greater than 0, so it follows that all components of curves in fibers of the ruling on $S_{n-1}$ are smooth rational curves. The generic fiber of this ruling is a smooth curve of genus 0 whose intersection number with $-K_{S_{n-1}}$ is $2$. Therefore, the genus formula for curves on surfaces says that the following situations occur for $C$ a smooth rational curve contained in the fiber of the ruling on $S_{n-1}$:
\begin{enumerate}
\item $C \cap (-K_{S_{n-1}}) = 0$ and $C^2 = -2$,
\item $C\cap (-K_{S_{n-1}}) = 1$ and $C^2 = -1$,
\item $C\cap (-K_{S_{n-1}}) = 2$ and $C^2 = 0$.
\end{enumerate}
In the first case, $C$ is disjoint from $-K_{S_{n-1}}$, in the second case $C$ intersects just one component of $-K_{S_{n-1}}$ and in the third case, $C$ is a general fiber of the ruling. Therefore, the component $R'$ chosen above is a $(-1)$ curve. Hence we may perform a type I modification along $R'$ which modifies $S_{n-1}$ by contracting $R'$ and blowing up the point $R' \cap E_{n-1}$ in $S_{n-2}$. 

Repeating this argument, we can obtain a birational model of $\mathscr{X}$ for which $S_{n-1}$ has no reducible fibers, hence is a projective bundle over $E_{n-1}$. The result then follows by iterating this argument for each $S_i$.
\end{proof}
Now let's make a numerical observation. Every rational surface $S$ has minimal model either $\mathbb{P}^2$ or $\mathbb{F}_n$ for some $n$. Both $\mathbb{P}^2$ and $\mathbb{F}_n$ admit full exceptional collections, therefore, by Orlov's formula \cite{orl-monoid}, $S$ admits a full exceptional collection. Let $k_0$ and $k_n$ be the number of exceptional objects in this exceptional collection on $\DD^b(S_0)$ and $\DD^b(S_n)$ respectively. 
\begin{proposition}\label{prop:24}
Let $S_0 \cup \dots \cup S_n$ be the special fiber of a type II degeneration of K3 surfaces, and assume that if $i \neq 0,n$, $S_i$ are $\mathbb{P}^1$ bundles over an elliptic curve. Then $k_0 + k_n = 24$.
\end{proposition}
\begin{proof}
The topological Euler characteristic of $S_0 \cup \dots \cup S_n$ is 24. Alan Thompson informs us that this can be proved for the special fiber of any semistable degeneration of K3 surfaces using the Clemens-Schmid exact sequence \cite{mor}, but we opt for a quicker proof here which is specific to the case of type II degenerations. We have that there is a distinguished triangle of perverse sheaves on $\pi^{-1}(0)$,
\[
\prescript{p}{}\psi_\pi\mathbb{C} \longrightarrow \prescript{p}{}\phi_\pi\mathbb{C} \longrightarrow \mathbb{C}
\]
see \cite[pp. 261]{ps}. Here $\prescript{p}{}\psi_\pi$ denotes the perverse nearby cycles functor and $\prescript{p}{}\phi_\pi$ the perverse vanishing cycles functor, which are just the nearby and vanishing cycles functors composed with $[1]$. The hypercohomology of $\prescript{p}{}\psi_\pi$ is isomorphic to the cohomology of $\pi^{-1}(t)$ for $t \neq 0$, hence it has Euler characteristic 24, since $\pi^{-1}(t)$ is a K3 surface. It follows from \cite[Theorem 10]{rud} that the Euler characteristic of the sheaf $\prescript{p}{}\phi_\pi\mathbb{C}$ is the Euler characteristic of the critical locus of $\pi$ up to sign, which is a disjoint union of $n$ elliptic curves. Therefore, the Euler characteristic of $\prescript{p}{}\phi\mathbb{C}$ is 0. It follows then that the topological Euler characteristic of $\pi^{-1}(0)$ is equal to that of $\pi^{-1}(t)$, which we know is 24.

Now applying the Mayer-Vietoris exact sequence, along with the fact that the Euler characteristic of an elliptic curve is 0, we find that
\[
\sum_{i=1}^n \chi_\mathrm{top}(S_i) = 24. 
\]
We know that a $\mathbb{P}^1$ bundle over an elliptic curve has Betti numbers $b_1 = b_2 = b_3 = 2$ and $b_0 = b_4 = 1$, hence its topological Euler characteristic is 0. Thus $\chi_\mathrm{top}(S_0) + \chi_\mathrm{top}(S_n) = 24$. Since $h^{1}(S_i) = h^{3}(S_i) = h^{0,2}(S_i) = 0$ for $i = 0$ or $n$, the topological Euler characteristics of $S_0$ and $S_n$ are equal to the rank of $\mathrm{HH}_0(\DD^b(S_0))$ and $\mathrm{HH}_0(\mathrm{D}^b(S_n))$ respectively by the Hochschild-Kostant-Rosenberg isomorphism, which says that for a smooth projective variety $X$
\[
\mathrm{HH}_m(\DD^b(X)) = \bigoplus_{q-p=m}\HH^q(X,\Omega_X^p).
\]
Kuznetsov \cite[Corollary 7.5]{kuz} shows that Hochschild homology is additive with respect to semiorthogonal decompositions. Therefore, $\chi_\mathrm{top}(S_0) = \rank \mathrm{HH}_0(\DD^b(S_0)) = k_0$ and $\chi_\mathrm{top}(S_n) = \rank \mathrm{HH}_0(\DD^b(S_n)) = k_n$. Thus $k_0 + k_n = 24$.
\end{proof}
Theorem \ref{thm:simpII} and Proposition \ref{prop:24} have the following corollary.
\begin{corollary}\label{cor:puttogether}
Let $\mathscr{X}$ be a type II degeneration of K3 surfaces $\mathscr{X}$ with $n+1$ irreducible components. There is birational model $\mathscr{X}'$ of $\mathscr{X}$ of with central fiber $\mathscr{X}_0'$ and $K$-coordinatized perverse sheaf of categories $\mathscr{S}_K$ over $S^2$ with $n$ boundary components and with respect to a stratification given by a set of points $\Sigma$ so that $K, \Sigma$ and $\mathscr{S}_K$ have the following properties.
\begin{enumerate}
\item All edges of $K$ have both ends in $\mathrm{Vert}(K)$.
\item The category of global sections is $\Perf(\mathscr{X}_0')$. 
\item The monodromy around all boundary components of $\mathscr{S}_K$ is trivial.
\item $|\Sigma| = 24$ and for each $p \in \Sigma$, the category $\mathscr{A}_p = \Perf_k$.
\end{enumerate}
\end{corollary}
Our goal now is to show that $\Perf(\mathscr{X}_0)$ is equivalent to $\Perf(\mathscr{X}_0')$ if $\mathscr{X}_0$ and $\mathscr{X}_0'$ denote the special fibers of type II degenerations of K3 surfaces related by a type I modification. Assume that $X = S_1 \cup S_2$ is normal crossings of dimension 2 and $Z = S_1 \cap S_2$ is smooth. Assume that there are smooth (possibly empty) divisors $Z_1$ and $Z_2$ in $S_1$ and $S_2$ which are disjoint from $Z$ and so that $Z \cup Z_1$ is anticanonical in $S_1$ and $Z \cup Z_2$ is anticanonical in $S_2$. Assume that $p$ is a point contained in $Z$. Then we can blow up either $S_1$ or $S_2$ in $p$ to get varieties $\widetilde{S}_1$ and $\widetilde{S}_2$. This produces  two different normal crossings varieties, $X_1 = \widetilde{S}_1 \cup S_2$ and $X_2 = S_1 \cup \widetilde{S}_2$, which look like part of a type II degeneration of K3 surfaces, which are related to one another by type I modification. We would like to show that $X_1$ and $X_2$ have equivalent derived categories of perfect complexes.

We will use the following notation. Let $S$ be a surface, and let $Z_1$ and $Z_2$ be disjoint smooth divisors in $S$ so that $Z_1 \cup Z_2$ is the vanishing locus of a global section of $-K_S$, and let $\mathrm{pt}$ be a point in $Z_1$, with $k_1:Z_1 \hookrightarrow S,k_2:Z_2 \hookrightarrow S$ and $i:\mathrm{pt}\hookrightarrow Z_1$ be the embedding maps. Let $\widetilde{S}$ be the blow up of $S$ in $\mathrm{pt}$ and $\widetilde{k}_1:Z_1\hookrightarrow \widetilde{S},\widetilde{k}_2:Z_2\hookrightarrow \widetilde{S}$ the embedding maps of the proper transform of $Z_1$ and $Z_2$ into $\widetilde{S}$.

\begin{proposition}\label{prop:ifpoint}
The category 
\[
\DD^b_{\dg}(\mathrm{pt}) \times_{\mathbb{R}i_*,\mathbb{L}\ell_1^*} \DD^b_{\dg}(S)
\] 
is quasiequivalent to $\DD^b_{\dg}(\widetilde{S})$. The functor $c_{\mathbb{R}i_*,\mathbb{L}\ell_1^*}$ is quasiisomorphic as a dg bimodule to $\mathbb{L}\widetilde{k}_1^*$, and $(a,b,\mu)\mapsto \mathbb{L}\ell_2^*(b)$ is quasiisomorphic to $\mathbb{L}\widetilde{\ell}_2^*$ as a dg bimodule.
\end{proposition}
\begin{proof}
We have that there is a semiorthogonal decomposition of a blow up of a surface $S$ at a point given by
\[
\langle \mathbf{R}q_*(\mathbf{L}p^*\DD^b(\mathrm{pt}) \otimes \mathscr{O}_{\mathbb{P}^1}(-1)), \mathbf{L}\pi^*\DD^b( S)\rangle.
\]
as described in \cite{orl-monoid}. Following the proof of Proposition \ref{prop:todiv} we obtain that there is a quasiequivalence between $\DD^b_{\dg}(\widetilde{S})$ and 
\[
\DD^b_{\dg}(\mathrm{pt}) \times_{\mathbb{R}i_*((-) \otimes^\mathbb{L} \mathbb{L}j^*\mathscr{O}(-1)), \mathbb{L}k_1^*}\DD^b_{\dg}(S).
\]
where $j : \mathrm{pt} \rightarrow \mathbb{P}^1$ is the embedding of the point $\mathrm{pt}$ into the exceptional $\mathbb{P}^1$. Therefore, since $\mathbb{L}j^*\mathscr{O}_{\mathbb{P}^1}(-1)$ is trivial, the first statement follows. Following \ref{prop:twosections}, we obtain the statements about the functors.
\end{proof}
%

\begin{theorem}\label{prop:flip}
With notation as above, $\Perf(X_1)$ and $\Perf(X_2)$ are quasiequivalent.
\end{theorem}
\begin{proof}
Before we proceed, note that the category $A_2(\mathscr{C})$ always admits a quasiisomorphism to itself, given by
\[
\sigma_2 : (a,b,\mu) \mapsto (b, \cone(\mu),\xi)
\]
where $\xi: b \rightarrow \cone(\mu)$ is the natural closed map of degree 0. Furthermore, the functors $f_i \cdot \sigma_2$ and $f_{i+1}$ (where $i$ is taken modulo 3) are equivalent.

Using the fact that $\DD^b(\widetilde{V}_1)$ is quasiequivalent to $\DD^b_{\dg}(S_1) \times_{\mathbb{L}k_1^*,\mathbb{R}i_*} \DD^b_{\dg}(Z)$, we see that $\DD^b_{\dg}(\widetilde{S}_1)$ is quasiequivalent to the homotopy limit over the diagram,
\[
\begin{CD}
\DD^b_{\dg}(S_1) @>\mathbb{L}k_1^*>> \Tw \DD^b_{\dg}(Z) @<f_1<< A_2(\DD^b_{\dg}(Z)) @>f_2>> \Tw \DD^b_{\dg}(Z) @<\mathbb{R}i_*<< \DD^b_{\dg}(\mathrm{pt})
\end{CD}
\]
There's a functor from this homotopy limit to $A_2(\DD^b_{\dg}(Z))$ by the universal property of homotopy limits, and composing this functor with $f_3$ we get a functor to $\DD^b_{\dg}(Z)$, which is equivalent to $\mathbb{L}\widetilde{k}_1^*$ where $\widetilde{k}_1$ is the embedding of $Z$ into $\widetilde{S}_1$. Therefore, $\Perf(\widetilde{X}_1)$ is quasiequivalent to the homotopy limit over the diagram
\begin{equation}\label{biglimit}
\begin{CD}
\DD^b_{\dg}(S_1) @>\mathbb{L}k_1^*>> \DD^b_{\dg}(Z) @<f_1<< A_2(\DD^b_{\dg}(Z)) @>f_2>> \DD^b_{\dg}(Z) @<\mathbb{R}i_*<< \DD^b_{\dg}(\mathrm{pt}) \\
@. @. @Vf_3VV @. @ . \\
@. @.  \DD^b_{\dg}(Z) @. @. \\
@. @.  @A \mathbb{L}k_2^* AA @. @. \\
@. @.  \DD^b_{\dg}(V_2) @. @. 
\end{CD}
\end{equation}
However, applying Proposition \ref{prop:ifpoint} along with the symmetry of $A_2(\DD^b_{\dg}(Z))$ mentioned in the beginning of the proof, the homotopy limit over the diagram
\[
\begin{CD}
@. @. A_2(\DD^b_{\dg}(Z)) @>f_2>> \DD^b_{\dg}(Z) @<\mathbb{R}i_*<< \DD^b_{\dg}(\mathrm{pt}) \\
@. @. @Vf_3VV @. @ . \\
@. @.  \DD^b_{\dg}(Z) @. @. \\
@. @.  @A \mathbb{L}k_2^* AA @. @. \\
@. @.  \DD^b_{\dg}(S_2) @. @. 
\end{CD}
\]
is equivalent to $\DD^b_{\dg}(\widetilde{S}_2)$ and the functor from the limit to $\DD^b_{\dg}(Z)$ induced by $f_1$ is equivalent to $\mathbb{L}\widetilde{k}_2^*$. Therefore, the homotopy limit over the diagram in Equation \ref{biglimit} is also quasiequivalent to $\Perf(X_2)$.
\end{proof}
The following theorem uses the formalism of perverse sheaves of categories to show that $\Perf(\mathscr{X}_0)$ is not changed by a type I modifications. The same result can likely be extracted from the work of Bridgeland \cite{brid}. We expect that this holds in higher dimensions as well.
\begin{theorem} \label{thm:typeImod}
If $\mathscr{X}$ and $\mathscr{X}'$ are type II degenerations of K3 surfaces related by type I modifications, then $\Perf(\mathscr{X}_0)$ and $\Perf(\mathscr{X}_0')$ are quasiequivalent.
\end{theorem}

\begin{proof}
Apply Proposition \ref{prop:flip}.
\end{proof}
It follows then, by Proposition \ref{prop:tImod} and Corollary \ref{cor:puttogether} that if $\mathscr{X}$ is a type II degeneration of K3 surfaces, there is a perverse sheaf of categories whose category of global sections is $\Perf(\mathscr{X}_0)$.
\subsection{Connection with homological mirror symmetry}
The reader should interpret this in the following way. The global sections of the perverse sheaf of categories that we are studying should be equivalent to the Fukaya category of some symplectic fibration $h: U \rightarrow S^2 \setminus \{D_0,\dots, D_n\}$ whose fibers are mirror to an elliptic curve $E$ (for instance symplectic 2-tori). Therefore, the twisted, derived Fukaya category of a smooth fiber $h^{-1}(t)$, denoted $\DD^\pi\mathscr{F}(h^{-1}(t))$, should be quasiequivalent to $\DD^b_{\dg}(E)$ for some elliptic curve $E$. The singular fibers of $h$ should be nodal 2-tori whose vanishing cycles correspond to spherical objects in $\DD^b_{\dg}(E)$ under mirror symmetry. Proposition \ref{prop:tImod} says that this fibration should have trivial symplectic monodromy around the boundary components. This should be interpreted as saying that we can  compactify this fibration by adding a smooth symplectic torus at each boundary component. Proposition \ref{prop:24} says that this fibration should have 24 degenerate fibers. Therefore, if it admits the structure of a complex fibration, then it must be in fact a K3 surface. This dovetails nicely with work done by the first author along with C. Doran and A. Thompson \cite{dht}, where they conjecture that there is a bijection between certain fibration structures on a Calabi--Yau variety $X$ and Tyurin degenerations on the mirror Calabi--Yau variety $X^\vee$. In \cite{dt}, Doran and Thompson extend this to type II degenerations of K3 surfaces. In homological mirror symmetry, this conjecture takes the following form.
\begin{conj}\label{conj:dht}
Let $X$ be a $d$-dimensional Calabi--Yau variety and let $X^\vee$ be its mirror. Assume there is a fibration $f:X\rightarrow \mathbb{P}^1$ on $X$ by Calabi--Yau $(d-1)$-folds. Let $Z_1,\dots, Z_n$ be smooth fibers of $\pi$. Then there is a type II degeneration of $X^\vee$ with special fiber $Y$ so that
\[
\DD^\pi \mathscr{F}(X\setminus \{Z_1,\dots, Z_n\}) \cong \Perf(Y).
\]
According to the general picture presented by Seidel in \cite{seidicm}, the Fukaya category of $X \setminus \{Z_1,\dots,Z_n\}$ should deform to the Fukaya category of $X$. We conjecture that this deformation agrees with the deformation of $Y$ to $X^\vee$ given by Friedman \cite{fried}.
\end{conj}

\bibliographystyle{plain}
\bibliography{references}

\end{document}